\documentclass[11pt]{amsart}
 \usepackage{mathabx}
\usepackage{amssymb}
\usepackage{amsmath}
\usepackage{amsfonts}
\usepackage{pb-diagram}
\usepackage{centernot}
\usepackage{mathtools}
\usepackage{xurl}
\usepackage{hyperref}
\usepackage{stmaryrd}
 \usepackage{relsize}

\usepackage{cleveref}
\usepackage{mathrsfs}
\usepackage{soul}
\usepackage{enumerate}
\usepackage{import}

\usepackage{tikz}
\usetikzlibrary{fadings}
\usetikzlibrary{patterns}
\usetikzlibrary{shadows.blur}
\usetikzlibrary{shapes}

\usepackage{pgfplots}
\usepackage{float}
\usetikzlibrary{positioning,chains,fit,shapes,calc}
\usepgfplotslibrary{external}

\numberwithin{equation}{section}

\usepackage[textsize=tiny]{todonotes}

\newtheorem{theorem}{Theorem}[section]
\newtheorem*{theorem*}{Theorem}

\newtheorem{lemma}[theorem]{Lemma}
\newtheorem{proposition}[theorem]{Proposition}
\newtheorem*{proposition*}{Proposition}
\newtheorem{corollary}[theorem]{Corollary}
\newtheorem*{corollary*}{Corollary}
\newtheorem*{fact*}{Fact}

\newtheorem{prop}[theorem]{Proposition}

\newtheorem{claim}[theorem]{Claim}
\newtheorem*{claim*}{Claim}
\newtheorem{fact}[theorem]{Fact}
\newtheorem{cor}[theorem]{Corollary}

\theoremstyle{definition}
\newtheorem*{definition*}{Definition}
\newtheorem{definition}[theorem]{Definition}

\newtheorem{question}[theorem]{Question}
\newtheorem*{question*}{Question}
\newtheorem{example}[theorem]{Example}

\theoremstyle{remark}

\newtheorem{remark}[theorem]{Remark}

\newcommand{\heq}{\operatorname{heq}}

\newcommand{\Aut}{\mathsf{Aut}}
\newcommand{\Av}{\mathsf{Av}}

\newcommand{\RCF}{\mathsf{RCF}}
\newcommand{\Th}{\mathsf{Th}}

\newcommand{\tp}{\ensuremath{\mathsf{tp}}}

\newcommand{\id}{\ensuremath{\mathsf{id}}}

\newcommand{\Cb}{\ensuremath{\mathsf{Cb}}}

\newcommand{\acl}{\ensuremath{\mathsf{acl}}}
\newcommand{\dcl}{\ensuremath{\mathsf{dcl}}}

\newcommand{\cU}{\ensuremath{\mathbb{M}}}
\newcommand{\cL}{\ensuremath{\mathcal{L}}}

\newcommand{\abar}{\bar{a}}

\newcommand{\Ind}[1]
{#1\setbox0=\hbox{$#1x$}\kern\wd0\hbox to 0pt{\hss$#1\mid$\hss} \lower.9\ht0\hbox to 0pt{\hss$#1\smile$\hss}\kern\wd0}

\newcommand{\indep}{\mathop{\mathpalette\Ind{}}}

\newcommand{\notind}[1]
{#1\setbox0=\hbox{$#1x$}\kern\wd0
\hbox to 0pt{\mathchardef\nn=12854\hss$#1\nn$\kern1.4\wd0\hss}
\hbox to 0pt{\hss$#1\mid$\hss}\lower.9\ht0 \hbox to 0pt{\hss$#1\smile$\hss}\kern\wd0}

\newcommand{\Sh}{\operatorname{Sh}}
\newcommand{\eq}{\operatorname{eq}}
\newcommand{\St}{\operatorname{St}}
\newcommand{\Stab}{\operatorname{Stab}}
\newcommand{\RCVF}{\operatorname{RCVF}}

\newcommand{\NIP}{\ensuremath{\mathsf{NIP}}}

\newcommand{\ACVF}{\operatorname{ACVF}}

\newcommand{\quot}{\operatorname{quot}}

\title{Externally definable fsg groups in NIP theories}

\author{Artem Chernikov}
\address{Department of Mathematics, 1101 Kirwan Hall, University of Maryland College Park, MD 20742-4015, USA}
\email{artem@umd.edu}

\date{\today}

\begin{document}

\maketitle 

\begin{abstract}
We show that every fsg group  externally definable in an NIP structure $M$ is definably isomorphic to a group interpretable in $M$. Our proof relies on  honest definitions and a group chunk result reconstructing a hyper-definable group from its multiplication given generically  with respect to a translation invariant  definable Keisler measure on it. 
We obtain related results on externally (type-)definable sets and groups, including a proof of a conjecture of Eleftheriou on fsg groups in real closed valued fields, and a description of externally definable, definably amenable subgroups of definable groups.
\end{abstract}

\tableofcontents

\section{Introduction}

The study of stable groups is at the core of geometric stability
theory and its applications, starting with the work of Poizat on the
theory of generics and the work on group and field reconstruction
theorems by Zilber, Hrushovski and others (see e.g.~\cite{poizat2001stable, pillay1996geometric, wagner1997stable}). Analogously, groups
definable in $o$-minimal structures were studied extensively and
close connections to the theory of real Lie groups were discovered, culminating  in a resolution of Pillay's conjecture
for compact $o$-minimal groups. Its proof \cite{NIP1} and subsequent work \cite{NIP2, NIP3} has brought to light
the importance of a general theory of groups definable in $\NIP$
structures \cite{simon2015guide}  and the study of invariant measures on definable subsets of the group (see e.g.~\cite{chernikov2018model} for a short survey). A deep analysis of definable groups in algebraically closed valued fields was carried out in \cite{hrushovski2019valued}. Methods  of
topological dynamics were introduced into the picture starting with Newelski \cite{N1}, and connections with \emph{tame} dynamical systems as studied by Glasner, Megrelishvili and others (see e.g.~\cite{glasner2007structure, Gla18}) have emerged, starting with \cite{chernikov2018definably, ibarlucia2016dynamical}. This led to the development of the theory of generics in definably amenable NIP
groups \cite{chernikov2018definably, stonestrom2023f} and a proof of Newelski's Ellis group conjecture connecting the model theoretic connected component $G/G^{00}$ with the ideal group of the associated dynamical system \cite{chernikov2018definably} for definably amenable groups. Further recent work on NIP groups involves: further work around (generic, local) compact domination \cite{simon2017vc, conant2021pseudofinite};  study of the structure of convolution semigroups of Keisler measures \cite{chernikov2022definable, chernikov2023definable2, chernikov2024definable}; a proof of the revised Newelski's conjecture beyond the definably amenable case (over countable models)  \cite{chernikov2024definable}; finiteness of the archimedean rank of local connected component of an NIP group \cite{hrushovski2024approximate}.

A particularly nice class of definably amenable groups arising in the work on Pillay's $o$-minimal groups conjecture is that of groups with \emph{finitely satisfiable generics}, or \emph{fsg} groups, introduced in \cite{NIP1}. It contains stable groups, groups with a stably dominated generic type in algebraically closed valued fields \cite{hrushovski2019valued}, and corresponds precisely to (definably) compact groups in $o$-minimal structures \cite{NIP1} and in the $p$-adics \cite{johnson2023note}. In the NIP context (see Section \ref{sec: type-def fsg groups}), these are precisely the groups that admit a (unique) translation-invariant measure $\mu$ on their definable subsets which is moreover \emph{generically stable} --- an abstract analogue of Haar measure in the definable category. The theory of stable groups largely generalizes to fsg groups in NIP theories \cite{NIP1, NIP2, NIP3, zbMATH06111354, chernikov2018definably} and partially beyond NIP \cite{chernikov2024definable}.

A set  $X \subseteq M^n$ is \emph{externally definable} if $X = Y \cap M^n$ for $Y \subseteq N^n$ a definable (with parameters) set in some elementary extension $N \succeq M$. A characteristic property of stable theories is that for every model, every externally definable set is already definable (this can hold for a particular model $M$ of an unstable theory, e.g.~for the fields $\mathbb{R}$ or $\mathbb{Q}_p$). An important theorem of Shelah \cite{shelah2014strongly} (generalizing Baisalov and Poizat \cite{baisalov1998paires} in the $o$-minimal case) shows that the expansion of an NIP structure $M$ by all externally definable sets $M^{\Sh}$ eliminates quantifiers and remains NIP. This suggests that externally definable sets in NIP structures are well behaved and play a key role in the study of NIP theories, further explained by the existence of \emph{honest definitions} introduced in \cite{chernikov2013externally, chernikov2015externally} (see Section \ref{sec: Sh exp, honest def}). 
Here we are interested in externally definable groups, i.e.~groups with the underlying set and the graph of their operations externally definable, mostly in the NIP context.  Some issues around external definability in groups were considered in e.g.~
\cite{NIP2, zbMATH06578125, wreo1695, zbMATH07479576, arXiv:2504.05566, baro2021open, baro2024ellis, arXiv:2208.00954}. 

Especially relevant for us here is \cite{chernikov2013externally} --- where it was demonstrated that many properties and invariants of a definable group such as definable amenability, fsg or $G^{00}$, are preserved  when passing to Shelah's expansion (see Section \ref{sec: props of grps preserved in Sh exp}). In this paper we begin studying externally definable  groups in NIP structures, aiming to describe them in terms of the internally (type-, or even hyper-)definable ones, and first focus on the fsg case. We summarize the main results of the paper about fsg groups.

\begin{theorem*}[Theorem \ref{thm: main for fsg subgroups}]
	Assume $T$ is NIP, $M \models T$ and $G$ is an $M$-definable group. Assume that $H$ is an externally definable subgroup of $G(M)$ which is fsg (in $\Th(M^{\Sh})$). Then $H$ is already an  $M$-definable subgroup of $G$ in $T$. (A similar result holds for type-definable groups, see Theorem \ref{thm: main for fsg subgroups} for the exact statement.)
\end{theorem*}

\begin{theorem*}[Theorem \ref{thm: type-def and ext def fsg implies def}]
Let $T$ be  NIP, $G$ type-definable over $M_0 \models T$, and $M \succ M_0$ is $|M_0|^+$-saturated. Assume that $G$ is fsg in $T$, and also that $G(M)$ is definable in $M^{\Sh}$. Then $G(M)$ is already definable in $M$.
\end{theorem*}

\begin{theorem*}[Theorem \ref{thm: ext def fsg groups is to definable}]
		Let $T$ be NIP and $M \models T$. Then a group $G$ definable in $M^{\Sh}$ is fsg (in the sense of $T' := \Th(M^{\Sh})$) if and only if it is definably (in $M^{\Sh}$) isomorphic to a group definable in $M^{\eq}$ and fsg in $T$. And if $G$ was only type-definable in $T'$, it is type-definably isomorphic to a group hyper-definable in $T$ over some larger model (see Remark \ref{rem: type-def fsg in MSh recognize as hyper-definable}).
\end{theorem*}

These results also hold for an arbitrary reduct of $M^{\Sh}$ expanding $M$ (Remarks \ref{rem: no fsg subgroups applies to reducts} and \ref{rem: ext def fsg iso def applies to reducts}). Our proof of Theorem \ref{thm: ext def fsg groups is to definable}  relies on honest definitions and the following measure-theoretic group chunk theorem generalizing \cite[Section 3.4]{hrushovski2019valued} from partial definable types to partial \emph{type-definable} types. See also \cite{TaoBlog} for a probabilistic interpretation of the classical group chunk and group configuration theorems in algebraically closed fields. We note that in some related special cases, e.g.~in the study of definable groups internal to the predicate in dense pairs of $o$-minimal or related ``geometric'' structures, one can use a dimension-theoretic or a topological group chunk \cite{van1990weil, eleftheriou2021pillay, peterzil2022definable}. But in a general fsg NIP group we have neither dimension nor topology at our disposal, so we work directly with the canonical generically stable translation-invariant measure.

\begin{theorem*}[Corollary \ref{cor: measure theoretic group chunk}]
Let $T$ be NIP, $G$ an $M$-type-definable group, and $\mu \in \mathfrak{M}_{G}(\cU)$ an $M$-definable left-invariant measure (such $M$ and $\mu$ exist  whenever $G$ is a definably amenable NIP group, Fact \ref{fac: def am def measure}). Then $G$ can be recovered, up to an $M$-type-definable isomorphism of $M$-hyperdefinable groups in $\cU^{\heq}$, from the $M$-type-definable partial type $\pi_{\mu} = \{\varphi(x) \in L(\cU) : \mu(\varphi(x)) = 1 \}$ and the generically given group operation on $\pi_{\mu}$.
\end{theorem*}

For a definable measure $\mu$, the partial type $\pi_{\mu}$ need not be definable, but only \emph{type-definable}, hence the equality of germs on $\pi_{\mu}$ is only a type-definable equivalence relation. Our proof of Theorem \ref{thm: ext def fsg groups is to definable} then proceeds as follows. Let $L' := L^{\Sh}$, $T' := \Th_{L'}(M^{\Sh})$ and $\widetilde{M}' \succ^{L'} M^{\Sh}$ be a saturated model (then $M' := \widetilde{M}'\restriction_{L} \succ^{L} M$ is saturated) and  $G'$ an $L'(M)$-definable fsg group, witnessed by a global $G'(M')$-invariant measure $\mu' \in \mathfrak{M}_{G'}^{L'}(M')$ generically stable over $M$ (see Section \ref{sec: type-def fsg groups}). Then $\mu := \mu' \restriction_{L}$ remains generically stable in $T$, and $\mu'$ is the unique measure extending $\mu$ via:
\begin{theorem*}[Theorem \ref{thm: corresp for gs measures}]
	Let $T$ be NIP, $M \models T$, $T' := \Th_{L'}\left(M^{\Sh} \right)$ and $\widetilde{M}'$ a monster model for $T'$. Then the map $\mu' \in \mathfrak{M}_x^{L'}(M') \mapsto \mu := \mu'\restriction_{L}  \in \mathfrak{M}_x^{L}(M')$ defines a bijection between global generically stable measures in $T'$ and in $T$; and 
	$\mu'$ is the unique measure in $\mathfrak{M}_x^{L'}(M')$ extending $\mu$.
	\end{theorem*}

\noindent Using honest definitions we can then  generically  recover the group operation of $G'$ on $\pi_{\mu}$ definably in $M' \models T$ (working over a larger model $M_{\omega} \succ^{L'} M$). Applying Corollary \ref{cor: measure theoretic group chunk}, we find a group $G^*$ hyper-definable in $T$ over $M_{\omega}$ and $M_{\omega}$-type-definably (in $T'$) isomorphic to $G'$. By a compactness argument (Section \ref{sec: eliminable hyper-def group}) we can then replace $G^*$ by a group $G_0$ definable in $T^{\eq}$. In the case of Theorem \ref{thm: main for fsg subgroups} we are already given a group $G$ definable in $M$ from the start, so we can recognize $H$ directly in $T$ as the stabilizer $\Stab_{G}(\mu)$, where again $\mu = \mu'\restriction_{L}$ for $\mu' \in \mathfrak{M}_{H}^{L'}(M')$ a translation-invariant generically stable measure witnessing that $H$ was fsg in $T'$ (see Section \ref{sec: Externally definable fsg subgroups of definable groups are definable}).

As an application of Theorem \ref{thm: ext def fsg groups is to definable}, we prove a conjecture of Pantelis Eleftheriou:
\begin{corollary*}[Corollaries \ref{cor: fsg subgroups in RCVF} and \ref{cor: ext def fsg groups in RCVF}]
Let $\RCVF$ be the theory of real closed valued fields, $M' \models \RCVF$ arbitrary. 
\begin{enumerate}
	\item If $G$ is definable in the reduct $M$ of $M'$ to $\RCF$, and $H \leq G$ is definable in $M'$ and is fsg, then $H$ is already definable in $M$.
	\item If $G$ is definable in $M'$ and is fsg, then $G$ is definably (in $M'$) isomorphic to a group definable in $M$.
\end{enumerate}
\end{corollary*}

Going beyond fsg, we consider externally type-definable, definably amenable subgroups of definable groups. Working in a monster model of $M^{\Sh}$, we can describe them as unions of  directed systems of subgroups (uniformly) type-definable in $T$, and use it to get a description entirely in $M$:
\begin{theorem*}[Proposition \ref{prop: ext def am descr in M'} and Theorem \ref{thm: ext def abelian final approx in M}]
	Let $T$ be NIP,  $M \models T$, $\widetilde{M}' \succ M^{\Sh}$ a monster model for $T'$, $G$ a definable group in $M$ and $H(M') \leq G(M')$ a subgroup which is $L'(M)$-type-definable and definably amenable in $T' = \Th(M^{\Sh})$. 
	 Then there exist an $L(M)$-type-definable set $\St_G(x,\bar{y})$ and $L'(M)$-type-definable set $Y \subseteq (M')^{\bar{y}}$, with $|\bar{y}| = |T|$, so that:
	\begin{enumerate}
		\item for every $\bar{a} \in Y(M')$, $\St_G(M'; \bar{a})$  is a type-definable in $M' \models T$ subgroup of $H(M')$ of the form $\Stab_{G}(\mu_{\bar{a}}) \leq H(M')$ for a measure $\mu_{\bar{a}} \in \mathfrak{M}^{L}_{G}(M')$ generically stable over $\bar{a}$ in $T$;
		\item $H(M') = \bigcup_{\bar{a} \in Y} \St_{G}(M', \bar{a})$;
		\item the family $\{\St_{G}(M', \bar{a}) : \bar{a} \in Y\}$ is directed.
	\end{enumerate}

	Working entirely in $M$ and assuming that $H(M) \leq G(M)$ is $L'(M)$-definable and definably amenable  in $T'$, there is an $M$-type-definable in $T$ set $\Gamma(x, \bar{y})$ and a strict $L'(M)$-pro-definable set $Z \subseteq \cU^{\bar{y}}$ so that $H(M) = \bigcup_{\bar{a} \in Z} \Gamma(M, \bar{a})$ and the family $\{\Gamma(M, \bar{a}) : \bar{a} \in Z\}$ is (uniformly) directed.
\end{theorem*}
Here the sets $\Gamma(M,\bar{a})$ are given by the intersections of approximate stabilizers of a (possibly infinite) sequence of measures supported on finite subsets of $H(M)$ (we cannot expect $\Gamma(M, \bar{a})$ to be subgroups of $H(M)$, see Example \ref{ex: ext def ab wo type-def subgrps}). Our proof of this theorem proceeds as follows. First we generalize a result of 
Hils, Hrushovski and Simon \cite{hils2021definable} from definable abelian to type-definable, definably amenable NIP groups:
\begin{theorem*}[Theorem \ref{thm: approx def am by stab gs}]
	($T$ NIP) If $G$ is type-definable over $M$ and  definably amenable, then there are type-definable over $M$ sets $\St_G \subseteq G(\cU) \times \cU^{\bar{y}}, X_{G} \subseteq \cU^{\bar{y}}$ with $|\bar{y}| = |T|$ so that $G(\cU) = \bigcup_{\bar{a} \in X_{G}} \St_{G}(\cU, \bar{a})$, $\St_G(\cU; \bar{a})$ is the stabilizer of a global generically stable measure $\mu_{\bar{a}}$ for every $\bar{a} \in X_{G}$, and the family $\{\St_{G}(\cU, \bar{a}) : \bar{a} \in X_{G}\}$ is directed.\end{theorem*}
\noindent We stress that these type-definable stabilizers of generically stable measures are \emph{not fsg} in general, i.e.~the measures $\mu_{\bar{a}}$ need not be supported on their stabilizers $\St_{G}(\cU, \bar{a})$ (see Example \ref{ex: no fsg subroups in R+}). In order to carry out some compactness arguments in its proof and later, we observe  that the set of global generically stable measures in an NIP theory can be identified with a  hyper-definable set  of  length $|T|$ tuples (similar to how the set of global generically stable types can be viewed as a pro-definable set in $\cU^{\eq}$, an important point in  Hrushovski-Loeser \cite{hrushovski2016non} and subsequent work):   

\begin{proposition*}[Section \ref{sec: Generically stable measures as a hyper-definable set}]
	($T$ NIP) We can identify the space of global generically stable measures supported on a type-definable set $Y$ with a hyper-definable set $\widetilde{\mathcal{M}}_Y$ (of length $|T|$ tuples), so that the basic operations (e.g.~$\otimes$ and definable push-forwards) and invariants (e.g.~stabilizers when $Y$ is a type-definable group) are uniformly type-definable under this identification.
\end{proposition*}
\noindent Then, given an externally (type-)definable group $H$ we apply Theorem \ref{thm: approx def am by stab gs} to write it as a union of a directed system of type-definable in $T'$ stabilizers of generically stable measures in $T'$. But by earlier discussion for fsg groups (Theorem \ref{thm: corresp for gs measures} and (the proof of)  Theorem \ref{thm: main for fsg subgroups}) such stabilizers are already type-definable in $T$. Some additional compactness and uniformization arguments allow to conclude.

We give a brief overview of the paper. In Section \ref{sec: Preliminaries} we review the basic properties of Shelah's expansion and honest definitions (Section \ref{sec: Sh exp, honest def}); Keisler measures in (type-)definable groups  (Sections \ref{sec: Keisler measures in NIP theories} and \ref{sec: Keisler measures in definable groups and their stabilizers}); type-definable fsg groups (Section \ref{sec: type-def fsg groups}); and properties of definable groups preserved in Shelah expansion (Section \ref{sec: props of grps preserved in Sh exp}). In Section \ref{sec: Externally definable fsg subgroups of definable groups} we consider externally definable fsg subgroups of definable groups. In Section \ref{sec: Transfer of properties of measures between T and T sh} we study which properties of measures transfer between $\Th(M)$ and $\Th(M^{\Sh})$, showing a 1-to-1 correspondence between   global generically stable measures in $T$ and in $\Th(M^{\Sh})$ (Lemma \ref{lem: unique extension of def meas to Sh}, Theorem \ref{thm: corresp for gs measures}). We  also give an example of a global $M$-definable type $p'$  in $\Th_{L'}(M^{\Sh})$ so that $p'\restriction_{L}$ divides over $M$ in $T$  (Example \ref{ex: def in Msh forking in M}). In Section \ref{sec: Externally definable fsg subgroups of definable groups are definable} we use this to prove Theorem  \ref{thm: main for fsg subgroups} and Corollary \ref{cor: fsg subgroups in RCVF}. In Section \ref{sec: Type-definable and externally definable fsg groups are definable} we prove Theorem \ref{thm: type-def and ext def fsg implies def}, and a related statement for definability of $\bigvee$-definable groups that are externally definable and fsg as such  (Proposition \ref{prop: ext def V def fsg implies definable}). In Section \ref{sec: Generically stable measures as a hyper-definable set} we consider a presentation for the set of global generically stable measures in an NIP theory as a hyper-definable  set.  In Section \ref{sec: abelian ext def subgroups} we give a description of externally definable, definably amenable subgroups of definable groups.
First we give a new argument that definable amenability of an $M$-definable group is preserved in $M^{\Sh}$ (Corollary \ref{cor: def am is preserved in Msh new proof}). Then we prove Proposition \ref{prop: ext def am descr in M'} and Theorem \ref{thm: ext def abelian final approx in M}, discussing some related issues on the way. We also note that every externally definable abelian subgroup is contained in a definable abelian subgroup (Proposition \ref{prop: ext def ab in def ab}). 
 In Section \ref{sec: Hyperdefinable group chunk for partial type-definable types} we establish a hyperdefinable group chunk theorem for partial type-definable types, generalizing \cite[Section 3.4]{hrushovski2019valued}. Section \ref{sec: Basics on hyperdefinability} develops type-definable partial types in $\cU^{\heq}$, and in Section \ref{sec: Hyperdefinable group chunk proof} we show the reconstruction of a hyper-definable group from a group chunk given on a type-definable partial type (Theorem \ref{thm: hyperdef group chunk gives group}), as well as uniqueness of the reconstructed group up to a type-definable isomorphism (Theorem \ref{thm: equiv of cats groups vs chunks}, Corollary \ref{cor: unique group from group chunk}). In Section \ref{sec: Externally definable fsg groups} we describe externally definable fsg groups. In Section \ref{sec: Type-definable filters associated to definable measures} we consider partial type-definable types $\pi_{\mu}$ induced by definable measures $\mu$, and deduce a ``measure theoretic'' group chunk (Corollary \ref{cor: measure theoretic group chunk}) from the results of the previous section. In Section \ref{sec: eliminable hyper-def group} we show that if a hyperdefinable group is type-definably isomorphic to a definable group in some expansion of the theory, it is already definable in $\cU^{\eq}$. In Section \ref{sec: Externally definable fsg groups are isomorphic to definable ones} we prove Theorem \ref{thm: ext def fsg groups is to definable}. We also give an example of an externally definable fsg  group that is not definable (Example \ref{ex: ext def fsg not def}) and prove Corollary \ref{cor: ext def fsg groups in RCVF}. Finally, in Section \ref{sec: discussion} we discuss some further related results, with an eye for future applications. In Section \ref{sec: acl in MSh} we describe the algebraic closure in $\Th(M^{\Sh})$ when $T$ has disintegrated algebraic closure. In Section \ref{sec: hyperdef sets from measures} we discuss uniform type-definability of certain sets arising from generically stable measures.
In Section \ref{sec: ext type-def sets} we give an internal description of externally type-definable and $\bigvee$-definable sets in NIP theories as intersections (respectively, unions) of (uniformly) directed type-definable (respectively, $\bigvee$-definable) families. Using this,  we give a (soft) description of externally definable subgroups of definable groups, as unions of directed families of their $\bigvee$-definable  subgroups (Proposition \ref{prop: ext def subgroups approx V-def}).
 
\section{Preliminaries}\label{sec: Preliminaries}

\subsection{Notation}
If $f$ is a (partial) function, we write $\Gamma_f$ to denote its graph. All first-order theories will be complete, unless stated otherwise. If $T'$ is an $L'$-theory and $T \subseteq T'$ is a reduct of $T'$ in a language $L \subseteq L'$, given a set $A \subseteq \cU' \models T'$ (hence $\cU' \restriction_{L} \models T$) we let $S^L_x(A)$ denote the space of types over $A$ in $T$ in a (possibly infinite) tuple of variables $x$. Given a formula $\varphi(x)$, $\varphi^0(x) := \neg \varphi(x), \varphi^1(x) := \varphi(x)$.

\subsection{Shelah expansion, honest definitions and a nice monster model}\label{sec: Sh exp, honest def}

\begin{definition}\label{def: Shelah exp context}
	We fix an NIP $L$-theory $T$, $M \models T$, and consider its \emph{Shelah expansion} $M^{\Sh}$, i.e.~an expansion by all externally definable subsets (of all sorts of $M$), and choose a well-behaved monster model for it. We let $N$ be an $|M|^{+}$-saturated elementary extension of $M$, let $(N,M)$ be an elementary pair in the language $L_P := L \cup \{P(x)\}$, and let $(N',M') \succ^{L_P} (N,M)$ be an $|N|^+$-saturated elementary extension of $(N,M)$.  Note that every externally definable subset of $M$ is of the form $\varphi(M,d)$ for some $\varphi(x,y) \in L$ and $d$ in $N$.
We let $M^{\Sh}$ be the expansion of $M$ in the language $L'=L^{\Sh}$ where we add a new relation symbol $R_{\varphi(x,d)}$ for every $\varphi(x,y) \in L$ and $d$ in $N$ (in particular, every element of $M$ gets named by a constant symbol; for $\varphi(x) \in L$, we identify $\varphi(x)$ and $R_{\varphi(x)}(x)$). We let $T' := \Th\left(M^{\Sh} \right)$. Note that then we can expand $M'$ to an $|N|^+$-saturated $L^{\Sh}$-elementary extension $\widetilde{M}'$ of $M^{\Sh}$ by interpreting $R^{\widetilde{M}'}_{\varphi(x,d)}(x)$ as $\varphi(M',d)$ (evaluated in $(N',M')$). We view $\widetilde{M}'$ as a monster model for $T'$. Note that in particular $M' \indep^{u}_{M} N$ in $T$ (i.e.~$\tp^L(M'/MN)$ is finitely satisfiable in $M$), so $M' \cap N = M$.
If we need a bigger monster model of $T'$, we choose an $|N'|^+$-saturated $(N'', M'') \succ^{L_P} (N',M')$ and expand $M''$ to an $L'$-structure $\widetilde{M}''$ via $R^{\widetilde{M}''}_{\varphi(x,d)}(x)$ as $\varphi(M'',d)$  (evaluated in $(N'',M'')$). Then $\widetilde{M}'' \succ^{L'} \widetilde{M}'$ is $|\widetilde{M}'|^+$-saturated. 

We will say that $M_1$ with $M \preceq^{L} M_1 \preceq^{L} M'$ is \emph{coherent} if there exists $N_1$ with $N \preceq^L N_1 \preceq^{L} N'$ so that $(N,M) \preceq^{L_P} (N_1, M_1) \preceq^{L_P} (N',M')$. In particular $M$ itself is a coherent model. Note that then automatically $M^{\Sh} \preceq^{L'} M_1 \preceq^{L'} \widetilde{M}'$ (with the $L'$-structure on $M_1$ induced from $\widetilde{M}'$), and that every small subset of $M'$ is contained in a small coherent model (by L\"owenheim--Skolem in the pair $(N',M')$).
\end{definition}

 The following is due to Shelah \cite{shelah2014strongly}, and follows from the existence of honest definitions (see Fact \ref{fac: honest defs} and \cite[Corollary 1.10]{chernikov2013externally}):
\begin{fact}\label{fac: Sh exp qe}
The theory $\Th_{L'} \left(M^{\Sh} \right)$	admits quantifier elimination.
\end{fact} 

\begin{remark}\label{rem: Shelah QE formulas}
	As externally definable sets are closed under Boolean combinations, it follows that every  formula $\varphi(x) \in L'$ is equivalent in $T'$ to $R_{\psi(x,d)}$ for some $\psi(x,z) \in L$ and $d \in N$.
\end{remark}

The following is a variant of stating the existence of \emph{honest definitions} over models in NIP theories:
\begin{fact}\cite[Corollary 1.3]{chernikov2013externally}\label{fac: honest defs}
	Let $T$ be an $L$-theory, $(N,M)$ be an $L_P$-structure with $M \prec N \models T$, $M$ named by $P$, and $(N',M') \succ (N,M)$ be $|M|^+$-saturated. Assume $\varphi(x,y) \in L$ is NIP, and $c \in N^y$. Here $x = (x_1, \ldots, x_n )$ and $y$ are arbitrary tuples of variables, and we write $P(x)$ for $\bigwedge_{i \in [n]} P_i(x_i)$.  Then there is $\theta(x,z) \in L$ and $d \in (M')^{z}$ such that $\varphi(M,c) = \theta(M,d)$ and $\theta(x,d) \land P(x) \vdash \varphi(x,c)$ (in $(N',M')$).  
	Moreover, no type $p(x) \in S^{L}_{x}(M')$ invariant over $M$ (in the sense of $T$) is consistent with $P(x)\land \varphi(x,c) \land \neg \theta(x,d)$ (again, in $\Th_{L_{P}}(N',M')$).

\

 
\tikzset{
pattern size/.store in=\mcSize, 
pattern size = 5pt,
pattern thickness/.store in=\mcThickness, 
pattern thickness = 0.3pt,
pattern radius/.store in=\mcRadius, 
pattern radius = 1pt}
\makeatletter
\pgfutil@ifundefined{pgf@pattern@name@_svhd73bv8}{
\pgfdeclarepatternformonly[\mcThickness,\mcSize]{_svhd73bv8}
{\pgfqpoint{0pt}{-\mcThickness}}
{\pgfpoint{\mcSize}{\mcSize}}
{\pgfpoint{\mcSize}{\mcSize}}
{
\pgfsetcolor{\tikz@pattern@color}
\pgfsetlinewidth{\mcThickness}
\pgfpathmoveto{\pgfqpoint{0pt}{\mcSize}}
\pgfpathlineto{\pgfpoint{\mcSize+\mcThickness}{-\mcThickness}}
\pgfusepath{stroke}
}}
\makeatother

 
\tikzset{
pattern size/.store in=\mcSize, 
pattern size = 5pt,
pattern thickness/.store in=\mcThickness, 
pattern thickness = 0.3pt,
pattern radius/.store in=\mcRadius, 
pattern radius = 1pt}
\makeatletter
\pgfutil@ifundefined{pgf@pattern@name@_m8oiaqtov}{
\pgfdeclarepatternformonly[\mcThickness,\mcSize]{_m8oiaqtov}
{\pgfqpoint{0pt}{0pt}}
{\pgfpoint{\mcSize+\mcThickness}{\mcSize+\mcThickness}}
{\pgfpoint{\mcSize}{\mcSize}}
{
\pgfsetcolor{\tikz@pattern@color}
\pgfsetlinewidth{\mcThickness}
\pgfpathmoveto{\pgfqpoint{0pt}{0pt}}
\pgfpathlineto{\pgfpoint{\mcSize+\mcThickness}{\mcSize+\mcThickness}}
\pgfusepath{stroke}
}}
\makeatother
\tikzset{every picture/.style={line width=0.75pt}} 

\begin{tikzpicture}[x=0.75pt,y=0.75pt,yscale=-1,xscale=1]

\draw   (99,260.89) -- (306.22,260.89) -- (306.22,444.2) -- (99,444.2) -- cycle ;
\draw    (204.2,445) -- (204.2,260.89) ;
\draw    (307.02,349.36) -- (99,349.36) ;
\draw [color={rgb, 255:red, 155; green, 155; blue, 155 }  ,draw opacity=1 ]   (99,233.8) -- (99,260.89) ;
\draw [color={rgb, 255:red, 155; green, 155; blue, 155 }  ,draw opacity=1 ]   (306.22,233.8) -- (99,233.8) ;
\draw [color={rgb, 255:red, 155; green, 155; blue, 155 }  ,draw opacity=1 ]   (306.22,260.89) -- (306.22,233.8) ;
\draw [color={rgb, 255:red, 155; green, 155; blue, 155 }  ,draw opacity=1 ]   (204.2,233.8) -- (204.2,260.89) ;
\draw [fill={rgb, 255:red, 155; green, 155; blue, 155 }  ,fill opacity=0.18 ]   (114.67,348.88) .. controls (112.28,399.36) and (167.54,429.64) .. (204.2,425.39) ;
\draw [fill={rgb, 255:red, 155; green, 155; blue, 155 }  ,fill opacity=0.18 ]   (114.68,349.48) -- (204.2,349.15) ;
\draw [fill={rgb, 255:red, 155; green, 155; blue, 155 }  ,fill opacity=0.18 ]   (204.2,425.39) -- (204.2,349.15) ;
\draw  [draw opacity=0][fill={rgb, 255:red, 155; green, 155; blue, 155 }  ,fill opacity=0.18 ] (204.2,425.39) -- (114.67,348.88) -- (203.92,348.56) -- cycle ;

\draw [pattern=_svhd73bv8,pattern size=10.5pt,pattern thickness=0.75pt,pattern radius=0pt, pattern color={rgb, 255:red, 155; green, 155; blue, 155}]   (114.68,348.68) .. controls (112.29,399.16) and (167.54,429.64) .. (204.2,425.39) .. controls (240.86,421.14) and (247.64,410.28) .. (260.79,394.26) .. controls (273.94,378.24) and (275.14,359.19) .. (273.54,349.1) .. controls (271.95,339) and (264.31,322.93) .. (249.63,308.18) .. controls (234.95,293.44) and (215.1,285.6) .. (205,283.74) .. controls (194.9,281.88) and (184.15,279.44) .. (163.03,283.74) .. controls (141.91,288.04) and (130.66,298.97) .. (124.24,307.65) .. controls (117.81,316.33) and (115.07,324.34) .. (114.68,348.68) -- cycle ;
\draw [pattern=_m8oiaqtov,pattern size=10.5pt,pattern thickness=0.75pt,pattern radius=0pt, pattern color={rgb, 255:red, 155; green, 155; blue, 155}]   (114.68,348.68) .. controls (112.29,399.16) and (167.54,429.64) .. (204.2,425.39) .. controls (240.86,421.14) and (223.6,407.01) .. (229.44,392.13) .. controls (235.29,377.26) and (241.66,374.6) .. (250.69,369.82) .. controls (259.73,365.04) and (288.42,381.51) .. (286.82,371.41) .. controls (285.23,361.32) and (290.34,311.77) .. (275.67,297.03) .. controls (260.99,282.28) and (214.03,309.51) .. (203.94,307.65) .. controls (193.84,305.79) and (195.44,298.62) .. (176.84,300.74) .. controls (158.24,302.87) and (154.52,308.18) .. (147.62,312.96) .. controls (140.71,317.75) and (131.68,321.47) .. (114.68,348.68) -- cycle ;
\draw    (267.17,408.6) -- (253.27,396.63) ;
\draw [shift={(251.76,395.32)}, rotate = 40.76] [color={rgb, 255:red, 0; green, 0; blue, 0 }  ][line width=0.75]    (10.93,-3.29) .. controls (6.95,-1.4) and (3.31,-0.3) .. (0,0) .. controls (3.31,0.3) and (6.95,1.4) .. (10.93,3.29)   ;
\draw    (117.86,300.21) -- (130.44,333.41) ;
\draw [shift={(131.15,335.28)}, rotate = 249.25] [color={rgb, 255:red, 0; green, 0; blue, 0 }  ][line width=0.75]    (10.93,-3.29) .. controls (6.95,-1.4) and (3.31,-0.3) .. (0,0) .. controls (3.31,0.3) and (6.95,1.4) .. (10.93,3.29)   ;
\draw    (141.05,422.14) -- (154.24,403.58) ;
\draw [shift={(155.4,401.95)}, rotate = 125.39] [color={rgb, 255:red, 0; green, 0; blue, 0 }  ][line width=0.75]    (10.93,-3.29) .. controls (6.95,-1.4) and (3.31,-0.3) .. (0,0) .. controls (3.31,0.3) and (6.95,1.4) .. (10.93,3.29)   ;

\draw (178.84,350.72) node [anchor=north west][inner sep=0.75pt]    {$M$};
\draw (285.35,350.92) node [anchor=north west][inner sep=0.75pt]    {$N$};
\draw (176.55,260.46) node [anchor=north west][inner sep=0.75pt]    {$M'$};
\draw (282.06,260.66) node [anchor=north west][inner sep=0.75pt]    {$N'$};
\draw (171.25,234.77) node [anchor=north west][inner sep=0.75pt]  [color={rgb, 255:red, 155; green, 155; blue, 155 }  ,opacity=1 ]  {$M''$};
\draw (273.96,234.97) node [anchor=north west][inner sep=0.75pt]  [color={rgb, 255:red, 155; green, 155; blue, 155 }  ,opacity=1 ]  {$N''$};
\draw (260.41,402.79) node [anchor=north west][inner sep=0.75pt]  [font=\footnotesize]  {$\varphi ( x,c)$};
\draw (100.7,279.91) node [anchor=north west][inner sep=0.75pt]  [font=\scriptsize]  {$\theta ( x,d)$};
\draw (101.47,425.72) node [anchor=north west][inner sep=0.75pt]  [font=\scriptsize]  {$\varphi ( M,c) =\theta ( M,d)$};

\end{tikzpicture}

\end{fact}

\begin{lemma}
	For any structure $M$, $(M^{\eq})^{\Sh}$ is a (proper in general) reduct of $(M^{\Sh})^{\eq}$, i.e.~every sort of $(M^{\eq})^{\Sh}$ is also a sort of $(M^{\Sh})^{\eq}$, and every $\emptyset$-definable subset of any product of sorts in $(M^{\eq})^{\Sh}$ is also $\emptyset$-definable in $(M^{\Sh})^{\eq}$.
\end{lemma}
\begin{proof}
It is clear that every sort of $(M^{\eq})^{\Sh}$ is also a sort of $(M^{\Sh})^{\eq}$.

Let $L$ be the language of $M$ and $L^{\Sh}$ the language of $M^{\Sh}$.
Let $N' \succ M^{\eq}$ be a saturated $L^{\eq}$-elementary extension, then $N' = N^{\eq}$ for some saturated $L$-elementary extension $N \succ M$.

Assume that $E_1, \ldots, E_n$ are some $L(\emptyset)$-definable equivalence relations in $T := \Th_L(M)$, let $S_{E_i}$  be the corresponding sorts in $T^{\eq}$, let $Y'$ be an $L^{\eq}(N^{\eq})$-definable subset of $\prod_{i=1}^n S_{E_i}^{N^{\eq}}$. Let $Y := Y' \cap \prod_{i=1}^n S_{E_i}^{M^{\eq}}$.

Say $Y'$ is defined in $N^{\eq}$ by $\varphi(y_1, \ldots, y_n, b_1, \ldots, b_m)$ for some 
$$\varphi(y_1, \ldots, y_n; y_{n+1}, \ldots, y_{n+m}) \in L^{\eq}$$
with the variable $y_i$ of sort $S_{E_i}$ and some $b_i \in S^{N^{\eq}}_{E_{n+i}}$. Let $f_{E_i}$ be the function symbol in $L^{\eq}$ for the quotient by $E_i$. 
By basic properties of $T^{\eq}$, there is some formula $\psi(x_1, \ldots, x_{n+m}) \in L$ with $x_i$ is the home sort so that 
\begin{gather*}
	T^{\eq} \vdash \varphi(f_{E_1}(x_1), \ldots, f_{E_{n+m}}(x_{n+m})) \equiv \psi(x_1, \ldots, x_{n+m}).
\end{gather*}
As $f_{E_i}$ are onto, there are some $a_{n+1}, \ldots, a_{n+m} \in N$  with $f_{E_i}(a_i) = b_i$. Let $X'$ be the subset defined in $N$ by $\psi(x_1, \ldots, x_n; a_{n+1}, \ldots, a_{n+m})$, then $X := X'(M)$ is definable in $M^{\Sh}$ and, using that each $f_{E_i}$ is onto in $M^{\eq}$ (and $M^{\eq}$ is closed under these functions), we get that $Y = f(X)$ for the $L^{\eq}(\emptyset)$-definable map
\begin{gather*}
	(x_1, \ldots, x_{n}) \mapsto (f_{E_1}(x_1), \ldots, f_{E_n}(x_n))
\end{gather*}
in $M^{\eq}$, hence $Y$ is definable (using a quantifier) in $(M^{\Sh})^{\eq}$.
\end{proof}
\begin{remark}
	In particular $((M^{\Sh})^{\eq})^{\Sh}$ is a reduct of $((M^{\Sh})^{\Sh})^{\eq}$, so, if $M$ is NIP, a reduct of $(M^{\Sh})^{\eq}$ already, so $(M^{\Sh})^{\eq}$ is closed under both taking $^{\Sh}$ and $^{\eq}$.
\end{remark}

\subsection{Keisler measures in NIP theories}\label{sec: Keisler measures in NIP theories}
For $r,s \in \mathbb{R}$ and $\varepsilon \in \mathbb{R}_{>0}$, we write $r \approx^{\varepsilon} s$ to denote $|r-s| \leq  \varepsilon$. We denote the collection of Keisler measures (in variable $x$ over $A$) as $\mathfrak{M}_{x}(A)$. Given $\mu \in \mathfrak{M}_{x}(A)$, we let $S(\mu)$ denote the support of $\mu$, i.e.~the (closed) set of all $p \in S_x(A)$ such that $\mu(\varphi(x))>0$ for every $\varphi(x) \in p$. Given a partial type $\pi(x)$ over $A$, we will consider the closed set 
\begin{gather*}
	\mathfrak{M}_\pi(A) := \left\{ \mu \in \mathfrak{M}_{x}(A) : p \in S(\mu) \Rightarrow p(x) \vdash \pi(x) \right\} \\
	= \left\{ \mu \in \mathfrak{M}_{x}(A) : \bigwedge_{\varphi(x) \in \pi} \mu(\varphi(x)) = 1 \right\}
\end{gather*}
of measures supported on $\pi(x)$. We will also write $S_{\pi}(A)$ for $\{p \in S_x(A) : p(x) \vdash \pi(x)\}$.

\begin{fact}\label{fac: types in sup of inv meas are inv NIP}(see e.g.~\cite[Lemma 2.10(3)]{chernikov2022definable}) 
	If $T$ is NIP and $\mu \in \mathfrak{M}_x(\cU)$ is $M$-invariant, then every type $p \in S(\mu)$ is also $M$-invariant.
\end{fact}

\begin{definition}\cite[Definition 2.2]{chernikov2013externally}
	Assume $M \prec \cU$. We say that $\nu \in \mathfrak{M}_x(\cU)$ is an \emph{heir} of $\mu \in \mathfrak{M}_x(M)$ if for any finitely many formulas $\varphi_i(x,a) \in L(\cU)$ (where $a$ is an arbitrary finite tuple in $\cU$) and $r_i \in [0,1)$, $i <n$, if $\bigwedge_{i < n} \nu(\varphi_i(x,a) > r_i)$ holds, then there exists some $b$ in $M$ so that $\bigwedge_{i < n} \mu(\varphi_i(x,b) > r_i)$ holds. 
\end{definition}
\begin{remark}\label{rem: unique def ext of definable measure}
	If the measure $\mu \in \mathfrak{M}_x(\cU)$ is \emph{definable} over $M$ (i.e.~$\mu$ is $M$-invariant, and for every $\varphi(x,y) \in L$ the function $F^{\varphi(x;y)}_{\mu, M}: S_{y}(M) \to [0,1]$, $q(y) \mapsto \mu(\varphi(x,b))$ for some/any $b \models q$, is continuous), then $\mu$ is the unique heir of $\mu|_{M}$ (see \cite[Remark 2.7]{NIP1}; they use a weaker notion of heir of a measure there, but uniqueness of heirs in their sense implies in particular the uniqueness of heirs in our sense).
\end{remark}

\begin{fact}\cite[Theorem 2.5]{chernikov2013externally}\label{fac: inv heir of a measure NIP}
	Assume $T$ is NIP and $M \models T$. Then every $\mu \in \mathfrak{M}_x(M)$ has a global extension $\nu \in \mathfrak{M}_x(\cU)$ which is both invariant over $M$ and an heir of $\mu$.
\end{fact}

We summarize some facts about generically stable measures in NIP theories (see \cite[Theorem 3.2]{NIP3}, \cite[Section 7.5]{simon2015guide}, \cite{zbMATH07456275}):

\begin{fact}\label{fac: NIP gen stab meas equivs}
Let $T$ be NIP and assume that $\mu \in \mathfrak{M}_x(\cU)$ is invariant over a small model $M$.  Then the following are equivalent:
	\begin{enumerate}
		\item $\mu$ is both definable and finitely satisfiable over $M$ (i.e.~for every $\varphi(x) \in L(\cU)$, if $\mu(\varphi(x)) > 0$ then $\models \varphi(a)$ for some $a \in M^x$);
		\item $\mu$ is \emph{fim} over $M$, i.e.~$\mu$ is Borel-definable (over $M$) and for any $\varphi(x,y) \in \mathcal{L}$ there exists a sequence of formulas $(\theta_{n}(x_1,\ldots,x_n))_{1 \leq n < \omega}$ in $\mathcal{L}(M)$ such that: 
	\begin{enumerate}
		\item for any $\varepsilon > 0$, there exists some $n_{\varepsilon} \in \omega$ satisfying: for any $k \geq n_{\varepsilon}$, if $\cU \models \theta_{k}(\abar)$ then 
		\begin{equation*} 
		\sup_{b \in \cU^{y}} |\Av(\abar)(\varphi(x,b)) - \mu(\varphi(x,b))| < \varepsilon;
		\end{equation*} 
		\item $\lim_{n \to \infty} \mu^{\otimes n} \left( \theta_n \left(\bar{x} \right) \right) = 1$.
	\end{enumerate}

		\item \label{item: fam} $\mu$ is \emph{fam} over $M$, that is for any formula $\varphi(x,y) \in L$ and $\varepsilon \in \mathbb{R}_{>0}$ there exist some $n \in \omega$ and $a_0, \ldots, a_{n-1} \in M^x$ so that for every $b \in \cU^y$,
		\begin{gather*}
		 \left \lvert \Av(\abar)(\varphi(x,b))	- \mu \left( \varphi(x,b) \right) \right \rvert \leq \varepsilon.
		\end{gather*}
		\end{enumerate}
If either of these properties holds for some small model $M$, we will say that $\mu$ is \emph{generically stable}. And we say that $\mu$ is \emph{generically stable over $A$}, for $A$ a small set, if $\mu$ is generically stable and $A$-invariant.
\end{fact}
\begin{remark}\label{rem: gs over a set but not fs}
If $\mu \in \mathfrak{M}_x(\cU)$ is generically stable over $A$, then it is definable over $A$ and finitely satisfiable (fim, fam) over any small model $M \supseteq A$ (but not necessarily finitely satisfiable in $A$ itself).
\end{remark}

We also have the following:
\begin{fact}\label{fac: basic props gen stab meas}
\cite[Proposition 3.3]{NIP3} Let $T$ be NIP and assume that $\mu \in \mathfrak{M}_x(\cU)$ is generically stable over $M$. Then  $\mu$ is the unique $M$-invariant extension of $\mu|_M$ (in particular its  unique global coheir; and its unique global heir).
\end{fact}

\begin{definition}\label{def: definable pushforward}
Let $f:\cU^{x} \to \cU^{y}$ be a partial $L(\cU)$-definable map and $\mu \in \mathfrak{M}_{x}(\cU)$.
\begin{enumerate}
	\item We say that \emph{$f$ is defined on $\mu$} if $\mu_x(\exists y \Gamma_{f}(x,y))=1$ (equivalently, if $p(x) \vdash \exists y \Gamma_{f}(x,y)$ for every $p \in S(\mu)$).
	\item If $f$ is defined on $\mu$,  we define the \emph{push-forward measure} $f_{*}(\mu)$ in $\mathfrak{M}_{y}(\cU)$, where for any formula $\varphi(y) \in \mathcal{L}_{y}(\mathcal{U}), f_*(\mu)(\varphi(y)) = \mu_x(\exists u (\Gamma_f(x,u) \land \varphi(u)))$. 

\end{enumerate}
\end{definition}

\begin{fact}\cite[Proposition 3.26]{chernikov2024definable}
We have $S(f_{\ast}\mu) = \{f_{\ast}p : p \in S(\mu)\}$, and	 if $\mu \in \mathfrak{M}_{x}(\cU)$ is generically stable over $M$ and $f$ is an $M$-definable map, then $f_{\ast}(\mu)$ is also generically stable over $M$.
\end{fact}

The following is a classical result on extending finitely additive measures to larger Boolean algebras:

\begin{fact}(\L o\'s-Marczewski \cite{Los1949})\label{fac: classic measure ext}
	Let $S$ be a set and $\mathcal{B}_0 \leq \mathcal{B}_1 \leq \mathcal{P}(S)$ be Boolean subalgebras, where $\mathcal{P}(S)$ is the power set. Let $\mu$ be a finitely additive probability measure on $\mathcal{B}_0$. Then there is a finitely additive probability measure $\nu$ on $\mathcal{B}_1$ extending $\mu$. Moreover, for any $X \in \mathcal{B}_1$ we can choose $\nu$ with $\nu(X) = r$ for any $r$ satisfying
	$$\sup \left\{ \mu(L) : L \in \mathcal{B}_0, L \subseteq X \right\} \leq r \leq \inf \left\{ \mu(U) : U \in \mathcal{B}_0, X \subseteq U \right\}$$
	(and clearly any $\nu$ extending $\mu$ has to satisfy this for every $X \in \mathcal{B}_1$).
\end{fact}

We will also need the following fact on finding generically stable measures in NIP theories:
\begin{fact}\cite[Proposition 3.3]{zbMATH06025969}\label{fac: symmetrization measure}
	Let $T$ be NIP, $M$ a small model and $\mu \in \mathfrak{M}_x(\cU)$ a global $M$-invariant measure. Then there is $\mu' \in \mathfrak{M}_x(\cU)$ so that:
	\begin{enumerate}
		\item $\mu'|_{M} = \mu|_{M}$;
		\item $\mu'$ is generically stable over some $M' \succ M$;
		\item if $f: \cU^x \to \cU^x$ is an $M$-definable function such that $f_{\ast}(\mu) = \mu$, then also $f_{\ast}(\mu') = \mu'$.
	\end{enumerate}
\end{fact}

\subsection{Keisler measures in type-definable groups and their stabilizers}\label{sec: Keisler measures in definable groups and their stabilizers}

\begin{definition}\label{def: definable conv} Suppose that $G$ is an $M$-type-definable group and $\mu \in \mathfrak{M}_{G}(\cU)$ is Borel-definable. Then for any measure $\nu \in \mathfrak{M}_{G}(\cU)$, the (definable) \emph{convolution of $\mu$ and $\nu$}, denoted $\mu * \nu$, is the unique measure in $\mathfrak{M}_{G}(\cU)$ such that for any formula $\varphi(x) \in \mathcal{L}(\cU)$, 
\begin{equation*} 
(\mu * \nu)(\varphi(x)) = ( \mu \otimes \nu)(\varphi(x \cdot y)). 
\end{equation*} 
We say that $\mu$ is \emph{idempotent} if $\mu * \mu = \mu$. 
\end{definition} 
\begin{remark}
	When $T$ is NIP, it suffices to assume that $\mu$ is invariant (under $\Aut(\mathbb{M}/M)$ for some small model $M$), as then $\mu$ is automatically Borel-definable (\cite{NIP3}). We refer to \cite[Section 3]{chernikov2022definable} for a detailed discussion of when convolution is well-defined, and to \cite{chernikov2022definable, chernikov2023definable2} for a study of the convolution semigroups on measures.
\end{remark} 

\begin{remark}\label{rem: type-def group formula}
		If $G$ is a type-definable (over $M$) group, without loss of generality defined by a partial type $G(x)$ closed under conjunctions, by compactness we may assume that $\cdot$ is an $M$-definable partial function defined, associative and with left and right cancellation on $\varphi_0(\cU)$, and $^{-1}$ is defined on $\varphi_0(\cU)$, for some $\varphi_{0}(x) \in G(x)$ (but $\varphi_0(\cU)$ is not necessarily closed under $\cdot$ or $^{-1}$). 
\end{remark}

\begin{definition}\label{def: stab of a measure}
For $G$ an $M$-type-definable group, $N \succeq M$, we have an action of $G(N)$ on $\mathfrak{M}_{G}(N)$ via $g \cdot \mu (\varphi(x)) := \mu_x(\varphi(g \cdot x))$ (where $\varphi(g \cdot x)$ is an abbreviation for $\exists u (\Gamma_{\cdot_{G}}(g,x,u) \land \varphi(u))$). For a measure $\mu \in \mathfrak{M}_{G}(\cU)$, we let 
$$\Stab(\mu) := \{g \in G(\cU) : g \cdot \mu = \mu \} $$ 
denote the left stabilizer of $\mu$. Similarly, we let $\Stab^{\leftrightarrow}(\mu) := \{g \in G(\cU) : g \cdot \mu = \mu \cdot g = \mu \}$ denote the two-sided stabilizer of $\mu$.
\end{definition}
\begin{fact}\label{fac: stab of def meas type def}(see \cite[Proposition 5.3]{chernikov2022definable})
	When $G = G(x)$ is an $M$-type-definable group and $\mu \in \mathfrak{M}_{G}(\cU)$ is a measure definable over $M \prec \cU$, then $\Stab_G(\mu)$ is an $M$-type-definable subgroup of $G(\cU)$. 
	
	More precisely, $\Stab_G(\mu) = G \cap \bigcap_{\varphi(x;y) \in  L,  \varepsilon \in \mathbb{Q}_{>0}} \Stab_{G,\varphi(x;y), \varepsilon}(\mu)$, where each $\Stab_{G,\varphi(x;y), \varepsilon}(\mu) \subseteq \varphi_{0}(\cU)$ (see Remark \ref{rem: type-def group formula})  is an $M$-definable set (not necessarily a subgroup or even a subset of $G$) so that:
	\begin{enumerate}
		\item  for any $b \in \cU^y$ and any $g  \in  \Stab_{G,\varphi(x;y), \varepsilon}(\mu)$, 
	$$\left \lvert  \mu(\varphi(x,b)) - \mu (\varphi(g \cdot x,b) )\right \rvert \leq \varepsilon;$$
	\item given $g$, if for all $b \in \cU^y$ we have $$\left \lvert  \mu(\varphi(x,b)) - \mu (\varphi(g \cdot x,b))\right \rvert \leq \varepsilon,$$
	 then $g \in \Stab_{G,\varphi(x;y), 2 \varepsilon}(\mu)$.
	\end{enumerate}
	And similarly for $\Stab^{\leftrightarrow}(\mu)$.
\end{fact}

\begin{fact}\cite[Proposition 3.37]{chernikov2024definable}\label{fac: pushforward gen stab}
Let $G$ be an $M$-definable group, and $\mu \in \mathfrak{M}_{G}(\cU)$ an idempotent generically stable measure. Then the following are equivalent:
\begin{enumerate}
	\item $\mu \in \mathfrak{M}_{H}(\cU)$, where $H := \Stab_{G}(\mu)$;
	\item $\mu^{(2)} = f_{\ast}\left(\mu^{(2)} \right)$, where $f: \left( \cU^{x} \right)^2 \to \left( \cU^{x} \right)^2$ is the $M$-definable map $f(x_1,x_0) = (x_1 \cdot x_0, x_0)$ (where $\cdot$ is viewed as a globally defined function whose restriction to $G$ defines the group operation).
	\item $\mu \otimes p = f_{\ast}(\mu \otimes p)$ for every $p \in S(\mu)$.
\end{enumerate}
If any of these conditions hold, we say that $\mu$ is \emph{generically transitive}.
\end{fact}

We recall:
\begin{definition}
	A type-definable group $G$ is \emph{definably amenable} (\emph{definably extremely amenable}) if there is a left-$G(\cU)$-invariant measure $\mu \in \mathfrak{M}_{G}(\cU)$ (respectively, a left-$G(\cU)$-invariant type $p \in S_{G}(\cU)$).
\end{definition}

\begin{fact}\cite[Lemma 5.8]{NIP2}\label{fac: def am def measure}
	If $T$ is NIP and $G$ is $M$-type-definable and definably amenable, then there exists a left-$G(\cU)$-invariant measure $\mu \in \mathfrak{M}_{G}(\cU)$ which is definable over some small model $M' \succ M$.
\end{fact}

\begin{remark}
	If $G$ is definable and $H \leq G$ is type-definable by a partial type $\pi(x)$ and definably amenable, then $H$ is already type-definable by some $\pi_0 \subseteq \pi$ with $|\pi_0| \leq |T|$.
\end{remark}
\begin{proof}
	By Fact \ref{fac: def am def measure} there is a left $H(\cU)$-invariant measure $\mu \in \mathfrak{M}_{H}(\cU)$ definable over some small model $M \prec \cU$. But by definition, every definable measure is definable over some set of size $\leq |T|$, hence $\Stab_{G}(\mu) \leq G(\cU)$ is also type-definable by some $\pi_1(x)$ of size $\leq |T|$ by Fact \ref{fac: stab of def meas type def}. Clearly $H(\cU) \subseteq \Stab_{G}(\mu)$. Conversely, $\Stab_{G}(\mu) \leq  H$: otherwise there is some $g \in \Stab_{G}(\mu) \setminus H$, so $g \cdot H \cap H = \emptyset$. By compactness this implies $g \cdot \psi(x) \cap \psi(x) = \emptyset$ for some $\psi \in \pi$, and $\mu(\psi(x) = \mu( g \cdot \psi(x)) = 1$ (as $\mu \in \mathfrak{M}_{H}(\cU)$) --- a contradiction. So $\pi(\cU) = \pi_1(x)$, and by compactness this implies that $\pi(x)$ is equivalent to some $\pi_0 \subseteq \pi$ with $|\pi_0| \leq |T|$.
\end{proof}

\subsection{Type-definable fsg groups}\label{sec: type-def fsg groups}
We summarize some facts on (type-) definable fsg groups, see \cite[Remark 4.4]{NIP3}, and \cite[Remark 3.3]{chernikov2024definable} for the type-definable case.

\begin{definition}\cite[Definition 6.1]{NIP2}
	Let $G$ be a type-definable group. We say that $G$ is \emph{fsg} (\emph{finitely satisfiable generics}) if there  exist some $p \in S_G(\cU)$ and small $M \prec \cU$ such that $G$ is type-definable over $M$ and for every $g \in G(\cU)$, the type $g \cdot p$ is finitely satisfiable in $G(M)$ (we stress that finite satisfiability in $G(M)$ is a stronger requirement than finite satisfiability in $M$ when $G$ is only type-definable rather than definable). In this case we will say that $G$ is fsg \emph{over $M$}.
\end{definition}

\begin{proposition}\label{prop: fsg type-def arb models}
	If $G$ is type-definable over $M$ and fsg, then for any $N \succ M$ which is $|M|^{+}$-saturated, $G$ is fsg over $N$.
\end{proposition}
\begin{proof}
	We adapt the proof from \cite[Remark 4.4]{NIP1} for definable groups. Let $N \succ M$ be any $|M|^{+}$-saturated model.

	By \cite[Lemma 6.2(ii)]{NIP2}, there is a global generic type $p$ of $G$, namely $p \in S_G(\cU)$ so that any $\varphi(x) \in p$ is generic, i.e.~there are $n_{\varphi} \in \omega$ and $g_1, \ldots, g_{n_{\varphi}} \in G(\cU)$ so that $G(\cU) \subseteq  \bigcup_{i \in [n_{\varphi}]} g_i \cdot (\varphi(\cU) \cap G(\cU))$. By \cite[Lemma 6.2(iii)]{NIP2}, there is a small $M_0 \prec \cU$ so that any generic $p \in S_G(\cU)$ is finitely satisfiable in $G(M_0)$, and for any generic $\varphi(x) \in L(\cU)$, the translates $g_1, \ldots, g_{n_{\varphi}}$ witnessing this can be chosen in $G(M_0)$.

	Now assume that $\psi(x) \in L(\cU)$ is generic, say $\psi(x) = \varphi(x,b)$ for some $\varphi(x,y) \in L, b \in \cU^y$, and  $G(\cU) \subseteq  \bigcup_{i \in [k]} g_i \cdot (\varphi(\cU,b) \cap G(\cU))$ for some $g_1, \ldots, g_{k} \in G(\cU)$. Let $q(y) := \tp(b/M)$. As $G$ is type-definable over $M$, $\varphi(x,b')$ is generic for any $b' \models q$ in $\cU$ (witnessed by any $g'_1, \ldots, g'_k$ so that $(b,g_1, \ldots, g_k) \equiv_M (b',g'_1, \ldots, g'_k)$).
	Then the partial type $q(y) \cup \{ \neg \varphi(d,y) : d \in G(M_0) \}$ is inconsistent, as every generic formula meets $G(M_0)$ by the choice of $M_0$. That is, there exists a finite tuple $\bar{d} = (d_1, \ldots, d_n) \in G(M_0)^n$ so that for every $b' \models q$ in $\cU$ we have $\models \bigvee_{i \in [n]} \varphi(d_i,b')$. As $N$ is $|M|^{+}$-saturated, there is a tuple $\bar{d}'$ in $N$ with $\bar{d}' \equiv_M \bar{d}$. In particular, $\bar{d}' \in G(N)^n$ and we still have that for every $b' \models q$ in $\cU$,  $\models \bigvee_{i \in [n]} \varphi(d'_i,b')$. This shows that $\psi(x)$ meets $G(N)$.
	
	As every $G(\cU)$-translate of a generic type is generic, this shows that $g \cdot p$ is finitely satisfiable in $G(N)$ for all $g \in G(\cU)$.
	\end{proof}

\begin{proposition}\label{prop: fsg iff inv gen stab meas}
Let $T$ be NIP and $G$ a group  type-definable over $M$. The following are equivalent:
\begin{enumerate}
	\item  $G$ is fsg;
	\item $G$ is fsg over any $|M|^{+}$-saturated $N \succ M$;
	\item there exists a generically stable left-$G(\cU)$-invariant  measure $\mu \in \mathfrak{M}_{G}(\cU)$;

	\item there exists a left-$G(\cU)$-invariant  measure $\mu \in \mathfrak{M}_{G}(\cU)$ generically stable over any $|M|^{+}$-saturated  $N \succ M$.
	\end{enumerate}
\end{proposition}
\begin{proof}
(1) implies (2) by Proposition \ref{prop: fsg type-def arb models}. 

(2) implies (4). Given $p \in S_G(\cU)$ such that $g \cdot p$ is finitely satisfiable in an arbitrary $N \succ M$ for all $g \in G(\cU)$, the standard construction of a left-invariant generically stable over $N$ measure $\mu_p$ in an fsg group goes through in the type-definable case (see \cite[Proposition 6.2]{NIP1} and \cite[Remark 4.4]{NIP3}).

(4) implies (3) is trivial.

(3) implies (1). If (3) holds for $\mu \in \mathfrak{M}_G(\cU)$ generically stable over some $|M|^+$-saturated $N \succ M$, then the proof of \cite[Proposition 3.26]{chernikov2024definable}) shows that $\mu$ is finitely satisfiable in $G(N)$, so any $p \in S(\mu)$ satisfies (1) with respect to $N$.
\end{proof}

In the case of definable fsg groups, we can omit the saturation requirements on the model.
\begin{fact}\cite[Proposition 3.4]{chernikov2014external}\label{fac: heir of a measure remains G-inv}
	If $G$ is an $M$-definable group, $\mu \in \mathfrak{M}_G(M)$ is $G(M)$-invariant and $\nu$ is a global heir of $\mu$, then $\nu \in \mathfrak{M}_{G}(\cU)$ is $G(\cU)$-invariant.
\end{fact}

\begin{fact}\label{fac: fsg groups basic props}
	Assume $G$ is type-definable and fsg, witnessed by some left $G$-invariant generically stable measure $\mu \in \mathfrak{M}_{G}(\cU)$ (using Proposition \ref{prop: fsg iff inv gen stab meas}).
	\begin{enumerate}
		\item $\mu$ is the unique left $G$-invariant, as well as the unique right $G$-invariant measure in $\mathfrak{M}_{G}(\cU)$.
		\item If $T$ is NIP and $G$ is in fact a definable group over $M \prec \cU$, then $\mu$ is generically stable over any small model $M \prec \cU$ so that $G$ is $M$-definable.
	\end{enumerate}
\end{fact}
\begin{proof}
(1)  By \cite[Theorem 4.3]{NIP3} for definable groups.  See \cite[Proposition 3.33]{chernikov2024definable} for type-definable groups, and without assuming NIP.

	(2) By Fact \ref{fac: inv heir of a measure NIP}, let $\mu' \in \mathfrak{M}_{G}(\cU)$ be a global $M$-invariant heir of $\mu|_{M}$. As $G$ is $M$-definable and $\mu|_{M}$ is $G(M)$-invariant, it follows by Fact \ref{fac: heir of a measure remains G-inv} that $\mu'$ is $G(\cU)$-invariant. But then $\mu = \mu'$ by (1). So $\mu$ is generically stable, and invariant over $M$, hence generically stable over $M$.
\end{proof}

	Recall:
	\begin{definition}\cite{PiTa}
A global type $p \in S_x(\cU)$ is \emph{generically stable} if it is $A$-invariant for some small $A \subset \cU$, and for any ordinal $\alpha$ (or just for $\alpha = \omega + \omega$), $(a_i : i \in \alpha)$ a Morley sequence in $p$ over $A$ and formula $\varphi(x) \in \cL(\cU)$, the set $\{i \in \alpha : \models \varphi(a_i) \}$ is either finite or co-finite. 
\end{definition}
\begin{remark}\label{rem: gen stab one Morley seq}
	As $\tp(\bar{a}/A) = \tp(\bar{a}'/A)$ for any two Morley sequences in $p$ over $A$ of the same order type, to check generic stability of $p$ it is enough to find \emph{some}  Morley sequence in it of order type $\alpha := \omega + \omega$ so that the set $\{i \in \alpha : \models \varphi(a_i) \}$ is either finite or co-finite (follows by taking an automorphism of $\cU$ over $A$ sending one Morley sequence to the other and replacing $ \varphi(x)$ by its image under this automorphism).
\end{remark}
\begin{remark}
	We note that in any theory, $p \in S_x(\cU)$ is generically stable over $M$ if and only if it is fim over $M$, viewed as a $\{0,1\}$-valued Keisler measure (see \cite[Proposition 3.2]{conant2020remarks}).
\end{remark}

We will also consider the special case of an fsg group with a $\{0,1\}$-valued measure:
\begin{definition} \cite[Definition 2.1]{PiTa}\label{def: generically stable group}
	A type-definable group $G(x)$ is (connected) \emph{generically stable} if there is a generically stable $p \in S_{G}(\cU)$ which is left $G(\cU)$-invariant, i.e. $g \cdot p = p$ for all $g \in G(\cU)$.
		\end{definition}

\begin{remark}\label{rem: def iso preserves fsg}
	In any theory $T$, assume $G$ and $H$ are groups (type-)definable over $M$ and $f: G \to H$ is an $M$-definable surjective homomorphism. Then if $G$ is definably amenable (fsg, generically stable), then so is $H$.
\end{remark}
\begin{proof}
Assume $\mu \in \mathfrak{M}_{G}(\mathbb{M})$ is left-$G(\cU)$-invariant. Then the push-forward measure $f_{\ast}\mu \in \mathfrak{M}_{H}(\mathbb{M})$ is left-$H(\cU)$-invariant. Indeed, fix any $\varphi(y) \in \cL(\cU)$ and $h \in H(\cU)$, as $f$ is surjective take $g \in G(\cU)$ with $f(g) = h$. Then $f_{\ast}\mu \left( \varphi( h \cdot_H y) \right) = \mu ( \varphi(h \cdot f(x)))$ by definition of push-forward measure, $= \mu (\varphi(f(g) \cdot_H f(x) )) = \mu (\varphi(f(g \cdot_G x) ))$ as $f$ is a homomorphism, $= \mu (\varphi(f(x) ))$ by left-$G(\cU)$-invariance of $\mu$, $= f_{\ast}(\varphi(y))$ by definition of push-forward measure again.

And if $\mu$ is generically stable over $M$, then $f_{\ast} \mu$ is also generically stable over $M$ (e.g.~fim is preserved, see \cite[Proposition 3.26]{chernikov2024definable}).
\end{proof}

\subsection{Properties of definable groups preserved in Shelah expansion}\label{sec: props of grps preserved in Sh exp}

We summarize some known facts about externally definable groups, mostly from \cite{chernikov2013externally}.
\begin{fact}\label{fac: group props preserved in Sh exp}
	Let $T$ be an NIP $L$-theory, $M \models T$, $ G = G(\cU)$ is an $M$-definable group.
	Let $T' := \Th(M^{\Sh})$ in $L'$.
	\begin{enumerate}
	
		\item $G$ is fsg in $T$ if and only if $G$ is fsg in $T'$.
		\item $G$ is generically stable in $T$ if and only if $G$ is generically stable in $T'$.
		\item $G$ is definably (extremely) amenable in $T$ if and only if $G$ is definably (extremely) amenable in $T'$.
		\item $G^0/G^{00}/G^{\infty}$ calculated in $T$ is equal to $G^0/G^{00}/G^{\infty}$ calculated in $T'$ (that is, in its monster model $M'$).
		\item In particular, if $H \leq G(M)$ is an externally definable subgroup of finite index, then it is already (internally) definable.
	\end{enumerate}
\end{fact}
\begin{proof}
(1) If $G$ is fsg in $T'$, there is some $p' \in S^{L'}_{G}(M')$ such that $g \cdot p'$ is finitely satisfiable in $M$ for all $g \in G(M')$ (e.g.~using Fact \ref{fac: fsg groups basic props}(2), taking any type in the support of such a measure). Let $p := p'\restriction_{L}$, then for all $g \in G(M')$ we have $g \cdot p = (g \cdot p')\restriction_{L}$ is finitely satisfiable in $M$. Alternatively, we can start with a generically stable over $M$ measure $\mu' \in \mathfrak{M}^{L'}_{G}(M')$ which is $G(M')$-invariant (using Fact \ref{fac: fsg groups basic props}(2)), then $\mu := \mu'\restriction_{L} \in \mathfrak{M}^{L}_{G}(M')$ is generically stable in $T$ by Proposition \ref{prop: properties of reducts of measures}(1), hence $G$ is fsg in $T$. The converse is by \cite[Theorem 3.19]{chernikov2014external} (using Fact \ref{fac: fsg groups basic props}(2), we can always find a type that is fsg in $G$ over $M$).

(2) Restricting the argument in (1) to types.

	(3) If $G$ is definably amenable (definably extremely amenable) in $T'$, there is a measure $\mu' \in \mathfrak{M}^{L'}_{G}(M')$ (a type $p' \in S^{L'}_{G}(M')$) which is left $G(M')$-invariant. But then $\mu := \mu'\restriction_{L} \in \mathfrak{M}^{L}_{G}(M')$ ($p := p'\restriction_{L} \in S^{L}_{G}(M')$) is still $G(M')$-invariant, witnessing that $G$ is definably amenable (definably extremely amenable) in $T$. The converse is by \cite[Theorem 3.17]{chernikov2014external}.
	
	(4),(5) By Theorem 4.5 and Corollaries 4.15 and 4.20 in \cite{chernikov2014external}.
\end{proof}

\section{Externally definable subgroups of definable groups}\label{sec: Externally definable fsg subgroups of definable groups}

\subsection{Transfer of properties of measures between $T$ and $\Th(M^{\Sh})$}
\label{sec: Transfer of properties of measures between T and T sh}

\begin{proposition}\label{prop: properties of reducts of measures}
	Assume $T \subseteq T'$ are complete $L \subseteq L'$-theories and $T'$ is NIP. Let $\mathbb{M}' \models T'$ be a monster model, then $\mathbb{M}' := \mathbb{M}' \restriction L$ is a monster model for $T$. Let $M' \prec \mathbb{M}'$ be a small model of $T'$, and $M := \mathbb{M} \restriction L$.
	\begin{enumerate}
		\item \label{item: gen stab of meas preserved in reduct} If $\mu'(x) \in \mathfrak{M}^{L'}_x\left(  \mathbb{M}' \right)$ is  generically stable  over $M$ in $T'$, then $\mu := \mu\restriction_{L}$ is generically stable over $M$ in $T$.
		\item  Assume $G$ is an $L(M)$-definable group and $\mu \in \mathfrak{M}^{L'}_{G}(\mathbb{M})$ is finitely satisfiable in $M$ and idempotent. Then $\mu$ is also idempotent.
		\item Assume $G$ is an $L(M)$-definable group and $\mu' \in \mathfrak{M}^{L'}_{G}(\mathbb{M}')$ is generically stable over $M$ and generically transitive. Then $\mu$ is also generically transitive (in $T$).
	\end{enumerate}
\end{proposition}
\begin{proof}
	(1) Fam clearly transfers to reducts: by assumption  $\mu'$ satisfies Fact \ref{fac: NIP gen stab meas equivs}\eqref{item: fam} in $T'$; as for any $L$-formula $\varphi(x,y)$ and $b \in \cU^y$ we have $\mu'(\varphi(x,b)) = \mu(\varphi(x,b))$, it follows that $\mu$ also satisfies Fact \ref{fac: NIP gen stab meas equivs}\eqref{item: fam} in $T$ (which is also NIP).
	
	(2) First note that if $\mu'$ is finitely satisfiable in $M$ (i.e.~for every $\varphi(x) \in L'(M')$, if $\mu'(\varphi(x))>0$ then $\varphi(M) \neq \emptyset$), then $\mu$ is also finitely satisfiable in $M$, hence $\mu$ is $M$-invariant (in $T$), and so it is Borel-definable over $M$ in $T$ as $T$ is NIP. We have:
	
	\begin{claim}\label{cla: restriction commutes with tensor}
		$\left( \mu' \otimes \nu' \right) \restriction_{L} = \left( \mu' \restriction_{L}  \right) \otimes \left( \nu' \restriction_{L}  \right)$ for an arbitrary $\nu' \in S^{L'}_x(\cU')$.
	\end{claim}
	\begin{proof}
	Indeed, fix an arbitrary formula $\varphi(x,y) \in L(\cU)$, and let $M \preceq N \prec \cU'$ be a small  model so that $\varphi(x,y) \in L(N)$. 
Let $f: S^{L'}_y(N) \to S^{L}_{y}(N)$ be defined via $f(q) := q \restriction_{L}$, then $f$ is a continuous surjection. Let $\tilde{\nu}'$ be the unique regular Borel probability measure on $S^{L'}_y(N)$ extending $\nu'|_{N}$, and let $\tilde{\nu}$ be the unique regular Borel probability measure on $S^{L}_y(N)$ extending $\nu|_{N}$, where $\nu := \nu'\restriction_{L}$. Let $f_{\ast} \tilde{\nu}'$ be the pushforward measure of $\tilde{\nu}'$, we then have $f_{\ast} \tilde{\nu}' = \tilde{\nu}$ (as they clearly agree on the clopens). Note also that for any  $\left( F^{\varphi(x;y)}_{\mu, N} \circ f \right) (q) = F^{\varphi(x;y)}_{\mu', N}(q)$ by $M$-invariance of $\mu$. Then have 
	\begin{gather*}
		\left(\mu \otimes \nu \right) \left( \varphi(x,y) \right) =  \int_{S^L_y(N)} F^{\varphi(x;y)}_{\mu, N} d \tilde{\nu} = \int_{S^{L'}_y(N)} \left( F^{\varphi(x;y)}_{\mu, N} \circ f  \right) d \tilde{\nu}'\\
		 = \int_{S^{L'}_y(N)} F^{\varphi(x;y)}_{\mu', N}  d \tilde{\nu}' = \left(\mu' \otimes \nu' \right) \left( \varphi(x,y) \right),
	\end{gather*}
where the first and last equalities are by definition of $\otimes$, the second equality is by the change of variables formula for push-forward measures (and as usual $F_{\mu,N}^{\varphi(x;y)}: S_{y}(N) \to [0,1]$ is defined by $F_{\mu,N}^{\varphi(x;y)}(q) = \mu(\varphi(x;b))$ for some (equivalently, by invariance of $\mu$, any) $b \models q$ in $\cU$).
	\end{proof}

	It follows that $\mu \ast \mu = \left(\mu' \ast \mu' \right)\restriction_{L} = \mu' \restriction_{L} = \mu$.
	
	(3) Let $\varphi(x) \in L(\cU)$ be arbitrary, we want to show that $\mu^{\otimes 2}(\varphi(x_1 \cdot x_0, x_0)) = \mu^{\otimes 2} (\varphi(x))$.
	By the above and generic transitivity of $\mu'$ we have
	\begin{gather*}
		\mu^{\otimes 2}(\varphi(x_1 \cdot x_0, x_0)) = (\mu')^{\otimes 2}(\varphi(x_1 \cdot x_0, x_0)) \\
		= (\mu')^{\otimes 2}(\varphi(x_1, x_0)) = (\mu)^{\otimes 2}(\varphi(x_1, x_0)),
	\end{gather*}
 hence $\mu$ is generically transitive.
	\end{proof}
	
	\begin{remark}
		It was pointed out to me by Kyle Gannon that fim transfers to reducts in arbitrary theories.  Indeed, by \cite[Lemma 3.6]{arXiv:2308.01801}, in any theory $T$, if $\mu$ is a global measure definable over a small model $M$ (so $\mu$ is invariant over $M$, and for every $\varphi(x,y) \in L$, the map $F^{\varphi}_{\mu}: q \in S_y(M) \mapsto \mu(\varphi(x,b)) \in [0,1]$, for some/any $b \models q$, is continuous), then $\mu$ is fim if and only if $\lim_{n \to \infty} \int_{S_{\bar{x}}(M)} \chi^{\varphi}_{\mu,n} d(\mu^{\otimes n})' = 0$, where $\bar{x} = (x_1, \ldots, x_n)$, $\chi^{\varphi}_{\mu,n}: p \in S_{\bar{x}}(M) \mapsto \sup_{b \in \cU^y}|\Av(\bar{a};\varphi(x,b)) - F^{\varphi}_{\mu}(\tp(b/M))| \in [0,1]$ for some $\bar{a} \models p$ (equivalently any $\bar{a} \models p$; and the map $\chi^{\varphi}_{\mu,n}$ is continuous by the assumption on $F^{\varphi}_{\mu}$).  Now if  $\mu'\in \mathfrak{M}^{L'}_x\left(  \mathbb{M}' \right)$ is fim in $T'$, it is in particular fam, hence $\mu := \mu'\restriction_{L}$ is also fam, hence definable, in the reduct $T$ (in fact, by \cite[Theorem 4.8]{conant2020remarks}, dfs is preserved under reducts). Using Claim \ref{cla: restriction commutes with tensor}, as the value $ \chi_{\mu,n}^{\varphi}$ depends only on the reduct of $\mu'$ to the language containing $\varphi$, the aforementioned integral remains $0$ calculated in $T$.
	\end{remark}

Finite satisfiability is used crucially to obtain invariant types/measures in the reduct. The following is an example of a global $M$-definable type $p'$  in $\Th_{L'}(M^{\Sh})$ so that $p'\restriction_{L}$ divides over $M$ in $T$ (motivated by an example of forking for formulas not being preserved in reducts suggested by Atticus Stonestrom):
\begin{example}\label{ex: def in Msh forking in M}

	We let $T$ be the $L := \{ \land  \}$-theory of infinitely branching dense meet trees (with associated partial order $x \leq y \iff x \land y = x$), then $T$ is complete, NIP (and dp-minimal) and has quantifier elimination (see for example \cite[Fact 3.10]{CHERNIKOV_MENNEN_2025}). 
		
	Let $M \models T$ be arbitrary, let $N \succ M$ be $|M|^{+}$-saturated, let $X$ be a branch in $M$, i.e.~a maximal  subset of $M$ of pairwise comparable elements (hence linearly ordered). 
	Note that
	\begin{gather}
		\forall n \in \omega, a_1, \ldots, a_n \in M, \exists x \in X \bigwedge_{i \in [n]} x \not \leq a_i. \label{eq: def type reduct ex -1}
	\end{gather}
	(indeed, if $a_i \notin X$ then $a_i \perp a'_i$ for some $a'_i \in X$ by maximality; by density of the tree $X$ has no maximal element, choose $x \in X$ with $x > \{a_i : i \in X\} \cup \{a'_i : a_i \notin X\}$. Then $x \not \leq a_i$ for all $i \in [n]$ --- for $a_i \notin X$, otherwise $a'_i < x \leq a_i $).
	
	Then $X$ is externally definable, by $x < c$ for some $c \in N, c > X$ (in $N$). We use the notation from Definition \ref{def: Shelah exp context}, so let $(N'', M'') \succ^{L_P} (N', M') \succ^{L_P} (N,M)$ be elementary extensions, with each pair saturated over the previous one; this gives us saturated elementary extensions $M^{\Sh} \prec^{L'}  \widetilde{M}' \prec^{L'} \widetilde{M}''$.  
	
	By \eqref{eq: def type reduct ex -1}, $(N,M )$ satisfies: for any $n \in \omega$, $a_1, \ldots, a_n \in M$, there exists $a \in M, a < c, a \not \leq a_i$ for all $i \in [n]$. Then the same holds in $(N',M')$ by elementarity, so by saturation of the pair $(N'', M'')$ we can choose $\alpha \in M''$ so that, working in $N''$, $\alpha < c$ and $\alpha \not \leq  a$ for all $a \in M'$.
	
\

 
\tikzset{
pattern size/.store in=\mcSize, 
pattern size = 5pt,
pattern thickness/.store in=\mcThickness, 
pattern thickness = 0.3pt,
pattern radius/.store in=\mcRadius, 
pattern radius = 1pt}
\makeatletter
\pgfutil@ifundefined{pgf@pattern@name@_zw92ku9gp}{
\pgfdeclarepatternformonly[\mcThickness,\mcSize]{_zw92ku9gp}
{\pgfqpoint{0pt}{0pt}}
{\pgfpoint{\mcSize+\mcThickness}{\mcSize+\mcThickness}}
{\pgfpoint{\mcSize}{\mcSize}}
{
\pgfsetcolor{\tikz@pattern@color}
\pgfsetlinewidth{\mcThickness}
\pgfpathmoveto{\pgfqpoint{0pt}{0pt}}
\pgfpathlineto{\pgfpoint{\mcSize+\mcThickness}{\mcSize+\mcThickness}}
\pgfusepath{stroke}
}}
\makeatother

 
\tikzset{
pattern size/.store in=\mcSize, 
pattern size = 5pt,
pattern thickness/.store in=\mcThickness, 
pattern thickness = 0.3pt,
pattern radius/.store in=\mcRadius, 
pattern radius = 1pt}
\makeatletter
\pgfutil@ifundefined{pgf@pattern@name@_pa50lnerb}{
\pgfdeclarepatternformonly[\mcThickness,\mcSize]{_pa50lnerb}
{\pgfqpoint{0pt}{-\mcThickness}}
{\pgfpoint{\mcSize}{\mcSize}}
{\pgfpoint{\mcSize}{\mcSize}}
{
\pgfsetcolor{\tikz@pattern@color}
\pgfsetlinewidth{\mcThickness}
\pgfpathmoveto{\pgfqpoint{0pt}{\mcSize}}
\pgfpathlineto{\pgfpoint{\mcSize+\mcThickness}{-\mcThickness}}
\pgfusepath{stroke}
}}
\makeatother

 
\tikzset{
pattern size/.store in=\mcSize, 
pattern size = 5pt,
pattern thickness/.store in=\mcThickness, 
pattern thickness = 0.3pt,
pattern radius/.store in=\mcRadius, 
pattern radius = 1pt}
\makeatletter
\pgfutil@ifundefined{pgf@pattern@name@_scmks7nlz}{
\pgfdeclarepatternformonly[\mcThickness,\mcSize]{_scmks7nlz}
{\pgfqpoint{0pt}{-\mcThickness}}
{\pgfpoint{\mcSize}{\mcSize}}
{\pgfpoint{\mcSize}{\mcSize}}
{
\pgfsetcolor{\tikz@pattern@color}
\pgfsetlinewidth{\mcThickness}
\pgfpathmoveto{\pgfqpoint{0pt}{\mcSize}}
\pgfpathlineto{\pgfpoint{\mcSize+\mcThickness}{-\mcThickness}}
\pgfusepath{stroke}
}}
\makeatother
\tikzset{every picture/.style={line width=0.75pt}} 

\begin{tikzpicture}[x=0.75pt,y=0.75pt,yscale=-1,xscale=1]

\draw  [color={rgb, 255:red, 155; green, 155; blue, 155 }  ,draw opacity=1 ] (155.89,217.5) -- (6,9) -- (306.07,9.21) -- cycle ;
\draw   (155.89,217.5) -- (139.07,116.65) -- (171.52,116.43) -- cycle ;
\draw  [pattern=_zw92ku9gp,pattern size=6pt,pattern thickness=0.75pt,pattern radius=0pt, pattern color={rgb, 255:red, 155; green, 155; blue, 155}] (203.02,28.93) -- (251.23,28.93) -- (155.89,217.5) -- (135.33,93.37) -- (185.4,93.37) -- cycle ;
\draw  [pattern=_pa50lnerb,pattern size=6pt,pattern thickness=0.75pt,pattern radius=0pt, pattern color={rgb, 255:red, 155; green, 155; blue, 155}] (156.2,217.5) -- (105.67,116.55) -- (171.52,116.55) -- cycle ;
\draw  [pattern=_scmks7nlz,pattern size=6pt,pattern thickness=0.75pt,pattern radius=0pt, pattern color={rgb, 255:red, 155; green, 155; blue, 155}] (189.57,29.86) -- (181.23,72.51) -- (129.31,72.51) -- (88.98,72.51) -- (62.56,29.86) -- cycle ;
\draw    (47.72,122.57) -- (84.71,60.32) ;
\draw [shift={(85.73,58.6)}, rotate = 120.72] [color={rgb, 255:red, 0; green, 0; blue, 0 }  ][line width=0.75]    (10.93,-3.29) .. controls (6.95,-1.4) and (3.31,-0.3) .. (0,0) .. controls (3.31,0.3) and (6.95,1.4) .. (10.93,3.29)   ;
\draw    (47.72,122.57) -- (119.45,132.5) ;
\draw [shift={(121.43,132.77)}, rotate = 187.88] [color={rgb, 255:red, 0; green, 0; blue, 0 }  ][line width=0.75]    (10.93,-3.29) .. controls (6.95,-1.4) and (3.31,-0.3) .. (0,0) .. controls (3.31,0.3) and (6.95,1.4) .. (10.93,3.29)   ;
\draw    (266.06,83.17) -- (228.69,53.81) ;
\draw [shift={(227.12,52.58)}, rotate = 38.16] [color={rgb, 255:red, 0; green, 0; blue, 0 }  ][line width=0.75]    (10.93,-3.29) .. controls (6.95,-1.4) and (3.31,-0.3) .. (0,0) .. controls (3.31,0.3) and (6.95,1.4) .. (10.93,3.29)   ;
\draw    (206.72,163.83) -- (167.4,127.64) ;
\draw [shift={(165.93,126.28)}, rotate = 42.63] [color={rgb, 255:red, 0; green, 0; blue, 0 }  ][line width=0.75]    (10.93,-3.29) .. controls (6.95,-1.4) and (3.31,-0.3) .. (0,0) .. controls (3.31,0.3) and (6.95,1.4) .. (10.93,3.29)   ;
\draw [color={rgb, 255:red, 155; green, 155; blue, 155 }  ,draw opacity=1 ]   (152.02,30.32) -- (156.2,217.5) ;
\draw [line width=2.25]    (153.88,116.55) -- (156.2,217.5) ;
\draw    (101.49,163.83) -- (151.98,146.85) ;
\draw [shift={(153.88,146.22)}, rotate = 161.41] [color={rgb, 255:red, 0; green, 0; blue, 0 }  ][line width=0.75]    (10.93,-3.29) .. controls (6.95,-1.4) and (3.31,-0.3) .. (0,0) .. controls (3.31,0.3) and (6.95,1.4) .. (10.93,3.29)   ;
\draw  [fill={rgb, 255:red, 0; green, 0; blue, 0 }  ,fill opacity=1 ] (150.63,59.99) .. controls (150.63,58.97) and (151.46,58.14) .. (152.49,58.14) .. controls (153.51,58.14) and (154.34,58.97) .. (154.34,59.99) .. controls (154.34,61.02) and (153.51,61.85) .. (152.49,61.85) .. controls (151.46,61.85) and (150.63,61.02) .. (150.63,59.99) -- cycle ;
\draw  [fill={rgb, 255:red, 0; green, 0; blue, 0 }  ,fill opacity=1 ] (151.56,85.02) .. controls (151.56,84) and (152.39,83.17) .. (153.41,83.17) .. controls (154.44,83.17) and (155.27,84) .. (155.27,85.02) .. controls (155.27,86.05) and (154.44,86.88) .. (153.41,86.88) .. controls (152.39,86.88) and (151.56,86.05) .. (151.56,85.02) -- cycle ;
\draw  [fill={rgb, 255:red, 0; green, 0; blue, 0 }  ,fill opacity=1 ] (101.49,56.75) .. controls (101.49,55.72) and (102.32,54.89) .. (103.35,54.89) .. controls (104.37,54.89) and (105.2,55.72) .. (105.2,56.75) .. controls (105.2,57.77) and (104.37,58.6) .. (103.35,58.6) .. controls (102.32,58.6) and (101.49,57.77) .. (101.49,56.75) -- cycle ;
\draw  [fill={rgb, 255:red, 0; green, 0; blue, 0 }  ,fill opacity=1 ] (194.06,102.64) .. controls (194.06,101.62) and (194.89,100.79) .. (195.92,100.79) .. controls (196.94,100.79) and (197.77,101.62) .. (197.77,102.64) .. controls (197.77,103.66) and (196.94,104.49) .. (195.92,104.49) .. controls (194.89,104.49) and (194.06,103.66) .. (194.06,102.64) -- cycle ;

\draw (270.53,8.03) node [anchor=north west][inner sep=0.75pt]   [align=left] {$\displaystyle N''$};
\draw (28.93,118.26) node [anchor=north west][inner sep=0.75pt]    {$N$};
\draw (262.82,80.27) node [anchor=north west][inner sep=0.75pt]    {$M'$};
\draw (191.73,161.78) node [anchor=north west][inner sep=0.75pt]    {$M=N\cap M'$};
\draw (82.86,161.15) node [anchor=north west][inner sep=0.75pt]    {$X$};
\draw (153.98,42.97) node [anchor=north west][inner sep=0.75pt]  [font=\large]  {$c$};
\draw (155.27,75.51) node [anchor=north west][inner sep=0.75pt]  [font=\normalsize]  {$\alpha $};
\draw (103.84,40.72) node [anchor=north west][inner sep=0.75pt]  [font=\large]  {$b$};
\draw (197.41,89.61) node [anchor=north west][inner sep=0.75pt]  [font=\large]  {$a$};
\end{tikzpicture}

	\begin{claim}
	$p' := \tp^{L'}(\alpha/M')$ is definable over $M$ in $T'$.
\end{claim}
\begin{proof}

We follow the analysis of formulas in $T$ in \cite[Theorem 3.16]{arXiv:2204.13790} (omitted in the final version of the paper \cite{CHERNIKOV_MENNEN_2025}). In $T$, every formula in any number of variables is equivalent to a Boolean combination of formulas $\varphi(x,y,z)$ in (at most) three singleton variables, which by quantifier elimination in $T$ we can assume to be quantifier-free. We also have  $R_{\varphi(\bar{x}) \land \psi(\bar{x})} \leftrightarrow^{T'} R_{\varphi(\bar{x})} \land  R_{\psi(\bar{x})} $ and $R_{\neg \varphi(\bar{x})} \leftrightarrow^{T'} \neg R_{\varphi(\bar{x})} $ for all $\varphi(\bar{x}), \psi( \bar{x}) \in L(N)$; and given a type, a definition with respect to it for a Boolean combination of formulas is given by a Boolean combination of their definitions. And for every  $\varphi(\bar{x}) \in L(M)$, $R_{\varphi(\bar{x})}(M'') = \varphi(M'')$ (evaluated in $N'' \models T$).
So it suffices to show the following for every atomic $\varphi(x,y_1,y_2) \in L$ with $x,y_1,y_2$ singletons:
	\begin{enumerate}
		\item for every fixed $b \in N$, the set $\{a \in M' : N'' \models \varphi(\alpha;a,b)\}$ is of the form $\psi(M')$ for some $\psi(y_1) \in L(N)$ (hence $L'(M)$-definable in $\widetilde{M}'$ via $R_{\psi(y_1)}(y_1)$);
		\item  the set $\{(a_1,a_2) \in (M')^2 : N'' \models  \varphi(\alpha,a_1,a_2)\}$ is of the form $\psi(M')$ for some $\psi(y_1,y_2) \in L(N)$.
	\end{enumerate} 

Up to logical equivalence in $T$,  restricting to formulas involving both $x$ and at least one $y_i$, up to permuting the variables $y_1,y_2$ and replacing by Boolean combinations, we only need to consider the following possibilities for $\varphi(x,y_1,y_2)$: $x = y_1$, $x = x \land y_1$ (i.e.~$x \leq y_1$), $x = y_1 \land y_2$, $x \land y_1 = y_1$ (i.e.~$x \geq y_1$), $x \land y_1 = x \land y_2$, 
$x \land y_1 = x \land y_1 \land y_2$ (i.e.~$(x \land y_1) \leq y_2$), $x \land y_1 \land y_2 = y_1 \land y_2$ (i.e.~$x \geq (y_1 \land y_2)$) (see the proof of \cite[Theorem 3.16]{arXiv:2204.13790}).

We have $N'' \models \alpha <c \land \alpha \not \leq a $ for all $a \in M'$ by the choice of $\alpha$, hence 
\begin{gather}
	\alpha \not \leq a \land b  \textrm{ for any } a \in M', b \in N \label{eq: def type reduct ex 1}
\end{gather}
(as otherwise $\alpha \leq a \land b \leq a \in M'$).

Also, for any $a \in M'$ we have $\alpha = \alpha \land c > a \land c$ (otherwise $\alpha = \alpha  \land c \leq a \land c \leq a$, contradicting \eqref{eq: def type reduct ex 1}). Then
\begin{gather}
	\alpha \land a = (\alpha \land c) \land a = a \land c \textrm{ for all }a \in M' \label{eq: def type reduct ex 2}
\end{gather}

%

We consider the non-trivial cases below.

\begin{enumerate}
\item $\varphi(x;y_1,y_2) = x \leq y_1$.

For any $a \in M'$, $\alpha \leq a$ does not hold by \eqref{eq: def type reduct ex 1}.
\item $\varphi(x;y_1,y_2) = (x = y_1 \land y_2)$.

For any $b \in N$ and $a \in M'$, $\alpha = a \land b$ does not hold (by \eqref{eq: def type reduct ex 1}). Similarly, for any $a_1, a_2 \in M'$,  $\alpha = a_1 \land a_2$ does not hold (by \eqref{eq: def type reduct ex 1} again).

\item $\varphi(x;y_1,y_2) = x \geq y_1$.

For any $a \in M'$, by \eqref{eq: def type reduct ex 2} we have $\alpha \geq a \Leftrightarrow \alpha \land a = a \Leftrightarrow a \land c = a \Leftrightarrow a \leq c$.

		\item $\varphi(x,y_1,y_2) = (x \land y_1 = x \land y_2)$.
\begin{enumerate}
		\item  Fix $b \in N$. If $b \land c \geq \alpha$, then $\alpha \land b = \alpha \land a$ does not hold for any $a \in M'$ (as then 
		$\alpha \land b = \alpha \land c \land b = \alpha > a \land c = \alpha \land a$ by \eqref{eq: def type reduct ex 2}). Otherwise $b \land c < \alpha$. Then $\alpha \land b = (\alpha \land c) \land b = b \land c$, so, using \eqref{eq: def type reduct ex 2}, $\alpha \land b = \alpha \land a \Leftrightarrow b \land c = a \land c \Leftrightarrow \widetilde{M}' \models R_{b \land c = y_1 \land c}(a)$.
		\item For any $a_1, a_2 \in M'$, by \eqref{eq: def type reduct ex 2} we have $\alpha \land a_1 = \alpha \land a_2 \Leftrightarrow a_1 \land c = a_2 \land c \Leftrightarrow \widetilde{M}' \models R_{y_1 \land c = y_2 \land c}(a_1,a_2)$.
	\end{enumerate}
	
	\item $\varphi(x;y_1,y_2) = (x \land y_1) \leq y_2$.
	Fixed $b \in N$. Using \eqref{eq: def type reduct ex 2}, for any $a \in M'$ we have $\alpha \land a \leq b \Leftrightarrow \alpha \land a \land b = \alpha \land b \Leftrightarrow a \land b \land c = a \land c \Leftrightarrow a \land c \leq b$.
	
	Similarly, for any $a_1,a_2 \in M'$ we get  $\alpha \land a_1 \leq a_2 \Leftrightarrow a_1 \land c \leq a_2 $.

	\item $\varphi(x;y_1,y_2) = (x \geq y_1 \land y_2)$.
	\begin{enumerate}
		\item  Fix $b \in N$. By \eqref{eq: def type reduct ex 2} we have $\widetilde{M}'' \models R_{x \geq y_1 \land b}(\alpha,a) \Leftrightarrow \alpha \geq (a \land b) \Leftrightarrow \alpha \land (a \land b) = a \land b \Leftrightarrow  a\land b \land c = a \land b   \Leftrightarrow c \geq a \land b \Leftrightarrow  \widetilde{M}' \models R_{c \geq (y_1 \land b)}(a)$.
\item For any $a_1, a_2 \in M'$, $a_1 \land a_2 \in M'$, hence $\alpha \geq a_1 \land a_2 \Leftrightarrow (a_1 \land a_2) \leq c \Leftrightarrow \widetilde{M}' \models R_{c \geq (y_1 \land y_2)}(a_1,a_2)$. \qedhere
	\end{enumerate}
\end{enumerate}
\end{proof}

By compactness, we can choose  $a \in M' \setminus M$ so that $X < a < c$, then $a < \alpha$ (by \eqref{eq: def type reduct ex 1}). So $(x > a) \in p'$. But the formula $x > a$ divides over $M$ in $T$, as by density of the tree, quantifier elimination and saturation we can choose $\leq$-incomparable $a_i \in M', i \in \omega$  with $a_i \equiv^L_M a$.
\end{example}
\begin{question}
	Is this possible for an $o$-minimal theory $T$?
\end{question}

\begin{fact}\label{fac: all meas over Sh exp are def}
Assume $T$ is NIP and $M \models T$.
	\begin{enumerate}
		\item  If all types in $S_x(M)$ are definable, then all measures in $\mathfrak{M}_x(M)$ are definable.
		\item All measures in $\mathfrak{M}_x^{L'}(M^{\Sh})$ are definable (in $T' = \Th_{L'}(M^{\Sh})$).
	\end{enumerate}
\end{fact}
\begin{proof}
	(1) By \cite[Theorem 2.7]{chernikov2014external}.
	
	(2) By (1) for $T'$, as  all types in $S_x^{L'}(M^{\Sh})$ are definable in $T'$ (by definition of Shelah's expansion).
\end{proof}

The following  key lemma in particular generalizes \cite[Lemma 3.18]{chernikov2014external} from types to measures, and from $M$ to arbitrary coherent models:
\begin{lemma}\label{lem: unique extension of def meas to Sh}
	Let $T$ be NIP, $M \models T$ and $T' := \Th_{L'}(M^{\Sh})$ (we are following the notation in Definition \ref{def: Shelah exp context}, in particular $\widetilde{M}'$ is a monster model for $T'$). Let $M_1 \preceq^{L'} \widetilde{M}'$ be  a small coherent model (i.e.~$(N,M) \preceq^{L_P} (N_1,M_1) \prec^{L_P} (N',M')$ for some $N_1$; so e.g.~$M_1 = M$).
	\begin{enumerate}
		\item 	Let $\mu \in \mathfrak{M}_x^{L}(M_1)$ be definable (in $T$). Then there exists a unique measure $\mu' \in \mathfrak{M}_x^{L'}(M_1)$ so that $\mu' \restriction_{L} = \mu$, and $\mu'$ is definable over $M_1$ in $T'$ (moreover, if $\mu$ is definable over $A \subseteq M_1$ in $T$, then $\mu'$ is also definable over $A$ in $T'$). In fact  for every $\psi(x) \in L'(M_1)$ we have:
		\begin{gather*}
			\mu'(\psi(x)) = \sup \left\{ \mu(\varphi(x)) :  \varphi(x) \in L(M_1), \varphi(M_1) \subseteq \psi(M_1) \right\} \\
			= \inf \left\{ \mu(\varphi'(x)) :  \varphi'(x) \in L(M_1), \psi(M_1) \subseteq \varphi'(M_1) \right\}.
		\end{gather*}
		
		\item If  $\bar{\mu} \in \mathfrak{M}_x^{L}(M')$ is the definable over $M_1$ extension of $\mu \in \mathfrak{M}_x^{L}(M_1)$ (in $T$), then  there exists a unique measure $\bar{\mu}' \in \mathfrak{M}_x^{L'}(M')$ so that $\bar{\mu}' \restriction_{L} = \bar{\mu}$. Moreover, $\bar{\mu}'$ is the $M_1$-definable extension of $\mu'$  in $T'$.
	\end{enumerate}
\end{lemma}
\begin{proof}

(1) Let $\mu \in \mathfrak{M}_x^{L}(M_1)$ be definable, and let $\bar{\mu} \in \mathfrak{M}_x^{L}(M')$ be the (unique, by Remark \ref{rem: unique def ext of definable measure}) global $M_1$-definable extension of $\mu$ (in $T$).
 Let $\bar{\mu}' \in \mathfrak{M}_x^{L'}(M')$ be arbitrary so that $\bar{\mu}' \restriction_{L} = \bar{\mu}$ (note that there exists one by Fact \ref{fac: classic measure ext} applied for the Boolean algebra of $L'(M')$-definable subsets of $M'$), and let $\mu' := \bar{\mu}'|_{M_1}$. Note that $(\bar{\mu}'|_{M}) \restriction_{L} = \mu' \restriction_{L} = \mu$.

Take any $\psi(x,y) \in L'$ and $b \in M_1$, and let $r := \mu'(\psi(x,b))$. We have that $\psi(x,y)$ is equivalent in $T'$ to $R_{\varphi(x,y;e)}(x,y)$ for some $\varphi(x,y;y') \in L$ and $e$ in $N^{y'}$, i.e.~$\psi(M') = \varphi(M';e)$.
By Fact \ref{fac: honest defs} (including the ``moreover'' part) applied to $(N_1,M_1)$ there is some $\theta(x,y;z) \in L$ and some $c \in (M')^z$ so that $\theta(M',c) \subseteq \varphi(M',e) = R_{\varphi(x,y;e)}(M')$ and no type $q \in S_{x,y}^{L}(M')$ invariant over $M_1$ in $T$ is consistent with $\{ \varphi(x,y;e) \setminus \theta(x,y;c)\} \cup \{P(x) \land P(y)\}$ in $\Th_{L_P}(N',M')$ (equivalently, with $R_{\varphi(x,y;e)}(x,y) \setminus \theta(x,y;c)$ in $T'$).

Let $\bar{p}' \in S(\bar{\mu}')$ be arbitrary. Then for $\bar{p}:= \bar{p}' \restriction_{L}$ we have $\bar{p} \in S(\bar{\mu})$ by definition of support. As $\bar{\mu}$ is definable over $M_1$ in $T$, it is in particular invariant over $M_1$ in $T$, hence by NIP every type $\bar{p} \in S(\bar{\mu})$ is also $M_1$-invariant in $T$ (Fact \ref{fac: types in sup of inv meas are inv NIP}). 
Suppose that $\left( R_{\varphi(x,y;e)}(x,b) \setminus \theta(x,b;c) \right) \in  \bar{p}'$, and let $a \models \bar{p}'$ in a bigger monster model $\widetilde{M}''$ (as in Definition \ref{def: Shelah exp context}). Then $(a,b) \models R_{\varphi(x,y;e)}(x,y) \setminus \theta(x,y;c)$ and $\tp^L(a,b/M')$ is invariant over $M_1$ in $T$ (as, working in $T$, $\tp^L(b/M')$ is obviously invariant over $M_1$ as $b \in M_1$ and $\tp^L(a/M'b)$ is invariant over $M_1$ as $a \models \bar{p}$, so we can apply left transitivity of invariance, see e.g.~\cite[Remark 2.20]{chernikov2012forking}). But this contradicts the choice of $\theta(x,y;c)$. As $\bar{p}' \in S(\bar{\mu}')$ was arbitrary, this shows that  $\bar{\mu}' \left( R_{\varphi(x,y;e)}(x,b) \setminus \theta(x,b;c) \right) = 0$, so $\bar{\mu}'(\theta(x,b,c)) = \bar{\mu}(\theta(x,b,c)) = r$. 

As the measure $\bar{\mu}$ is $M_1$-definable in $T$, for every $\varepsilon \in \mathbb{R}_{>0}$ there is some $\chi_{\varepsilon}(y,z) \in L(M_1)$ so that $\chi_{\varepsilon}(y,z) \in \tp^{L}((b,c)/M_1)$ and for every $(b',c') \in (M')^{y,z}$ with $\models \chi_{\varepsilon}(b',c')$ we have $\bar{\mu}(\theta(x,b',c')) \approx^{\varepsilon} \bar{\mu}(\theta(x,b,c))$, in particular $\bar{\mu}(\theta(x,b',c')) \geq r - \varepsilon$. So we have 
\begin{gather*}
	(N',M') \models \chi_{\varepsilon}(b,c) \land \forall x \in P \left( \theta(x,b,c) \rightarrow \varphi(x,b,e) \right).
\end{gather*}
As $(N,M) \preceq (N_1,M_1) \prec (N',M') $, $b \in M_1^y$ and $e \in N^z$,  there is some $c_{\varepsilon} \in M_1^z$ so that 
\begin{gather*}
	(N,M) \models \chi_{\varepsilon}(b,c_{\varepsilon}) \land \forall x \in P \left( \theta(x,b,c_{\varepsilon}) \rightarrow \varphi(x,b,e) \right).
\end{gather*}
That is, for every $\varepsilon >0$ there is some $c_{\varepsilon} \in M_1^{z}$ so that $\theta(M';b,c_{\varepsilon}) \subseteq \varphi(M';b,e)$ and $\mu(\theta(x,b,c_{\varepsilon})) = \bar{\mu}(\theta(x,b,c_{\varepsilon}))  \geq r - \varepsilon$. Applying the same argument with $\neg \varphi(x,y;e)$ instead of $\varphi(x,y;e)$, we find some $\theta'(x,y;z') \in L$ so that 
for every $\varepsilon >0$ there is some $c'_{\varepsilon} \in M_1^{z'}$ with $\varphi(M';b,e) \subseteq \theta'(M';b, c'_{\varepsilon})$ and $\mu(\theta'(x,b,c'_{\varepsilon})) = \bar{\mu}(\theta'(x,b,c'_{\varepsilon})) \leq r + \varepsilon$. But if now $\mu'' \in \mathfrak{M}_x^{L'}(M)$ is an arbitrary measure in $T'$ extending $\mu$, it in particular has to satisfy $\mu(\theta(x,b,c_{\varepsilon})) \leq \mu''(\varphi(x,b,e)) \leq \mu(\theta'(x,b,c'_{\varepsilon}))$	for every $\varepsilon > 0$, hence $\mu''(\psi(x,b)) =  r = \mu'(\psi(x,b))$. As $\psi(x,b) \in L'(M_1)$ was arbitrary, it follows that $\mu'$ is the unique measure  in $\mathfrak{M}_x^{L'}(M)$  extending $\mu$.

Finally, if $M_1 = M$ then all measures in $\mathfrak{M}^{L'}_x(M)$ are definable in $T'$ by Fact \ref{fac: all meas over Sh exp are def}, but this need not be the case over an arbitrary coherent $M_1$. In that case we let 
$$\chi'_{\varepsilon}(y) := \exists c_{\varepsilon} \left( \chi_{\varepsilon}(y,c_{\varepsilon}) \land  \forall x \left( \theta(x,y,c_{\varepsilon}) \rightarrow R_{\varphi(x,y;e)}(x,y) \rightarrow \theta'(x,y,c'_{\varepsilon})  \right) \right).$$

Then, by the above, $\chi'_{\varepsilon}(y) \in \tp^{L'}(b/M_1)$, and for any $b' \in (M_1)^y$ with $\widetilde{M}' \models \chi'_{\varepsilon}(b')$ we have $\mu'(\psi(x,b')) \approx^{\varepsilon} \mu'(\psi(x,b)) $. This shows that $\mu'$ is definable over $M_1$ in $T'$ (also shows the ``moreover'' part as $\chi'_{\varepsilon}$ only uses the same parameters in $M_1$ as $\chi_{\varepsilon}$).

(2) Applying (1) to $\bar{\mu}$ and $(M'\restriction_{L})^{\Sh}$ (as a structure in a language $L'' \supseteq L'$), we see that $\bar{\mu}$ has a unique extension to a measure in $\mathfrak{M}^{L''}_x((M')^{\Sh})$, hence in particular a unique extension to a measure $\bar{\mu}' \in \mathfrak{M}^{L'}_x(M')$ (note that each of the two different extensions to a subalgebra would extend to the full algebra by Fact \ref{fac: classic measure ext}).

Note that $\mu' \in \mathfrak{M}_x^{L'}(M_1)$ is definable in $T'$ by (1). Let $\bar{\mu}'' \in \mathfrak{M}_x^{L'}(M')$ be the unique $M_1$-definable extension of $\mu'$, in $T'$ (by Remark \ref{rem: unique def ext of definable measure}). We claim that $\bar{\mu}'' \restriction_{L} = \bar{\mu}$, hence  $\bar{\mu}'' = \bar{\mu}'$ as wanted. Assume not, then there is some $\varphi(x,y) \in L$ and $c \in (M')^y$ with $\bar{\mu}(\varphi(x,c)) = r  < s = \bar{\mu}''(\varphi(x,c)) $. Let $\varepsilon := (s-r)/3$. As $\bar{\mu}$ is definable over $M_1$ in $T$, there is some $\chi(y) \in L(M_1)$ so that $M' \models \chi(c)$ and $\bar{\mu}(\varphi(x,c')) \approx^{\varepsilon} \bar{\mu}(\varphi(x,c))$ for all $c' \in (M')^y$ with $M' \models \chi(c')$. Similarly, as $\bar{\mu}''$ is definable over $M_1$ in $T'$, there is some $\chi''(y) \in L'(M_1)$ so that $\widetilde{M}' \models \chi''(c)$ and $\bar{\mu}''(\varphi(x,c')) \approx^{\varepsilon} \bar{\mu}''(\varphi(x,c))$ for all $c' \in (M')^y$ with $\widetilde{M}' \models \chi(c')$.
As $\widetilde{M}' \models \chi(c) \land \chi''(c)$ and $M_1 \prec^{L'} \widetilde{M}'$ (by coherence, see Definition \ref{def: Shelah exp context}), there is some $c' \in M_1^y$ so that $\widetilde{M}' \models \chi(c') \land \chi''(c')$, hence $\bar{\mu}(\varphi(x,c')) \leq r + \varepsilon < s - \varepsilon \leq  \bar{\mu}''(\varphi(x,c')) $.
But $\bar{\mu}'' |_{M_1} = \mu'$, $\bar{\mu} |_{M_1} = \mu$ and $\mu' \restriction_{L} = \mu$, a contradiction.
\end{proof}

Using this, we establish a one-to-one correspondence between global generically stable measures in $T$ and in $\Th(M^{\Sh})$:
\begin{theorem}\label{thm: corresp for gs measures}
	Let $T$ be NIP, $M \models T$, $T' := \Th_{L'}\left(M^{\Sh} \right)$ and $\widetilde{M}'$ is a monster model for $T'$ (we are following the notation in Definition \ref{def: Shelah exp context}). Then the map $\bar{\mu}' \in \mathfrak{M}_x^{L'}(M') \mapsto \bar{\mu} := \bar{\mu}'\restriction_{L}  \in \mathfrak{M}_x^{L}(M')$ defines a bijection between global generically stable measures in $T'$ and global generically stable measures in $T$. Moreover, 
	$\bar{\mu}'$ is the unique measure in $\mathfrak{M}_x^{L'}(M')$ extending $\bar{\mu}$; and for any small set $A \subset M'$: 
	\begin{enumerate}
		\item if $\bar{\mu}'$ is fam over  $A$ in $T'$, then $\bar{\mu}$ is fam over $A$ in $T$;
		\item if $\bar{\mu}$ is generically stable over $A$ in $T$, then $\bar{\mu}'$ is generically stable over $A$ in $T'$ (see Remark \ref{rem: gs over a set but not fs}).
	\end{enumerate}
\end{theorem}
\begin{proof}

If $\bar{\mu}' \in \mathfrak{M}_x^{L'}(M')$ is fam over some small $A \subset M'$ in $T'$, then $\bar{\mu} := \bar{\mu}' \restriction_{L}$ is also fam over $A$ in $T$ by Proposition \ref{prop: properties of reducts of measures}(1). 

Conversely, assume $\bar{\mu} \in \mathfrak{M}_x^{L}(M')$ is generically stable in $T$, over some small set $A \subset M'$. Let $M_1$ with $A \subseteq M_1 \prec^{L'} \widetilde{M}'$ be a small coherent model (exists by Definition \ref{def: Shelah exp context}), in particular $\bar{\mu}$ is definable over $M_1$ in $T$. Then, by Lemma \ref{lem: unique extension of def meas to Sh}, $\bar{\mu}$ extends to a unique measure $\bar{\mu}' \in \mathfrak{M}_x^{L'}(M')$, and $\bar{\mu}'$ is definable over $A$ in $T'$ (by the ``moreover'' part of Lemma \ref{lem: unique extension of def meas to Sh}(1)).

Let $\mu := \bar{\mu}|_{M_1},  \mu' := \bar{\mu}'|_{M_1}$.
		Let $\bar{\mu}'' \in  \mathfrak{M}_x^{L'}(M')$ be an extension of $\mu'$ finitely satisfiable in $M_1$ (every measure over a model has a global coheir, by Fact \ref{fac: classic measure ext} applied to externally definable subsets of this model). Then $\bar{\mu}'' \restriction_{L}$ is also finitely satisfiable in $M_1$, hence is equal to $\bar{\mu}$ (by generic stability and Fact \ref{fac: basic props gen stab meas}, $\bar{\mu}$ is the unique $M_1$-invariant extension of $\mu$ in $T$). But then $\bar{\mu}'' = \bar{\mu}'$ by uniqueness, hence  $\bar{\mu}'$ is also finitely satisfiable in $M_1$, so generically stable over $M_1$ in $T'$ by NIP (Fact \ref{fac: NIP gen stab meas equivs}).
\end{proof}

\subsection{Externally definable fsg subgroups of definable groups}
\label{sec: Externally definable fsg subgroups of definable groups are definable}
First we show that every externally definable fsg subgroup of a definable group is already definable:

\begin{theorem}\label{thm: main for fsg subgroups}
	Assume $T$ is NIP, $M \models T$ and $G$ is an $M$-definable group. Let $\widetilde{M}' \succ^{L'} M^{\Sh}$ be a monster model for $T' := \Th(M^{\Sh})$.
	 Assume that $H'$ is an externally definable subgroup of $G(M)$ which is fsg in $T'$ (respectively,  a subgroup of $G(M')$ type-definable over $M$ in $T'$). Then $H'$ is already an  $L(M)$-definable subgroup of $G$ in $T$ (respectively, $H'(M')$ is  $L(M)$-type-definable in $T$ defined by a partial type of size at most $|T|$).
	
\end{theorem}
\begin{proof}

Let $H' \leq G$ be $M$-type-definable and fsg, both in the sense of the $L'$-theory $T' := \Th \left(M^{\Sh} \right)$ (we are following the notation in Definition \ref{def: Shelah exp context}).

By Proposition \ref{prop: fsg iff inv gen stab meas} (as $T'$ is NIP by Fact \ref{fac: Sh exp qe}), there is $\mu' \in \mathfrak{M}^{L'}_{H'}(M')$ generically stable over some small $M^{\Sh} \prec^{L'} M_1 \prec^{L'} \widetilde{M}'$ (in $T'$) and such that $g \cdot \mu' = \mu'$ for all $g \in H'(M')$. 
Note that $H'$ is both the left and the right stabilizer of $\mu'$ in $G$ (indeed, clearly $H' \subseteq \Stab_{G}(\mu')$, and by \cite[Remark 3.39]{chernikov2024definable} $\Stab_{G}(\mu')$ is the smallest  among all $M$-type-definable subgroup $K \leq G$ with $\mu' \in \mathfrak{M}_{K}(M')$).

Let $\mu := \mu'|_{L}$. Then $\mu \in \mathfrak{M}_{G}^L(M')$ and, by Proposition \ref{prop: properties of reducts of measures}, $\mu$ is generically stable over $M_1$, idempotent and  generically transitive in $T$.
 Let $ H := \Stab_G(\mu)$ (in $G(M')$). Then $H$ is an $L(M_1)$-type-definable subgroup of $G(M')$   and $\mu \in \mathfrak{M}_{H}^L(M')$.

We claim that $H(M')$ is the stabilizer of $\mu'$ in $G(M')$ (hence equal to $H'(M')$). Clearly $H'(M') \subseteq H(M')$. 

Conversely, assume $g \in H(M')$, hence $g \cdot \mu = \mu$.  As $\mu$ is in particular definable over $M_1$ in $T$, by Lemma \ref{lem: unique extension of def meas to Sh}, $\mu'$ is the unique measure in $\mathfrak{M}^{L'}_{G}(M')$ extending  $\mu$. And note that $(g \cdot \mu') \restriction_{L} = g \cdot (\mu' \restriction_{L}) = g \cdot \mu = \mu$, so $g \cdot \mu' \in \mathfrak{M}^{L'}_{G}(M') $ is also an extension of $\mu$. Hence $g \cdot \mu' = \mu'$, so $g \in H'(M')$.

Hence $H'(M') = H(M')$. Say $H'$ was defined by a partial $L'(M)$-type $\pi'(x)$ (of size $\leq |T'|$), and $H$ is defined by a partial $L(M_1)$-type $\pi(x)$ (of size $\leq |T|$), without loss of generality both closed under conjunctions.    Then, by saturation of $\widetilde{M}'$, for every formula $\varphi(x) \in \pi'$  there are formulas $\psi(x;y) \in L$ and $b \in M_1^y$ with $\psi(x;b) \in \pi$ and $\varphi'(x) \in \pi'$ so that $\varphi'(x) \vdash^{T'} \psi(x,b) \vdash^{T'} \varphi(x)$.  Hence $H'(M')$ is already defined by some $\pi'' \subseteq \pi'$ of size $\leq |T|$. 
And as $M^{\Sh} \prec^{L'} \widetilde{M}'$, there is some $b' \in M^{y}$ so that  $\varphi'(x) \vdash^{T'} \psi(x,b') \vdash^{T'} \varphi(x)$. This implies that $H'$ is  $L(M)$-type-definable (by a partial type of size $\leq |T|$); and if $H'$ was actually $L'(M)$-definable in $T'$, we get that it is $L(M)$-definable.
\end{proof}

\begin{remark}
The proof of Theorem \ref{thm: main for fsg subgroups} goes through with $M_1$ instead of $M$, for an arbitrary coherent model $M_1 \prec^{L'} \widetilde{M}'$ so that $(M,N) \prec^{L_P} (M_1,N_1) \prec^{L_{P}}(M',N')$ for some $N_0 \prec^L N_1 \prec^L N'$.
\end{remark}

\begin{remark}\label{rem: no fsg subgroups applies to reducts}
	We note that Theorem \ref{thm: main for fsg subgroups} also applies to arbitrary reducts of $M^{\Sh}$ expanding $M$. Indeed, assume $M'' = M^{\Sh} \restriction_{L''}$ for some $L \subseteq L'' \subseteq L' = L^{\Sh}$, $G$ an $M$-definable group and $H$ is a subgroup of $G$ definable in $M''$ and fsg in the sense of $\Th_{L''}(M'')$. Then, Fact  \ref{fac: group props preserved in Sh exp}(1), $H$ is also fsg in the sense of $\Th_{(L'')^{\Sh}}((M'')^{\Sh})$, and $(M'')^{\Sh} = (M')^{\Sh}$ (up to a correct identification of the languages), so Theorem \ref{thm: main for fsg subgroups} applies.
	\end{remark}

\begin{corollary}\label{cor: fsg subgroups in RCVF}
Let $\RCVF$ be the theory of real closed valued fields, $M'' \models \RCVF$ arbitrary. If $G$ is definable in the reduct $M$ of $M''$ to $\RCF$, and $H \leq G$ is definable in $M''$ and is fsg, then $H$ is already definable in $M$.
\end{corollary}
\begin{proof}
	Let $L''$ be the language of rings with an additional unary predicate for the valuation ring. Given any $L''$-model $M'' \models \RCVF$, let $M \models \RCF$ be its reduct to the language of rings $L$. As the valuation ring is convex, it is externally definable in $M$, hence $M''$ is a (proper) reduct of $M^{\Sh}$. Using Remark \ref{rem: no fsg subgroups applies to reducts}, Theorem \ref{thm: main for fsg subgroups} applies.
\end{proof}

\subsection{Fsg groups both type-definable and externally definable}
\label{sec: Type-definable and externally definable fsg groups are definable}
We also show that, for a sufficiently saturated $M$, every type-definable  fsg group in $M$ which is externally definable, is already definable.

\begin{theorem}\label{thm: type-def and ext def fsg implies def}
Let $T$ be an NIP $L$-theory, $G$ a type-definable over $M_0 \models T$ group  and $M \succ M_0$ is $|M_0|^{+}$-saturated. Assume $G$ is fsg in $T$ as a type-definable group, and $G(M)$ is definable in $M^{\Sh}$. Then $G$ is already a definable group in $M$ (hence $H=G$ is fsg as a definable group in $T'$).
\end{theorem}
\begin{proof}
	
	First we generalize Fact \ref{fac: group props preserved in Sh exp}(1) and show that $G$ is still fsg as a type-definable group in $T' := \Th_{L'}(M^{\Sh})$. 	The proof as for definable groups in \cite[Theorem 3.19]{chernikov2014external}  could be adapted, but we give a proof with measures instead.

		As $G$ is type-definable over $M_0$ in $T$ and $M \succ^L M_0$ is $|M_0|^+$-saturated, by Proposition \ref{prop: fsg iff inv gen stab meas} there is a global measure $\bar{\mu} \in \mathfrak{M}_G^L(M')$ which is $G(M')$-left-invariant and generically stable over $M$ in $T$. Let $\mu := \bar{\mu}\restriction_{M}$.

		By Lemma \ref {lem: unique extension of def meas to Sh}, there exists a unique measure $\bar{\mu}' \in \mathfrak{M}_x^{L'}(M')$ with $\bar{\mu}' \restriction_{L} = \bar{\mu}$, clearly $\bar{\mu}' \in \mathfrak{M}_{G}^{L'}(M')$. By Theorem \ref{thm: corresp for gs measures}, $\bar{\mu}'$ is generically stable over $M$ in $T'$.

		By the ``moreover'' part of Lemma \ref{lem: unique extension of def meas to Sh}(2), $\bar{\mu}'$ is definable over $M$ in $T'$. Let $\mu' := \bar{\mu}'|_{M}$.
		Let $\bar{\mu}'' \in  \mathfrak{M}_x^{L'}(M')$ be an extension of $\mu'$ finitely satisfiable in $M$ (every measure over a model has a global coheir, by Fact \ref{fac: classic measure ext} applied to externally definable subsets of this model). Then $\bar{\mu}'' \restriction_{L}$ is also finitely satisfiable in $M$, hence is equal to $\bar{\mu}$ (by Fact \ref{fac: basic props gen stab meas}, $\bar{\mu}$ is the unique $M$-invariant extension of $\mu$ in $T$). But then $\bar{\mu}'' = \bar{\mu}'$ by uniqueness, hence  $\bar{\mu}'$ is also finitely satisfiable in $M$, so generically stable over $M$ in $T'$.
 And $\bar{\mu}'$ is also $G(M')$-invariant: given $g \in G(M')$, $g \cdot \bar{\mu} = \bar{\mu} $ and $(g \cdot \bar{\mu}') \restriction_{L} = g \cdot (\bar{\mu}'\restriction_{L}) = g \cdot \bar{\mu} = \bar{\mu}$, so $g \cdot \bar{\mu}' = \bar{\mu}'$ by uniqueness. By Proposition \ref{prop: fsg iff inv gen stab meas}, $\bar{\mu}'$ witnesses that the type-definable group $G$ is fsg in $T'$.

	Now  $G$ is type-definable over $M_0$ is $T$ by some partial type $\{\psi_{\alpha}(x) : \alpha < |M_0|\}$ with $\psi_{\alpha} \in L(M_0)$, without loss of generality closed under finite conjunctions. And $G(M)$ is also definable in $M^{\Sh}$, that is 	there is some formula 
	$H(x) \in L'(M)$ so that $G(M) = H(M)$.
	For every $\alpha$, as $\psi_{\alpha}(M) \supseteq H(M)$, $H(M') = \theta(M',c)$ for some $\theta(x,y) \in L$ and $c \in N^y$,  and $(N',M') \succ (N,M)$, we still have $\psi_{\alpha}(M') \supseteq H(M')$, hence $G(M') \supseteq H(M')$, but in general this could be a proper containment (e.g.~take $G$ as the type-definable subgroup of infinitesimals on a unit circle group in $T =\RCF$).
	
	We claim that still $\bar{\mu}'$ is supported on $H(x)$, that is $\bar{\mu}' \in \mathfrak{M}_{H}^{L'}(M')$  in $T'$.
	Indeed,	by Lemma \ref{lem: unique extension of def meas to Sh}(1), 
	$$\mu'(H(x)) = \inf \left\{ \mu(U(x)) :  U(x) \in L(M), H(M) \subseteq U(M) \right\}.$$
	As $\mu \in \mathfrak{M}^{L}_G(M)$, we have $\mu(\psi_{\alpha}(x)) = 1$ for all $\alpha$.
	But if $U(x) \in L(M)$ and $U(M) \supseteq H(M) = G(M)$, then by $|M_0|^+$-saturation of $M$ we have $U(M) \supseteq \psi_{\alpha}(M)$ for some $\alpha$ (using closure under conjunctions), so $\mu(U(x)) = 1$. We conclude $\mu'(H(x)) = 1$.
	
	So we have $H(M') \leq G(M')$, and we claim $H(M') = G(M')$. Assume not, then $H(M') \cap g \cdot H(M') = \emptyset$ for some $g \in G(M')$. As $\bar{\mu}' \in \mathfrak{M}^{L'}_{H}(M')$, we have $S(\bar{\mu}') \subseteq S^{L'}_{H}(M'), S(g \cdot \bar{\mu}') \subseteq S^{L'}_{g \cdot H}(M')$, so $S(\bar{\mu}') \cap  S(g \cdot \bar{\mu}') = \emptyset$, so $\bar{\mu}' \neq g \cdot \bar{\mu}'$ --- a contradiction as $\bar{\mu}'$ is $G(M')$-invariant.
	
 Finally, as $H(M') = G(M')$, saturation of the $L_P$-structure $(N', M')$ implies $G(M') = \bigcap_{\alpha \in s}\psi_{\alpha}(M')$ for some finite $s$, so $G$ is $L(M_0)$-definable in $T$.
\end{proof}

The following proposition gives some criteria that could be used to see that certain externally definable groups are not fsg in $\Th(M^{\Sh})$:

\begin{proposition}\label{prop: ext def V def fsg implies definable} ($T$ NIP) Suppose $M_0 \prec M$ and $M$ is $|M_0|^+$-saturated. 
\begin{enumerate}
	\item Suppose $G$ is type-definable over $M_0$ and $G(M) = H(M)$ for $H$ a definable group in $M^{\Sh}$.  If $H$ is fsg as a definable group in $T' = \Th(M^{\Sh})$, then $G$ is fsg as a type-definable group in $T$ (hence again $G$ is already definable).
	\item Similarly, suppose $G$ is $\bigvee$-definable over $M_0$ and $G(M) = H(M)$ for $H$ a definable group in $M^{\Sh}$.  If $H$ is fsg as a definable group in $T'$, then $G$ is definable (and fsg) in $M$.
\end{enumerate}
\end{proposition}
\begin{proof}
	(1) Let $\bar{\mu}' \in \mathfrak{M}_{H}^{L'}(M')$ be a global left-$H(M')$-invariant measure generically stable over $M$ in $T'$ (Fact \ref{fac: fsg groups basic props}(2)). Let $\bar{\mu} := \bar{\mu}' \restriction_{L}$, then $\bar{\mu}$ is generically stable over $M$ in $T$ (Proposition \ref{prop: properties of reducts of measures}), hence generically stable over some $M_1$ with $M_0 \leq M_1 \prec M$ and $|M_1| \leq |T|$. Also $\bar{\mu} \in \mathfrak{M}_{G}^{L}(M')$ (if $G(x) = \bigcap_{i \in I}{\psi_{i}(x)}$ with $\psi_{i}(x) \in L(M_0)$, then $H(M) = G(M)$ implies $H(x) \vdash^{T'} \psi_i(x)$ for all $i \in I$, so $\bar{\mu}(\psi_i(x)) = \bar{\mu}'(\psi_i(x)) \geq \bar{\mu}'(H(x)) = 1$) and $\bar{\mu}$ is left-$H(M')$-invariant. Let $\mu := \bar{\mu}\restriction_{L}$, then $\mu$ is generically stable over a small model $M_1 \prec^L M$, $\mu \in \mathfrak{M}_{G}^L(M)$ and $\mu$ is left-$G(M)$-invariant (as $G(M) = H(M)$) --- witnessing that $G$ is fsg as a type-definable group in $T$.

(2) Similarly, suppose $G$ is $\bigvee$-definable over $M_0$, $M$ is $|M_0|^+$-saturated and $G(M) = H(M)$ for $H$ a definable group in $M^{\Sh}$.  If $H$ is fsg as a definable group in $T'$, then $G$ is definable (and fsg).

Indeed, let again $\bar{\mu}' \in \mathfrak{M}_{H}^{L'}(M')$ be a global left-$H(M')$-invariant measure generically stable over $M$ in $T'$. Say $G(x) = \bigvee_{i \in I}{\psi_{i}(x)}$ with $\psi_{i}(x) \in L(M_0)$, and without loss of generality $\{\psi_{i}(x) : i \in I\}$ is closed under disjunctions. Then $\psi_i(M) \subseteq H(M)$ for all $i \in I$ and $G(M') \leq H(M')$, as $H(M) = G(M)$. 
By Lemma \ref{lem: unique extension of def meas to Sh}(1) we have
	$$1 = \bar{\mu}'(H(x)) = \sup \left\{ \bar{\mu}'(U(x)) :  U(x) \in L(M), U(M) \subseteq H(M) \right\}.$$
If $U(x) \in L(M)$ and $U(M) \subseteq H(M) = G(M)$, then by $|M_0|^+$-saturation of $M$ we have $U(M) \subseteq \psi_{i}(M)$ for some $i \in I$ (using closure under disjunctions), so $\sup \left\{ \bar{\mu}'(\psi_i(x)) :  i \in I \right\} = 1$. Assume $G(M') \lneq H(M')$, then $G(M') \cap h \cdot G(M') = \emptyset$ for some $h \in H(M')$. We have $\bar{\mu}'(\psi_i(x)) > 1/2$ for some $i \in I$, so also $\bar{\mu}'(h \cdot \psi_i(x)) > 1/2$ by $H(M')$-invariance of $\bar{\mu}'$, contradicting these sets being disjoint. So $H(M') = G(M')$, hence $G(M') = \bigcup_{i \in I_0}\psi_{i}(M')$ for some finite $I_0 \subseteq I$ (by saturation of the $L_P$-structure $(N', M')$), so $G$ is $L(M_0)$-definable in $T$.
%
%
\end{proof}

\begin{example}\label{ex: infinitesim not fsg}
(1) Let $T = \RCF$, $M \models T$ is $|T|^{+}$-saturated, $G(M) = ([0,1)(M), + \mod 1)$ is the definable (non-standard) unit circle group, and $H(M) \leq G(M)$ is the subgroup of infinitesimals, then $H(M) = G^{00}(M)$ is definable in $M^{\Sh}$ (as it is convex with respect to the circular order) and a type-definable over $M_0 := \mathbb{R}$ subgroup of bounded index of $G$ in $M$, but not definable in $M$ by $o$-minimality. It follows that $H$ is not fsg in $T' := \Th_{L'}(M^{\Sh})$ (and $G^{00}$ is not fsg as a type-definable group in $T$), either by Theorem \ref{thm: main for fsg subgroups} or Proposition \ref{prop: ext def V def fsg implies definable}(1).

(2)	Let $M = (K, \Gamma, k, v) \models \ACVF$ with valuation ring $\mathcal{O}_v = \{a \in K : v(a) \geq 0\}$, and let $v'$ be a coarsening of $v$, i.e.~a valuation $v'$ on $K$ with $\mathcal{O}_{v'} \supseteq \mathcal{O}_v$. Then there is a convex subgroup $\Gamma' \leq \Gamma$ such that $v'K \cong vK/\Gamma'$. Then $\Gamma'$ is definable in $M^{\Sh}$, and hence $\mathcal{O}_{v'} = \{ a \in K : a \in \mathcal{O}_{v} \lor v(a) \in \Gamma' \}$ is also definable in $M^{\Sh}$ (see e.g.~\cite[Example 2.2]{zbMATH07845700}) and a subgroup of $(K,+)$. It follows that $\mathcal{O}_{v'}$ is not fsg in $M^{\Sh}$, by Theorem \ref{thm: main for fsg subgroups} (or by Proposition \ref{prop: ext def V def fsg implies definable} if $\Gamma'$ is either type-definable or $\bigvee$-definable in $M$). 
\end{example}

\subsection{Generically stable measures as a hyper-definable set}\label{sec: Generically stable measures as a hyper-definable set}
An important point in \cite{hrushovski2016non} and subsequent work is that the set of global generically stable types in NIP theories can be viewed as a pro-definable (i.e.~a type-definable set of length $|T|$ tuples)  set in $\cU^{\eq}$. Here we consider an analogous presentation for the set of global generically stable measures in an NIP theory, but as a \emph{hyper-definable}  set  of  length $|T|$ tuples. In particular this is convenient for carrying out some compactness arguments with generically stable measures (which do not form a closed subset of the space of all global Keisler measures), e.g.~in Section \ref{sec: abelian ext def subgroups}. Recall from Fact \ref{fac: NIP gen stab meas equivs}:
\begin{definition}
	A Keisler measure $\mu \in \mathfrak{M}_x(\cU)$ is \emph{fam} (finitely approximated measure) if for every $\varphi(x,y) \in L$ and $\varepsilon > 0$ there is some $n$ and $\bar{a}^{\mu}_{\varphi, \varepsilon} = (a_1, \ldots, a_n) \in (\cU^x)^n$ so that $\bar{a}^{\mu}_{\varphi, \varepsilon}$ is an \emph{$\varepsilon$-approximation for $\mu$ on $\varphi(x,y)$}, i.e.~for every $b \in \cU^y$ we have 
		\begin{gather*}
		 \left \lvert \Av(\bar{a}^{\mu}_{\varphi, \varepsilon}; \varphi(x,b))	- \mu \left( \varphi(x,b) \right) \right \rvert \leq \varepsilon.
		\end{gather*}
	For $A \subseteq \cU$, we say that $\mu$ is \emph{fam over $A$} if $\bar{a}^{\mu}_{\varphi, \varepsilon}$ can be chosen in $A$ for all $\varphi, \varepsilon$.
\end{definition}

\begin{fact}\label{fac: unif e approx for gs measures in NIP}

	\begin{enumerate}
		\item If $T$ is NIP and $M \prec \cU$, then $\mu \in \mathfrak{M}_x(\cU)$ is generically stable (over $M$) if and only if it is fam (over $M$).
		\item For every NIP formula $\varphi(x,y) \in L$ (in $T$) and $\varepsilon > 0$ there is some $n_{\varphi, \varepsilon} \in \omega$  depending only on $\varphi$ and $\varepsilon$ so that: for any $\mu \in  \mathfrak{M}_x(\cU)$, if $\mu$ is fam over $M$ then there is an $\varepsilon$-approximation $\bar{a}^{\mu}_{\varphi, \varepsilon} = (a_1, \ldots, a_n) \in M^n$ for $\mu$ on $\varphi(x,y)$ with $n \leq n_{\varphi, \varepsilon}$ (see e.g.~\cite[Proposition 2.13]{zbMATH06966829} for an explanation of the uniform bound).
\end{enumerate}
\end{fact}

\begin{remark}\label{rem: eps approx precise size}
	In Fact \ref{fac: unif e approx for gs measures in NIP}(2), taking $n'_{\varphi, \varepsilon} := (n_{\varphi, \varepsilon})! $, we can
choose an $\varepsilon$-approximation of size \emph{exactly} $n'_{\varphi, \varepsilon}$ (indeed, given a tuple $\bar{a}$ of size $n \leq n_{\varphi, \varepsilon}$ which is an $\varepsilon$-approximation, we can take a tuple of length $n'_{\varphi, \varepsilon}$ which is a concatenation of (integer many) copies of $\bar{a}$ --- it still gives an $\varepsilon$-approximation.
\end{remark}

Assume that $T$ is NIP, and let $x$ be a tuple of variables of fixed sort.

\begin{definition}\label{def: a^mu coding measure}
	Given any generically stable measure $\mu \in \mathfrak{M}_x(\cU)$, we associate to it a tuple $\bar{a}^{\mu} = \left(\bar{a}^{\mu}_{\varphi, 1/k} : \varphi(x,y) \in L, k \in \mathbb{N} \right)$ of length $|T|$ so that $\bar{a}^{\mu}_{\varphi, 1/k} \in (\cU^{x})^{n_{\varphi, 1/k}}$ is an $1/k$-approximation for $\mu$ on $\varphi(x,y)$ of length $n_{\varphi, 1/k}$ (exists by Remark \ref{rem: eps approx precise size}).
\end{definition}
 We will establish a bijection between the set of global generically stable measures (in $x$) and a particular hyperdefinable set of such tuples.

\begin{definition}
	We let $X' = X'_x$ be the set of all tuples 
	$$\bar{a} = \left(\bar{a}_{\varphi, 1/k} : \varphi(x,y) \in L, k \in \mathbb{N} \right)$$ 
	in $\cU$ satisfying: 
\begin{gather*}
	\bigwedge_{\varphi(x,y) \in L} \bigwedge_{k_1, k_2 \geq k \in \mathbb{N} } \forall b \left( \left\lvert\Av(\bar{a}_{\varphi, 1/k_1}; \varphi(x,b))  - \Av(\bar{a}_{\varphi, 1/k_2}; \varphi(x,b))  \right \rvert \leq 2/k \right).
\end{gather*}
\end{definition}
Then $X'$ is $\emptyset$-type-definable, and $\bar{a}^{\mu} \in X'$ for any generically stable $\mu  \in \mathfrak{M}_x(\cU)$. By definition, for every $\bar{a} \in X'$, $\varphi(x,y) \in L$ and $b \in \cU^y$, the limit $\lim_{k \to \infty} \Av(\bar{a}_{\varphi, 1/k}; \varphi(x,b))$ exists.

\begin{definition}\label{def: X_x}
	Let $X = X_x$ be the set of tuples $\bar{a} \in X'_x$ satisfying 
\begin{gather*}
	\bigwedge_{\varphi_1(x,y_1), \varphi_2(x,y_2) \in L} \bigwedge_{k \in \mathbb{N}} \forall b_1 \forall b_2 \Big( \neg \exists x( \varphi_1(x,b_1) \land \varphi_2(x,b_2)) \rightarrow \\
	\big  ( \lvert \Av \left( \bar{a}_{\varphi_1 \lor \varphi_2, 1/k}; \varphi_1(x,b_1) \lor \varphi_2(x,b_2)  \right) -\\
\Av(\bar{a}_{\varphi_1,  1/k}; \varphi_1(x,b_1)) + \Av(\bar{a}_{\varphi_2,  1/k}; \varphi_2(x,b_2))  ) \rvert \leq 3/k \big) \Big).
\end{gather*}
Then  $X$ is $\emptyset$-type-definable, and for every $\bar{a} \in X$, defining $\mu_{\bar{a}}$ via 
$$\mu_{\bar{a}}(\varphi(x,b)) := \lim_{k \to \infty} \Av(\bar{a}_{\varphi, 1/k}; \varphi(x,b))$$
 for all $\varphi(x,b) \in L(\cU)$, it satisfies $\mu_{\bar{a}}(\varphi_1(x,b_1) \lor \varphi_2(x,b_2)) = \mu_{\bar{a}}(\varphi_1(x,b_1)) + \mu_{\bar{a}}(\varphi_2(x,b_2))$ for all disjoint $\varphi_i(x,b_i) \in L(\cU)$. Hence $\mu_{\bar{a}}$ is a Keisler measure. 
\end{definition}

\begin{remark}\label{rem: 2/k approx}
	For an arbitrary $\bar{a} \in X$, definition of $X'$ guarantees that $\bar{a}_{\varphi, 1/k}$ is a $2/k$-approximation for $\mu_{\bar{a}}$ on $\varphi(x,y)$ (but not necessarily a $1/k$-approximation).
\end{remark}
Hence $\mu_{\bar{a}}$ is generically stable (by Fact \ref{fac: NIP gen stab meas equivs}(1)). We also have $\bar{a}^{\mu} \in X_x$ for every generically stable measure $\mu_x$, using that $\bar{a}^{\mu}_{\varphi, 1/k}$ is a $1/k$-approximation for $\mu$. 

\begin{definition}
For two tuples $\bar{a}, \bar{a}'$, we define $E = E_x$ to be the set of pairs $(\bar{a},\bar{a}')$ satisfying  
\begin{gather*}
	\bigwedge_{\varphi(x,y) \in L} \bigwedge_{k \in \mathbb{N}} \forall b \left( \left \lvert \Av \left(\bar{a}_{\varphi, 1/k}; \varphi(x,b) \right) - \Av \left(\bar{a}'_{\varphi, 1/k}; \varphi(x,b) \right) \right \rvert \leq 4/k \right).
\end{gather*}
\end{definition}
Then $E$ is an $\emptyset$-type-definable equivalence relation on $X$ (but not an intersection of definable equivalence relations), and  for any $\bar{a},\bar{a}' \in X$, $E(\bar{a}, \bar{a}')$ holds if and only if $\mu_{\bar{a}} = \mu_{\bar{a}'}$. Indeed, if $E(\bar{a}, \bar{a}')$ holds, then for any $\varphi(x,b)$ we have 
\begin{gather*}
\mu_{\bar{a}}(\varphi(x,b)) = \lim_{k \to \infty} \Av(\bar{a}_{\varphi, 1/k}; \varphi(x,b)) \\
= \lim_{k \to \infty} \Av(\bar{a}'_{\varphi, 1/k}; \varphi(x,b)) = \mu_{\bar{a}'}(\varphi(x,b)).	
\end{gather*}
Conversely, if $\mu_{\bar{a}} = \mu_{\bar{a}'}$, then for any $\varphi$ and $k$ we have that $\bar{a}_{\varphi, 1/k}$ is a $2/k$-approximation for $\mu_{\bar{a}}$  and $\bar{a}'_{\varphi, 1/k}$ is a $2/k$-approximation for $\mu_{\bar{a}'}$  on $\varphi(x,y)$ (by Remark \ref{rem: 2/k approx}), hence $E(\bar{a}, \bar{a}')$ holds. 

We have thus established the following:
\begin{proposition}
	Consider the hyperdefinable set $\widetilde{\mathcal{M}}_x := X_x/E_x$. For a generically stable measure $\mu \in \mathfrak{M}_{x}(\cU)$, we consider the hyperimaginary $[\mu] := \bar{a}^{\mu}/E \in \widetilde{\mathcal{M}}_x$. Then the map $\mu \mapsto [\mu]$ is a bijection from the space of global generically stable measures to $\widetilde{\mathcal{M}}_x$ (with the inverse $\bar{a}/E \mapsto \mu_{\bar{a}}$).
\end{proposition}

\begin{remark}
	We can think of $\Cb(\mu) := \dcl^{\heq}([\mu])$ as a canonical base for $\mu$. In particular we have $\sigma(\mu) = \mu$ if and only if $\sigma(\Cb(\mu)) = \Cb(\mu)$ for any automorphism $\sigma$ of $\cU^{\heq}$; and $\Cb(\mu)$ is the smallest $\dcl^{\heq}$-closed set over which $\mu$ is definable.
\end{remark}

 Next we observe that some of the important operations on generically stable measures correspond to type-definable maps on $\widetilde{\mathcal{M}}$.

\begin{fact}\label{fac: approx for product} (see \cite[Proposition 2.13]{zbMATH07456275})
	Let $\mu_x$ and $\nu_y$ be global generically stable measures and $\varphi(x,y;z) \in L$. Assume that $\bar{a} = (a_1, \ldots, a_m)$ is an $\varepsilon$-approximation for $\mu_x$ on $\varphi_1(x;y,z) :=  \varphi(x,y,z)$, and $\bar{b} = (b_1, \ldots, b_n)$ is an $\varepsilon$-approximation for $\nu_y$ on $\varphi_2(y;x,z) := \varphi(x,y,z)$, then $\bar{a} \times \bar{b} := \left( (a_i, b_j) : (i,j) \in [m] \times [n] \right)$ is a $2 \varepsilon$-approximation for $\mu_x \otimes \nu_y$ on $\varphi(x,y;z)$.
\end{fact}

\begin{proposition}\label{prop: tensor is type-def}
	The map $([\mu], [\nu]) \in \widetilde{\mathcal{M}}_x \times \widetilde{\mathcal{M}}_y \mapsto [\mu \otimes \nu] \in  \widetilde{\mathcal{M}}_{xy}$ is $\emptyset$-type-definable.
\end{proposition}
\begin{proof}

Indeed, let $\Gamma$ be the $\emptyset$-type-definable set of triples $(\bar{a}', \bar{a}'', \bar{a}''') \in X_x \times X_{y} \times X_{x,y}$ satisfying
\begin{gather*}
	\bigwedge_{\varphi(x,y;z) \in L} \bigwedge_{k \in \mathbb{N}} \forall b \Big( \Big \lvert \Av \left( \bar{a}'''_{\varphi(x,y;z), 1/k}; \varphi(x,y;b) \right) - \\
	  \Av \left(  \bar{a}'_{\varphi_1(x;y,z), 1/k} \times \bar{a}''_{\varphi_2(y;x,z), 1/k} ; \varphi(x,y;b) \right) \Big \rvert \leq  6/k \Big)
\end{gather*}
(using the notation in Fact \ref{fac: approx for product}).
Then $\models \Gamma(\bar{a}', \bar{a}'', \bar{a}''')$ if and only if $\mu_{\bar{a}'} \otimes \mu_{\bar{a}''} = \mu_{\bar{a}'''}$, using Remark \ref{rem: 2/k approx} and  Fact \ref{fac: approx for product} with $\varepsilon := 2/k$, and $\Gamma$ is $E$-invariant on each of its coordinates.
\end{proof}

Similarly we have:
\begin{proposition}\label{prop: push-forward is type-def}
	Let $f: \cU^x \to \cU^y$ be an $A$-definable function. Then $[\mu] \in \widetilde{\mathcal{M}}_x \mapsto [f_{\ast} \mu] \in \widetilde{\mathcal{M}}_y$ is an $A$-type-definable function.
\end{proposition}
\begin{proof}

Given $\varphi(y;z) \in L$, let $\psi(x;z) := \varphi(f(x);z) \in L(A)$. Recall that $(f_{\ast} \mu)_y$ is the unique measure satisfying $f_{\ast} \mu (\varphi(y;b)) = \mu_x( \psi(x;b) )$ for all $\varphi$ and $b \in \cU$. Then the graph is $A$-type-definable via 
\begin{gather*}
	\bigwedge_{\varphi(y,z) \in L} \bigwedge_{k \in \mathbb{N}} \forall b \Big( \left \lvert \Av \left(\bar{a}'_{\psi(x;z), 1/k} ; \psi(x,b) \right) - \Av \left(\bar{a}''_{\varphi(y;z), 1/k} ; \varphi(y,b) \right) \right \rvert \leq 4/k\Big).
\end{gather*}
\end{proof}

\begin{definition}\label{def: X_G}
	For a partial type $\pi(x)$ over a small set of parameters $A$, we let $\widetilde{\mathcal{M}}_{\pi}$ be the set of all $ [\mu] \in \widetilde{\mathcal{M}}_x$ for $\mu$  a generically stable measure in $\mathfrak{M}_{\pi}(\cU)$ (and similarly $X_{\pi}$ the set of all $\bar{a} \in X_x$ so that $\mu_{\bar{a}} \in \mathfrak{M}_{\pi}(\cU)$), i.e.~$\pi \subseteq p$ for every $p \in S(\mu)$ (equivalently, $\mu(\varphi(x))=1$ for all $\varphi(x) \in \pi$). Then $\widetilde{\mathcal{M}}_{\pi}$ is an $A$-type-definable subset of $\widetilde{\mathcal{M}}_{x}$ (and $X_{\pi}$ is an $A$-type-definable subset of $X_x$), via the ($E$-invariant) set of all $\bar{a} \in  X$ satisfying
	\begin{gather*}
		\bigwedge_{\varphi(x,b) \in \pi} \bigwedge_{k \in \mathbb{N}} \Av\left( \bar{a}_{\varphi, 1/k}; \varphi(x,b) \right) \geq 1 - 2/k.
	\end{gather*}
	In particular, if $G$ is a type-definable group, $\widetilde{\mathcal{M}}_{G}$ represents the set of all generically stable measures supported on $G$.
\end{definition}

\begin{remark}\label{rem: meas code rep inside the partial type}
	For every $[\mu] \in \mathfrak{M}_{\pi}(\cU)$, we have $\mu = \mu_{\bar{a}}$ for some representative $\bar{a} \in X_{\pi} \cap \pi(\cU)^{\bar{y}}$ (by Lemma \ref{lem: generically stable localizing}).
\end{remark}

\begin{proposition}
	Let $G$ be an $A$-type-definable group. Then the set  of all $[\mu] \in \widetilde{\mathcal{M}}_{G}$ with $\mu \in \mathfrak{M}_G(\cU)$ idempotent (generically transitive) is $A$-type-definable.
\end{proposition}
\begin{proof}
	By Propositions \ref{prop: tensor is type-def} and \ref{prop: push-forward is type-def}, as $\mu$ is idempotent if and only if $\mu = f_{\ast}(\mu \otimes \mu)$ for $f(x,y) = x \cdot y$, and $\mu$ is generically transitive if and only if $\mu \otimes \mu = g_{\ast}(\mu \otimes \mu)$ for $g(x,y) = (x, x \cdot y)$.
\end{proof}

\begin{lemma}\label{lem: unif defining stab of mu}
	Let $G = G(x)$ be an $A$-type-definable group. Then there is an $A$-type-definable set $\St_{G} \subseteq G \times \widetilde{M}_G$ so that for any $[\mu] \in \widetilde{M}_G$ we have $\St_{G}(\cU;[\mu]) = \Stab_{G}(\mu)$ (the left stabilizer of $\mu$, Definition \ref{def: stab of a measure}). And similarly we have an $A$-type-definable set  $\St_{G}^{\leftrightarrow}$  for the two-sided stabilizers.
\end{lemma}
\begin{proof}
For $\varphi(x,y) \in L$, let $\varphi'(x;y,x') := \varphi(x' \cdot x; y)$. We consider the definable sets 
\begin{gather*}
\St_{G, \varphi(x,y) , 1/k} := 
	\Big\{(g, \bar{a}_{\varphi', 1/k}) \in \cU^x \times  M^{\bar{y}_{\varphi', 1/k}}  : \\
	\forall b \  \left \lvert \Av\left( \bar{a}_{\varphi', 1/k}; \varphi'(x;b,1_{G})  \right)  -  \Av\left( \bar{a}_{\varphi', 1/k}; \varphi'(x;b,g) \right) \right \rvert \leq 5/k \Big\}
\end{gather*}
and the type-definable set 
\begin{gather*}
	\St_{G} := \left\{ (g, \bar{a}) \in G \times X_G :  \bigwedge_{\varphi \in L, k \in \mathbb{N}} (g, \bar{a}_{\varphi', 1/k}) \in \St_{G, \varphi(x,y) , 1/k} \right\}.
\end{gather*}
Given $\bar{a} \in X_{G}$, if $g  \in  \Stab_{G,\varphi(x;y), 1/k}(\mu_{\bar{a}})$, then (by Fact \ref{fac: stab of def meas type def}(1)) for all $b \in \cU^y$ we have 
	$$\left \lvert  \mu_{\bar{a}}(\varphi(x,b)) - \mu_{\bar{a}} (\varphi(g \cdot x,b) )\right \rvert \leq 1/k,$$
	 hence $(g,\bar{a}) \in \St_{G, \varphi(x,y) , 1/k}$ (using Remark \ref{rem: 2/k approx}).  Conversely, if $(g,\bar{a}) \in \St_{G, \varphi(x,y) , 1/k}$ then 
	 $$\left \lvert  \mu_{\bar{a}}(\varphi(x,b) ) - \mu_{\bar{a}} (\varphi(g \cdot x,b) )\right \rvert \leq 9/k$$
	  for all $b$, hence $g \in \Stab_{G,\varphi(x;y), 18/k}(\mu_{\bar{a}})$ (by Fact \ref{fac: stab of def meas type def}(2)). 
	  Hence 
	  \begin{gather}
	  	\forall \bar{a} \in X_{G}, \  \Stab_{G, \varphi, 1/k}(\mu_{\bar{a}}) \subseteq \St_{G, \varphi, 1/k}(\cU, \bar{a}) \subseteq \Stab_{G,\varphi, 18/k}(\mu_{\bar{a}});
	  \end{gather}
	  and so 
	\begin{gather*}
		\Stab_{G}(\mu_{\bar{a}}) = G \cap \bigcap_{\varphi,k}  \Stab_{G,\varphi(x;y), 1/k}(\mu_{\bar{a}}) \\
		= G \cap \bigcap_{\varphi,k} \St_{G, \varphi(x,y) , 1/k}(\cU, \bar{a}) = \St(\cU, \bar{a}).
	\end{gather*}
	
	We define $\St_{G}^{\leftrightarrow}$ and $\St^{\leftrightarrow}_{G, \varphi(x,y) , 1/k}$ for the two-sided stabilizers similarly.
\end{proof}

\begin{remark}\label{rem: one stabilizer implies two stabilizers}
	As $|M| \geq 2$  without loss of generality, given any formulas $\varphi_1(x,y_1), \varphi_2(x,y_2) \in L$ we can choose $\varphi(x, y) \in L$, with $y = (y_1, y_2, y', y'')$ and $\varphi(x,y) := (y' = y'' \land \varphi_1(x,y_1)) \lor (y' \neq y'' \land \varphi_2(x,y_2)) $ so that for any $\bar{a} \in X_{G}$ we have $\St_{G,\varphi,k}(x, \bar{a}) \vdash \St_{G,\varphi_i, k}(x,\bar{a})$ for both $i \in \{1,2\}$.
\end{remark}

\subsection{Externally definable, definably amenable subgroups}\label{sec: abelian ext def subgroups}

In this section we consider externally (type-)definable, definably amenable subgroups of definable groups.

First we note that Lemma \ref{lem: unique extension of def meas to Sh} also provides an alternative proof of Fact \ref{fac: group props preserved in Sh exp}(3) for definable amenability,  generalizing to type-definable groups and arbitrary models of $T'$:
\begin{corollary}\label{cor: def am is preserved in Msh new proof}
	($T$ NIP) Assume $M \models T$, $\widetilde{M}' \succ M^{\Sh}$ is a monster model, and $G$ is a type-definable and definably amenable group (working in a monster model $M' \models T$). Then $G$ remains definably amenable in $T' = \Th(M^{\Sh})$.
\end{corollary}
\begin{proof}
	By Fact \ref{fac: def am def measure} applied in $M' \models T$ (which is a monster model for $T$), there exists a left-$G(M')$-invariant measure $\bar{\mu} \in \mathfrak{M}_{G}^L(M')$ and small $M_1 \prec^{L} M'$ so that both $G$ is type-definable over $M_1$ and $\bar{\mu}$ is 	definable over  $M_1$. Let $M_2 \prec^{L'} \widetilde{M}'$ be a small coherent model containing $M_1$. By Lemma \ref{lem: unique extension of def meas to Sh}, let $\bar{\mu}'$ be the unique extension of $\bar{\mu}$ to a measure in $\mathfrak{M}_x^{L'}(M')$, in particular $\bar{\mu}' \in \mathfrak{M}_G^{L'}(M')$. As in the proof of Theorem \ref{thm: type-def and ext def fsg implies def},  $\bar{\mu}'$ is also $G(M')$-invariant: given $g \in G(M')$, $g \cdot \bar{\mu} = \bar{\mu} $ and $(g \cdot \bar{\mu}') \restriction_{L} = g \cdot (\bar{\mu}'\restriction_{L}) = g \cdot \bar{\mu} = \bar{\mu}$, so $g \cdot \bar{\mu}' = \bar{\mu}'$ by uniqueness. This shows that $G$ is definably amenable in $T'$.
\end{proof}
\begin{remark}
	However this proof does not seem to apply to definable extreme amenability however, since it is not known that if $G$ is definably  extremely amenable (NIP) then there exists a definable $G$-invariant \emph{type}.
\end{remark}

\begin{remark}
	Note that Theorems \ref{thm: main for fsg subgroups} and \ref{thm: type-def and ext def fsg implies def} need not hold if we only assume definable amenability instead of fsg. Indeed, we have the following facts:
	
	\begin{fact} \cite[Proposition 3.20]{chernikov2022definable}
	If $T$ is NIP and $G$ is a type-definable group that is definably amenable, then $G^{00}$ is also definably amenable (hence definably amenable in $\Th(M^{\Sh})$ by Fact \ref{fac: group props preserved in Sh exp}).
\end{fact}

\begin{fact} \cite[Remark 8.3(i)]{NIP2}\label{fac: G^00 ext def sometimes} Let $T$ be NIP and $M \models T$ saturated. Assume that $G$ is definable in $M$, definably amenable, and $G/G^{00}$ is a compact Lie group (any definably amenable group in an $o$-minimal theory satisfies these requirements). Then $G^{00}$ is definable in $M^{\Sh}$.
	\end{fact}

	As $G^{00}$ is rarely definable in $M$, even for definably compact groups in $o$-minimal theories (Example \ref{ex: infinitesim not fsg}), Theorem \ref{thm: type-def and ext def fsg implies def} implies in particular that  $G^{00}$ is externally definable, type-definable and  definably amenable, but not fsg in such situations.
\end{remark}

\begin{question}
\begin{enumerate}
	\item Is $G^{00}$ externally definable for all groups definable in $o$-minimal theories?
	\item 	We reiterate Problem 4.21 from \cite{chernikov2014external}: if $G$ is definable in an NIP theory $T$ and $\cU \models T$ is saturated, does naming $G^{00}(\cU)$ by a new predicate
preserve NIP? (Note that $G^{00}$ is not always externally definable, even in stable theories --- in $(\mathbb{Z}, +)$, it is an intersection of infinitely many formulas, not definable, so not externally definable. What about external  definability of the local analog $G^{00}_{\varphi}$ (see \cite[Section 7.14]{hrushovski2024approximate})?
\end{enumerate}
\end{question}

What description can one expect for externally definable (amenable) subgroups of definable groups?
\begin{example}\label{ex: ext def ab wo type-def subgrps}
	
	Let $M_0 = G(M_0) := (\mathbb{Z}, +, <)$ and $M \succ M_0$ a saturated extension.
	Then $H(M) := \mathbb{Z} \leq G(M)$ is externally definable by convexity in the order on $M$, but 
	\begin{enumerate}
		\item $H(M)$ is not type-definable in $M$ and does not contain any type-definable in $M$ subgroups except $\{0\}$;
		\item $H(M)$ is not of the form $K(\cU) \cap M$ for an $\cU$-definable subgroup $K(\cU) \leq G(\cU)$.
	\end{enumerate}
	Considering instead a convex subgroup given by a cut with large cofinality on both sides in a saturated model $M$ of $(\mathbb{R},+,<)$ shows that also 
	\begin{enumerate}
		\item[(3)]	externally definable (abelian) subgroups need not be $\Aut(M/A)$-invariant for any small $A \subset M$ either.
	\end{enumerate}
\end{example}

\noindent At least in this case we have the following approximations for $H$ is $T$:
\begin{definition}
	\begin{enumerate}
	\item  a formula $\varphi(x,y) \in L$ (e.g.~$y_1 < x < y_2$) so that the  family $\mathcal{F} := \{\varphi(M,b) : b \in M^y, \varphi(M,b) \subseteq H(M)\}$  of uniformly $M$-definable subsets of $H(M)$ is \emph{product-directed} (i.e.~for any $n \in \mathbb{N}$ and $X_1, \ldots, X_n \in \mathcal{F}$ there is $X \in \mathcal{F}$ with $(\bigcup_{i \in [n]} X_i)^2 \subseteq X$) and $H(M) = \bigcup \mathcal{F}$. 
	\item a type-definable (over a small set of parameters in $\cU$) subgroup $K(\cU) \leq G(\cU)$ so that $H(M) = K(M)$; moreover, $K$ is uniformly type-definable: there is a formula $\varphi(x,y) \in L$ (again $y_1 < x < y_2$) and a small family $\mathcal{F} = \{\varphi(\cU,b) : b \in \Omega \}$  with $\Omega \subset \cU^y$ of uniformly $\cU$-definable sets so that $\mathcal{F}$ is \emph{downwards directed} (i.e.~for any $n \in \mathbb{N}$ and $X_1, \ldots, X_n \in \mathcal{F}$ there is $X \in \mathcal{F}$ with $X \subseteq \bigcap_{i \in [n]} X_i$) and $\bigcap_{b \in \Omega} \varphi(M,b) = H(M)$.
\end{enumerate}
\end{definition}

Note that conversely by compactness, the union (respectively, the intersection) of $\varphi(M,b)$ for any such family gives an externally definable subgroup of $G(M)$.

\begin{question}
	Is every externally definable (definably amenable/abelian) subgroup of a definable group the union of a definable product-directed family of its subsets? Respectively, the intersection of a downwards product-directed family of uniformly externally definable sets? (See an analogous Question \ref{q: honest defs directed} for externally definable sets.)
 \end{question}

In this section we get a weaker presentation, as  the union of a (uniformly) product-directed \emph{type-definable} family of subsets indexed by a strict externally pro-definable set (Theorem \ref{thm: ext def abelian final approx in M}); as well as the trace on $M$ of a (not necessarily uniformly) $\cU$-type-definable subgroup of $G(\cU)$ (Remark \ref{rem: am subg is trace of type-def}).

We begin with a couple of auxiliary observations. 
\begin{remark}\label{rem: stab contained in a subgroup}
	If $G$ is a type-definable group, $H \leq G$ a type-definable subgroup and $\mu \in \mathfrak{M}_{H}(\cU)$, then $\Stab_{G}(\mu) = \Stab_{H}(\mu) \leq H$ (and similarly for the two-sided stabilizers).
	
	Indeed, suppose $g \in \Stab_{G}(\mu) \setminus \Stab_{H}(\mu)$, then $g \in \Stab_{G}(\mu) \setminus H$, hence $g \cdot H(\cU) \cap H(\cU) = \emptyset$, so by compactness $g \cdot \varphi(\cU) \cap \varphi(\cU) = \emptyset$ for some formula $\varphi(x)$ in the partial type defining $H$ (without loss of generality closed under conjunctions).  As $\mu \in \mathfrak{M}_{H}(\cU)$ we have $\mu(\varphi(x)) = 1$, and as $g \in \Stab_{G}(\mu)$ also $\mu(g \cdot \varphi(x)) = \mu(\varphi(x)) = 1$, a contradiction.
\end{remark}

\begin{lemma}\label{lem: generically stable localizing}($T$ NIP) 
	Assume $M_0 \preceq M$, $M$ is $|M_0|^{+}$-saturated,  $U$ is a type-definable set over $M_0$, $V$ a type-definable set over an arbitrary small model $M_1 \prec \cU$. If $\mu \in \mathfrak{M}_{U}(\cU)$ is generically stable over $M$ and $U(M) \subseteq V(\cU)$, then $\mu \in \mathfrak{M}_{V}(\cU)$.
\end{lemma}
\begin{proof}

Assume $U$ is defined by a partial type $\{\varphi_i(x) : i \in I\}$ with $\varphi_i(x) \in L(M_0)$ and $V$ is defined by a partial type $\{\psi_j(x, b_j) : j \in J\}$ with $\psi_j(x,y_j) \in L, b_j \in M_1^{y_j}$. We need to show that $\mu(\psi_j(x, b_j)) = 1$ for all $j \in J$. Assume this fails, say $\mu(\neg \psi_j(x, b_j)) = \varepsilon > 0$ for some $j \in J$.

	By Fact \ref{fac: NIP gen stab meas equivs}, $\mu$ is fim over $M$, in particular there exist $n \in \mathbb{N}$ and a formula $\theta(x_1,\ldots,x_n) \in L(M)$ such that: for any $\bar{a}$ in $\cU$ with $\models \theta(\abar)$ we have 
		\begin{equation*} 
		\sup_{b \in \cU^{y_j}} |\Av(\abar; \varphi_j(x,b)) - \mu(\varphi_j(x,b))| < \varepsilon/2,
		\end{equation*} 
and $\mu^{\otimes n} \left( \theta \left(\bar{x} \right) \right) >0$. For every finite $I_0 \subseteq I$ we have $\mu \left(\bigwedge_{i \in I_0} \varphi_i(x) \right) =1$ (as $\mu \in \mathfrak{M}_{U}(\cU)$), hence $\mu^{\otimes n} \left(\bigwedge_{i \in I_0, t \in [n]} \varphi_i(x_t) \right) = 1$, and so 
$$\mu^{\otimes n} \left(\bigwedge_{i \in I_0, t \in [n]} \varphi_i(x_t) \land \theta(x_1, \ldots, x_n) \right) > 0,$$
in particular this formula is consistent. Then, by $|M_0|^{+}$-saturation of $M$, we can find $\bar{a} = (a_1, \ldots, a_n)$ with each $a_i \in U(M)$ and $\models \theta(\bar{a})$. Then, by construction, for some $i \in [n]$ we have $a_i \in U(M) \setminus \psi_j(M,b_j)$, hence $a_i \in U(M) \setminus V(\cU)$, contradicting the assumption.
\end{proof}

We generalize \cite[Proposition 6.3]{hils2021definable} from definable abelian groups to type-definable, definably amenable groups:
\begin{theorem}\label{thm: approx def am by stab gs}
	Let $T$ be NIP, and $G$ type-definable over $M_0 \models T$ and  definably amenable. Then there are $M_0$-type-definable sets $\St_G \subseteq G(\cU) \times \cU^{\bar{y}}, X_{G} \subseteq \cU^{\bar{y}}$ with $|\bar{y}| = |T|$ so that:
	\begin{enumerate}
		\item for every $\bar{a} \in X_{G}$, $\St_G(\cU; \bar{a}) = \Stab_{G}(\mu_{\bar{a}}) \leq G(\cU)$ for a measure $\mu_{\bar{a}} \in \mathfrak{M}_{G}(\cU)$ generically stable over $\bar{a}$;
		\item $G(\cU) = \bigcup_{\bar{a} \in X_{G}} \St_{G}(\cU, \bar{a})$;
		\item the family $\{\St_{G}(\cU, \bar{a}) : \bar{a} \in X_{G}\}$ is directed, i.e.~for any small $I \subseteq X_{G}$ there is some $\bar{a}_{I} \in X_{G}$ so that $\St_{G}(\cU, \bar{a}) \subseteq \St_{G}(\cU, \bar{a}_{I}) $ for all $\bar{a} \in I$.
	\end{enumerate}
\end{theorem}

\begin{proof}
	
We will use some material from Section \ref{sec: hyperdef sets from measures}. Namely, we consider the $M_0$-type-definable set $X_{G}$ (Definition \ref{def: X_G}): generically stable measures in $\mathfrak{M}_G(\cU)$ are precisely the measures of the form $\mu_{\bar{a}}$ (Definition \ref{def: X_x}) for $\bar{a} \in X_G$ (but possibly $\mu_{\bar{a}} = \mu_{\bar{a}'}$ for $\bar{a} \neq \bar{a}' \in X_{G}$), and $\mu_{\bar{a}}$ is generically stable over $\bar{a}$. By Remark \ref{rem: meas code rep inside the partial type} we may assume that every element of every tuple in $X_G$ is in $G(\cU)$.   By Lemma \ref{lem: unif defining stab of mu} we have an $M_0$-type-definable set $\St_{G}(x,\bar{y})$ so that $\left( \Stab_{G}(\mu_{\bar{a}}) \right)(\cU) = \St(\cU, \bar{a})$ for all $\bar{a} \in X_G$. 

\begin{claim}\label{cla: ab of gs stabilizers}
	For every small model $M \succ M_0$ there is a measure $\mu_{M} \in \mathfrak{M}_{G}(\cU)$ and a small model $M' \succ M$ so that $\mu_{M}$ is generically stable over $M'$ and $G(M) \subseteq \Stab_{G}(\mu_{M})$.
\end{claim}
\begin{proof}
	 As $G$ is definably amenable and NIP, by Fact \ref{fac: def am def measure} there is a measure $\mu \in \mathfrak{M}_{G}(\cU)$  which is left-$G(\cU)$-invariant, and definable over some small model $M_1 \prec \cU$, without loss of generality $M \prec M_1$.
  Then, using Fact \ref{fac: symmetrization measure}, $\mu|_{M_1}$ can be extended to a measure $\mu_{M} \in \mathfrak{M}_{G}(\cU)$ which is generically stable over some small $M_1 \prec M' \prec \cU$ and left $G(M)$-invariant (see also \cite[Section 7]{NIP2}).
\end{proof}

In particular, for every small model $M$, $G(M) \subseteq \Stab_{G}(\mu)$ and $\mu = \mu_{\bar{a}}$ for some $\bar{a} \in X_{G}$, hence $G(M) \subseteq \Stab_{G}(\mu_{\bar{a}}) = \St_G(\cU; \bar{a})$. This shows that $G(\cU) =  \bigcup_{\bar{a} \in X_{G}} \St_{G}(\cU, \bar{a})$.

Now let $(\bar{a}_i : i \in I)$ be a small tuple of elements of $X_{G}$. Let $M$ be a small $(|M_0|+|T|)^{+}$-saturated model containing $M_0$ and all of $\bar{a}_i$, then $\mu_{\bar{a}_i} \in \mathfrak{M}_{G}(\cU)$ is generically stable over $M$ for all $i \in I$. Let $\mu_{M} \in \mathfrak{M}_{G}(\cU)$ be given  by Claim \ref{cla: ab of gs stabilizers}, then $\mu_{M}$ is generically stable over some small $M' \succ M$ and $\Stab_{G}(\mu_{M})$ is type-definable over $M'$. As  $G(M) \subseteq \Stab_{G}(\mu_{M})$, we have $\mu_{\bar{a}_i} \in \mathfrak{M}_{\Stab_{G}(\mu_{M})}(\cU)$ by Lemma \ref{lem: generically stable localizing}. Hence  $\Stab_{G}(\mu_{\bar{a}_i}) \leq \Stab_{G}(\mu_{M})$ for all $i \in I$, by Remark \ref{rem: stab contained in a subgroup}. As $\mu_{M} = \mu_{\bar{a}}$ for some $\bar{a} \in X_{G}$, we have $\St_{G}(\cU; \bar{a}_i) \subseteq \St_{G}(\cU; \bar{a})$ for all $i \in I$.
\end{proof}

This immediately implies a generalization of \cite[Corollary 6.4]{hils2021definable} from abelian to definably amenable groups:
\begin{corollary}
	($T$ NIP) If $G$ is definably amenable with no indiscernible linearly ordered (by inclusion) family of type-definable subgroups, then $G$ is fsg.
\end{corollary}
\begin{proof}
	Since the family $\{\St_{G}(\cU; \bar{a} : \bar{a} \in X_{G} \}$ is directed and its union is $G(\cU)$, it follows from the assumption that $G(\cU) = \St_{G}(\cU; \bar{a})$ for some $\bar{a} \in X_{G}$, hence $G(\cU) = \Stab_{G}(\mu_{\bar{a}})$ for a generically stable measure $\mu_{\bar{a}} \in \mathfrak{M}_{G}(\cU)$, witnessing that $G$ is fsg.
\end{proof}

\begin{remark}
Theorem \ref{thm: approx def am by stab gs} also holds with two-sided stabilizers of generically stable measures instead of just left stabilizers.
\end{remark}
\begin{proof}
	As $G$ is definably amenable and NIP, by \cite[Lemma 6.2]{chernikov2018definably} there is a measure $\mu \in \mathfrak{M}_{G}(\cU)$  which is both left and right $G(\cU)$-invariant, and definable over some small model. Then Claim \ref{cla: ab of gs stabilizers} holds with  $G(M) \subseteq \Stab^{\leftrightarrow}_{G}(\mu_{M})$ (the two-sided stabilizer of $\mu_{M}$, Definition \ref{def: stab of a measure}), applying Fact \ref{fac: def am def measure} to each of left and right multiplications by elements of $G(M)$. The proof of Theorem \ref{thm: approx def am by stab gs} then goes through working with the type-definable set $\St_G^{\leftrightarrow}$ (Lemma \ref{lem: unif defining stab of mu}) defining two-sided stabilizers instead of $\St_{G}$.
\end{proof}

The following example, checked in discussion with Kyle Gannon, demonstrates that we cannot assume the groups $\St_{G}(\cU;\bar{a}) \leq G$ in Theorem \ref{thm: approx def am by stab gs} to be fsg (equivalently, we cannot assume that the measures $\mu_{\bar{a}}$ are supported on their stabilizers $\Stab_{G}(\mu_{\bar{a}})$):
\begin{example}\label{ex: no fsg subroups in R+}
	Let $M_0 = (R,+, <, \ldots)$ be an $o$-minimal (expansion of an) ordered (abelian) group, $G(M_0) = (R,+)$, $\cU \succ M_0$ a monster model.  There are no non-trivial type-definable fsg subgroups of $G(\cU)$.
	
	Indeed, first note that there are no non-trivial definable subgroups of $G(\cU)$ (by $o$-minimality any such group would be a finite union of intervals with end-points in $\cU$, but taking $a$ close enough to the right end-point of the right-most interval $I$, $a+a \notin I$). Next, we have:
	\begin{claim}
		If $\mu \in \mathfrak{M}_{G}(\cU)$ is any measure, then the subgroup $\Stab_{G}(\mu) \leq G(\cU)$ is convex.
	\end{claim}
	\begin{proof}
		Assume $0 < c < b $ in $\cU$ and $b \in \Stab_{G}(\mu)$. Suppose $c \notin \Stab_{G}(\mu)$, i.e.~$\mu(X) \neq \mu(c+ X)$ for some $\cU$-definable set $X$. Then, by $o$-minimality, we have $\mu(I) \neq \mu(c+ I)$ for some non-empty $\cU$-definable interval $I$.
		\begin{enumerate}
			\item $I = (-\infty, d)$ for some $d \in \cU$.
			 Then  $I \subseteq c+I \subseteq b+ I$ and necessarily $\mu(c+ I) > \mu(I)$, contradicting $\mu(I) = \mu(b+ I)$.
			
			\item $I = (d, + \infty)$ for some $d \in \cU$.
			 Then $b+ I \subseteq c+ I \subseteq I$ and necessarily $\mu(c+ I) < \mu(I)$, contradicting $\mu(I) = \mu(b+ I)$.
			 \item $I = (a,d)$ for some $a<b \in \cU$. 
			 
			 \noindent Suppose $\mu(c+I) > \mu(I)$. Then $
			 \mu((-\infty, d)) = \mu( (-\infty, a]) + \mu( I) < \mu( (-\infty, a]) + \mu( c+ I) = \mu( (-\infty, a] \sqcup (c+I)) \leq \mu(b + (-\infty, d))$, thus $b \notin \Stab_{G}(\mu)$. 
			 
			\noindent  If $\mu(c+I) < \mu(I)$, we have $0 < \mu(I \setminus c+ I) \leq \mu((a,a+c))$, so $\mu((a, + \infty)) > \mu(b+ (a, + \infty))$, thus $b \notin \Stab_{G}(\mu)$.
			\item And similarly for closed intervals. \qedhere
		\end{enumerate}
	\end{proof}
 Every type-definable fsg group $H(\cU) \leq G(\cU)$ is of the form $\Stab_{G}(\mu)$ for a generically stable measure $\mu$ supported on it, hence definable in $\cU^{\Sh}$ by the claim,  so it has to be definable in $\cU$ by Theorem \ref{thm: type-def and ext def fsg implies def}.
\end{example}

\begin{remark}
	By \cite[Theorem 1.5]{chernikov2024definable}, Example \ref{ex: no fsg subroups in R+} also shows that there are no  idempotent generically stable measures in  $(R,+, <, \ldots)$ (except for $\tp(1/\cU)$), which is a $dp$-minimal ordered abelian group. Contrast this with the case of abelian groups in ACVF \cite[Lemma 5.1]{hrushovski2019valued}.
\end{remark}

Next we strengthen ``directed'' in Theorem \ref{thm: approx def am by stab gs} to ``uniformly directed'' (similar to Definition \ref{def: unif dir}):
\begin{proposition}\label{prop: getting unif directed}
 In Theorem \ref{thm: approx def am by stab gs}(3), the family $\{\St_{G}(\cU, \bar{a}) : \bar{a} \in X_{G}\}$ is moreover \emph{uniformly directed}, i.e.~for any  $\varphi(x,y) \in L, k \in \mathbb{N}$ there are $\varphi'(x,y') \in L, k' \in \mathbb{N}$ so that: for any small $I \subseteq X_{G}$ there is some $\bar{a}_{I} \in X_{G}$ so that $\St_{G,\varphi',1/k'}(x,\bar{a}) \vdash \St_{G,\varphi,1/k}(x,\bar{a}_{I})$ for all $\bar{a} \in I$ simultaneously.
\end{proposition}
\begin{proof}

First we refine Remark \ref{rem: stab contained in a subgroup} in the case when $H = \Stab_{G}(\nu)$ for some generically stable measure $\nu \in \mathfrak{M}_{G}(\cU)$ (possibly different from $\mu$):
\begin{claim}\label{cla: more precise directedness for stabilizers}
	For every $\varphi(x,y) \in L, k \in \mathbb{N}$ there exist $\xi(x,y') \in L, k' \in \mathbb{N}$ satisfying the following: if $\bar{a}_1, \bar{a}_2 \in X_{G}$ and $\mu_{\bar{a}_1} \in \mathfrak{M}_{\Stab_{G}(\mu_{\bar{a}_2})}(\cU)$, then $ \St_{G, \xi , 1/(k')}(x, \bar{a}_1) \vdash \St_{G, \varphi, 1/k}(x, \bar{a}_2)$.
\end{claim}
\begin{proof}
For every $\varphi(x,y) \in L, k \in \mathbb{N}$, $\St_{G, \varphi, 1/k}(\cU, \bar{a})$ is defined by $\psi_{\varphi, 1/k}(x;\bar{a})$ for some $\psi_{\varphi, 1/k}(x;y') \in L$ depending only on $\varphi, \varepsilon$ (but not on $\bar{a}$). Then  ``$\mu_{\bar{a}_1} \in \mathfrak{M}_{\Stab_{G}(\mu_{\bar{a}_2})}(\cU)$'' is a type-definable condition on $(\bar{a}_1, \bar{a}_2) \in X_G \times X_G$ via  a partial type $\pi(\bar{y}_1, \bar{y}_2)$ given by 
$$\bigwedge_{\varphi \in L, \  k, \ell \in \mathbb{N}}  \Av\left( (\bar{a}_1)_{\psi_{\varphi, 1/k}, 1/\ell}; \psi_{\varphi, 1/k}(x;(\bar{a}_2)_{\varphi', 1/k}) \right) \geq 1 - 2/\ell
$$
(see Lemma \ref{lem: unif defining stab of mu}).

Now suppose claim fails for $\varphi(x,y)$ and $k$. Then, using Remark \ref{rem: one stabilizer implies two stabilizers}, for any $n \in \mathbb{N}$ and $\varphi_i(x,y_i)  \in L, k_i \in \mathbb{N}$ for $i \in [n]$ there are some $b, \bar{a}_1, \bar{a}_2$ in $\cU$ with 
$$ \models X_{G}(\bar{a}_1) \land X_{G}(\bar{a}_2) \land \pi(\bar{a}_1, \bar{a}_2) \land   \bigwedge_{i \in [n]} \St_{G, \varphi_i, 1/(k_i)}(b, \bar{a}_1) \land \neg \St_{G, \varphi, 1/k}(b, \bar{a}_2).$$
By compactness we thus find some  $b, \bar{a}$ in $\cU$ so that
$$\models X_{G}(\bar{a}_1) \land X_{G}(\bar{a}_2) \land  \pi(\bar{a}_1, \bar{a}_2)  \land  \St_{G}(b, \bar{a}_1) \land \neg \St_{G, \varphi, 1/k}(b, \bar{a}_2).$$
But this contradicts Remark \ref{rem: stab contained in a subgroup}, with  $\mu := \mu_{\bar{a}_1}, \nu :=  \mu_{\bar{a}_2}, H := \Stab_G(\mu_{\bar{a}_2})$.
\end{proof}

Now in the proof of directedness in Theorem \ref{thm: approx def am by stab gs}, given any small tuple $(\bar{a}_i : i \in I)$  of elements of $X_{G}$,
we found a generically stable measure  $\mu_{M} \in \mathfrak{M}_{G}(\cU)$, say $\mu_{M} = \mu_{\bar{a}}$ for $\bar{a} \in X_{G}$, so that  $\mu_{\bar{a}_i} \in \mathfrak{M}_{\Stab_{G}(\mu_{M})}(\cU)$ for all $i \in I$. By Claim \ref{cla: more precise directedness for stabilizers} we have $\xi$ and $k'$ so that 
$$ \models \forall x \left( \St_{G, \xi , 1/k'}(x, \bar{a}_i) \rightarrow \St_{G, \varphi , 1/k}(x, \bar{a}) \right)$$
for all $i \in I$, as wanted.
\end{proof}

We give a more explicit proof of Proposition \ref{prop: getting unif directed} in the case of an abelian group $G$. For (automorphism-) invariant measures $\mu$,$\nu$ in an NIP group, we have that $\Stab_{G}(\mu) \subseteq \Stab_{G}(\mu \ast \nu)$ and $\Stab_{G}(\nu) \subseteq \Stab_{G}^{\leftarrow}(\mu \ast \nu)$ (the right stabilizer). This point was used e.g.~in \cite[Lemma 6.2]{chernikov2018definably} to show that every definably amenable NIP group admits a bi-invariant measure (see also \cite[Proposition 3.23]{chernikov2022definable}). The following is a  refinement of \cite[Section 6, (A)]{hils2021definable}: 
\begin{lemma}\label{lem: convolution stab uniformly}
	Let $G = G(x)$ be a type-definable abelian group, $n \in \omega$ and $\mu_i \in \mathfrak{M}_{G}(\cU)$ for $i \in [n]$ generically stable. Let $\mu :=  \bigast_{i \in [n]} \mu_i$. For $\varphi(x,y)$, let $\varphi'(x; y, x') := \varphi(x \cdot x', y)$. Then for any $i \in [n]$ and $\varepsilon > 0$, 
	$$ \Stab_{G, \varphi', \varepsilon} (\mu_i) \subseteq \Stab_{G, \varphi, 2 \varepsilon} (\mu)$$
	 (in the notation of Fact \ref{fac: stab of def meas type def}); in particular, $\Stab_{G}(\mu_i) \subseteq \Stab_{G}(\mu)$.
\end{lemma}
\begin{proof}
For $I \subseteq [n]$, let $\mu_{I} :=  \bigast_{i \in I} \mu_i$. As $\mu_i$ are $\otimes$-commuting by generic stability and $G$ is abelian, the order in the product does not matter, and $\mu_{I}$ is generically stable  (see \cite[Proposition 3.15]{chernikov2022definable}). In particular $\mu = \mu_i \ast \mu_{[n] \setminus \{i\}}$ for all $i \in [n]$.

Assume $g \in \Stab_{G, \varphi', \varepsilon} (\mu_i)$. Then, by Fact \ref{fac: stab of def meas type def}(1), for all $(b,g') \in \cU^{y}\times \cU^x$ we have  
\begin{gather}
	\left \lvert \mu ( \varphi'(x; b,g'))  -  \mu ( \varphi'(g \cdot x; b, g')) \right \rvert \leq \varepsilon. \label{eq: close for all params phi'}
\end{gather}

	Given an arbitrary $b \in \cU^y$, let $M \prec \cU$ be a small model so that all of $\mu_i$ are definable over $M$, $G$ is type-definable over $M$, $g \in G(M)$ and $b \in M^y$. We have
	\begin{gather*}
		\left \lvert \mu (\varphi(x;b) ) -  \mu (\varphi(g \cdot x;b)) \right \rvert \\
		= 		\left \lvert \mu_i \ast \mu_{[n] \setminus \{i\}} (\varphi(x;b) )  - \mu_i \ast \mu_{[n] \setminus \{i\}} (\varphi(g \cdot x;b) )\right \rvert \\
		= \int_{q \in S_y(M)} \left \lvert \mu ( \varphi(x \cdot g'; b))  -  \mu ( \varphi(g \cdot x \cdot g'; b)) \right \rvert d \mu'_{[n] \setminus \{i\}}(q) 
		\overset{\eqref{eq: close for all params phi'}}{\leq } \varepsilon,
	\end{gather*}
where $\mu'_{[n] \setminus \{i\}}$ is the regular Borel probability measure on $S_y(M)$ extending $\mu_{[n] \setminus \{i\}}$, and $g' \models q$ arbitrary in $\cU$. As $b$ was arbitrary, this implies $g \in \Stab_{G, \varphi, 2 \varepsilon} (\mu)$ by Fact \ref{fac: stab of def meas type def}(2).
\end{proof}

\begin{proof}[A more explicit proof of Proposition \ref{prop: getting unif directed} for  abelian $G$.]
	Let $I \subseteq X_G$ be small. For any finite $J \subseteq I$, let $\mu_{J} :=  \bigast_{\bar{a} \in J} \mu_{\bar{a}}$. As $\mu_{J} \in \mathfrak{M}_{G}(\cU)$ is generically stable, we have $\mu_{J} = \mu_{\bar{a}_{J}}$ for some $\bar{a}_{J} \in X_{G}$.  By Lemma \ref{lem: convolution stab uniformly}, for each $\bar{a} \in J$ we have $\Stab_{G, \varphi', \varepsilon} (\mu_{\bar{a}}) \subseteq \Stab_{G, \varphi, 2 \varepsilon} (\mu_{\bar{a}_{J}})$, hence, using the notation from the proof of Lemma \ref{lem: unif defining stab of mu},
$$ \models \forall x \left( \St_{G, \varphi'(x,y) , 1/(32 k)}(x, \bar{a}) \rightarrow \St_{G, \varphi(x,y) , 1/k}(x, \bar{a}_{J}) \right).$$
 As $I$ is small, by compactness there exists $\bar{a}' \in X_{G}$ so that 
 $$ \models \bigwedge_{\varphi(x,y) \in L, k \in \mathbb{N}}  \forall x \left( \St_{G, \varphi'(x,y) , 1/(32 k)}(x, \bar{a}) \rightarrow \St_{G, \varphi(x,y) , 1/k}(x, \bar{a}') \right)$$
 for all $\bar{a} \in I$ simultaneously. Hence the formula $\St_{G, \varphi'(x,y) , 1/(32 k)}(x, \bar{y}) \in \St_{G}(x;\bar{y})$ satisfies the requirement. 
\end{proof}

We also include a proof of Claim \ref{cla: more precise directedness for stabilizers} without using compactness theorem, in the hope for future applications:

\begin{proof}[A proof of Claim \ref{cla: more precise directedness for stabilizers} without using compactness.]

	For $\varphi(x,y) \in L$, where $y$ is any tuple of variables, let $\varphi'(x; y, y') := \varphi(y' \cdot x; y)$. The following are easy to verify from Fact \ref{fac: stab of def meas type def} (for any definable measure $\mu$), using that for all $b$ and $g$ we have $\varphi(g \cdot x;b) = \varphi'(x; b,g)$:
 \begin{gather}
 	\Stab_{G,\varphi', \varepsilon}(\mu) \subseteq \Stab_{G,\varphi, 2 \varepsilon}(\mu), \ 
 	\left( \Stab_{G,\varphi', \varepsilon}(\mu) \right)^{-1} \subseteq \Stab_{G,\varphi, 2 \varepsilon}(\mu), \label{eq: epsilon stabilizers 1}\\
 	\left( \Stab_{G, (\varphi')', \varepsilon}(\mu) \right) \cdot  \left( \Stab_{G, (\varphi')', \varepsilon}(\mu) \right)  \subseteq \Stab_{G, \varphi,  4 \varepsilon}(\mu)\label{eq: epsilon stabilizers 2}
 \end{gather}
(note that $\cdot$ and  $^{-1}$ are defined as $\Stab_{G, \varphi, \varepsilon}(\mu)  \subseteq \varphi_0(\cU)$ from Remark \ref{rem: type-def group formula} always).
%

 Without loss of generality $T$ contains a constant symbol. Then for any $\varphi_1(x,y_1), \varphi_2(x,y_2) \in L$ we can choose $\xi(x, \bar{y}) \in L$, with $\bar{y} = (y_1, y_2, y')$ so that $\Stab_{G,\xi, \varepsilon}(\mu)\subseteq \Stab_{G,\varphi_i, 2\varepsilon}(\mu)$ for both $i \in \{1,2\}$.

Also, given any formula $\varphi(x,y) \in L, k \in \mathbb{N}$, by Lemma \ref{lem: unif defining stab of mu} for any $\bar{a} \in X_{G}$ we have
\begin{gather}
	\forall \bar{a} \in X_{G}, \  \Stab_{G, \varphi, 1/k}(\mu_{\bar{a}}) \subseteq \St_{G, \varphi, 1/k}(\cU, \bar{a}) \subseteq \Stab_{G,\varphi, 18/k}(\mu_{\bar{a}}),\label{eq: epsilon stabilizers 3}
\end{gather} 
 and $\St_{G, \varphi, 1/k}(\cU, \bar{a})$ is defined by $\psi(x;\bar{a}_{\varphi, 1/k}, \bar{a}_{\varphi', 1/k})$ for some $\psi(x;y_{\psi}) \in L$ depending only on $\varphi, \varepsilon$ (but not on $\bar{a}$).
 
 Let us choose such $\xi(x, y_{\xi})$ for $\varphi_1 (x, y_1) := ((\varphi(x,y)')')'$ and $\varphi_2(x,y_2) := \psi(x,y_{\psi})$ where $\psi(x,\bar{a})$ is the formula as above defining  $\St_{G,\varphi_1, 1/(8k)}(\mu_{\bar{a}})$. And let $k' := 1/(18 \cdot 8 \cdot k)$ (in particular $1/k' < 1/2$).

Now suppose towards contradiction that $g \in \Stab_{G,\xi, 1/k'}(\mu) \setminus \Stab_{G,\varphi,1/k}(\nu)$ for $\nu \in \mathfrak{M}_{G}(\cU)$ a generically stable measure, hence $\nu = \mu_{\bar{a}}$ for some $\bar{a} \in X_G$.

Then $g \cdot \Stab_{G,\varphi_1, 1/(8k)}(\nu) \cap \Stab_{G,\varphi_1, 1/(8k)}(\nu) = \emptyset$ (otherwise there are some $h_1,h_2 \in \Stab_{G,((\varphi')')', 1/(8k)}(\nu)$ so that $g \cdot h_1 = h_2$, hence, using \eqref{eq: epsilon stabilizers 1} and \eqref{eq: epsilon stabilizers 2}, $g = h_2 \cdot h_1^{-1} \in \Stab_{G, \varphi, 1/k}(\nu)$ --- a contradiction).
Using \eqref{eq: epsilon stabilizers 3} and the choice of $k'$, we have $\St_{G, \varphi_1, 1/k'}(\cU, \bar{a}) \subseteq \Stab_{G,\varphi_1, 1/(8k)}(\nu)$, hence also $g \cdot \St_{G, \varphi_1, 1/k'}(\cU, \bar{a}) \cap \St_{G, \varphi_1, 1/k'}(\cU, \bar{a}) = \emptyset$. 
As $\mu \in \mathfrak{M}_{\Stab_{G}(\nu)}(\cU)$,  we have $\mu (\St_{G, \varphi_1, 1/k'}(x, \bar{a})) = 1$. And as $g \in \Stab_{G,\xi, 1/k'}(\mu)$ and $1/k' < 1/2$, by the choice of $\xi$ we have $\mu (g \cdot \St_{G, \varphi_1, 1/k'}(x, \bar{a})) > 1/2$ --- a contradiction since the total measure is $1$.
\end{proof}

Using Theorem \ref{thm: approx def am by stab gs} and earlier results, we get a description for externally type-definable, definably amenable subgroups of definable groups when \emph{working in a monster model of $M^{\Sh}$}, as a directed union of a family of type-definable \emph{in $T$} subroups:
\begin{proposition}\label{prop: ext def am descr in M'}
		Let $T$ be NIP,  $M \models T$, $\widetilde{M}' \succ M^{\Sh}$ a monster model for $T'$, $G$ a definable group in $M$ and $H(M') \leq G(M')$ a subgroup which is $L'(M)$-type-definable and definably amenable in $T' = \Th(M^{\Sh})$. 
	 Then there exist an $L(M)$-type-definable set $\St_G(x,\bar{y})$ and $L'(M)$-type-definable set $Y \subseteq (M')^{\bar{y}}$, with $|\bar{y}| = |T|$, so that:
	\begin{enumerate}
		\item for every $\bar{a} \in Y(M')$, $\St_G(M'; \bar{a})$  is a type-definable in $M' \models T$ subgroup of $H(M')$ of the form $\Stab_{G}(\mu_{\bar{a}}) \leq H(M')$ for a measure $\mu_{\bar{a}} \in \mathfrak{M}^{L}_{G}(M')$ generically stable over $\bar{a}$ in $T$;
		\item $H(M') = \bigcup_{\bar{a} \in Y} \St_{G}(M', \bar{a})$;
		\item the family $\{\St_{G}(M', \bar{a}) : \bar{a} \in Y\}$ is directed;
		\item $H(M) = \St_G(M; \bar{a})$ for all $\bar{a} \in Y$.
	\end{enumerate}
\end{proposition}
\begin{proof}
	Let $\St'_{H}(x,\bar{y}')$ and $X'_{H}(\bar{y}')$ be $L'(M)$-type-definable as given by  Lemma \ref{lem: unif defining stab of mu}  applied to $H $ in $T'$, and $L(M)$-type-definable $\St_{G}(x,\bar{y}), X_{G}(\bar{y})$ given by Lemma \ref{lem: unif defining stab of mu} for $G$ in $T$. Here $\bar{y}' = \left(y_{\varphi, 1/k} : \varphi(x,y) \in L', k \in \mathbb{N} \right)$ is a tuple of length $|T'|$, and $\bar{y} = \bar{y}'\restriction_{L} := \left(y_{\varphi, 1/k} : \varphi(x,y) \in L, k \in \mathbb{N} \right)$ is a subtuple of length $|T|$.

	So generically stable measures in $\mathfrak{M}^{L'}_{H}(\cU)$ are precisely the measures of the form $\bar{\mu}_{\bar{a}'}$ with $\bar{a}' \in X'_{H}(M')$, and  $\St'_{H}(M',\bar{a}') = \Stab_{H}(\mu_{\bar{a}'})$. Given $\bar{a}' \in X'_H$, let $\bar{a} := \bar{a}'\restriction_{L}$, then $\bar{a} \in X_{G}$,  $(\mu_{\bar{a}'})\restriction_{L} = \mu_{\bar{a}}$ (where the former is given by Definition \ref{def: X_x} in $T'$, and the latter in $T$), $\mu_{\bar{a}}$ is generically stable (in $T$) and $\mu_{\bar{a}'}$ is the unique measure extending it; and every $\bar{a} \in X_{G}$ extends (non-uniquely) to $\bar{a}' \in X'_{G}$  (both by Theorem \ref{thm: corresp for gs measures}).

	By Remark \ref{rem: stab contained in a subgroup} in $T'$, for any generically stable $\bar{\mu}' \in \mathfrak{M}^{L'}_{H}(M')$ we have $\Stab_{G}(\bar{\mu}') = \Stab_{H}(\bar{\mu}') \leq H(M')$; and, as the proof of Theorem \ref{thm: main for fsg subgroups} shows, $\Stab_{G}(\bar{\mu}') = \Stab_{G}(\bar{\mu}' \restriction_{L})$, so 
	\begin{gather}
		\St_{G}(M', \bar{a}) = \St'_{G}(M',\bar{a}') = \St'_{H}(M', \bar{a}') \leq H(M') \textrm{ for all } \bar{a}' \in X'_{H}. \label{eq: internal stab subrp of H}
	\end{gather}

	By Theorem \ref{thm: approx def am by stab gs} applied to $H$ in $T'$, the family $\{\St'_{H}(M',\bar{a}') : \bar{a}' \in X'_{H} \}$ is directed and its union is $H(M')$. 
Let 
$$Y := \{\bar{a} \in X_{G} : \exists \bar{a}' \in X'_{H}  \left( \bar{a}' \restriction_{L} = \bar{a} \ \land H(M) \  \subseteq  \St'_{G}(M', \bar{a}') \right)\}.$$
 By $L'(M)$-type-definability of $X'_{H}$ and compactness in $\widetilde{M}'$, $Y$ is also $L'(M)$-type-definable. And $Y$ and $\St_{G}$ satisfy the requirement.
\end{proof}

\begin{remark}\label{rem: am subg is trace of type-def}
	In particular, if $G$ is definable and $H \leq G(M)$ is externally (type-)definable and definably amenable in $T'$, then $H(M) = K(M)$ for some $K(M') \leq H(M')$ a type-definable in $M'$ subgroup of $H(M')$. Can we additionally assume that $K(M')$ is definably amenable in $T$?
\end{remark}

Next we would like to extract from Proposition \ref{prop: ext def am descr in M'} a description of $H(M)$ \emph{entirely in $M$}. As Example \ref{ex: ext def ab wo type-def subgrps} shows, in general there might be no $M$-type-definable subgroups of $H(M)$. We are going to present $H(M)$ as a union of a (uniformly, product-) directed family of type-definable subsets of the form $\St_{G}(M, \bar{a})$ with $\bar{a}$  in $M$ but \emph{not necessarily in $X_{G}$} (as $X_{G} \cap M$ can be empty). Such a set corresponds to the intersection of approximate stabilizers of a (possibly infinite) sequence of finitely supported in $M$ measures (as opposed to a single, not necessarily finitely supported, generically stable measure when $\bar{a} \in X_{G}$), hence no longer gives a subgroup of $G(M)$ --- but the directedness of the family will be preserved.

\begin{theorem}\label{thm: ext def abelian final approx in M}
	Let $T$ be NIP,  $M \models T$, $G$ a definable group in $M$ and $H(M) \leq G(M)$ a subgroup that is $L'(M)$-definable and definably amenable  in $T'$.
	 Then there is an $M$-type-definable in $T$ set $\Gamma(x, \bar{y})$ and a strict $L'(M)$-pro-definable set $Z \subseteq \cU^{\bar{y}}$ with $|\bar{y}| = |T|$ so that:

	\begin{enumerate}
		
		\item $H(M) = \bigcup_{\bar{a} \in Z} \Gamma(M, \bar{a})$;
		\item  the family $\{\Gamma(M, \bar{a}) : \bar{a} \in Z\}$ is uniformly product directed.
	\end{enumerate}
	
	\end{theorem}
\begin{proof}

We continue in the setting and notation of (the proof of) Proposition \ref{prop: ext def am descr in M'}. So $H(M') \leq G(M')$ is $L'(M)$-definable in $T'$. And, from 
Proposition \ref{prop: ext def am descr in M'}\eqref{eq: internal stab subrp of H}, we have $H(M') = \bigcup_{\bar{a}' \in X'_{H}} \St'_{H}(M', \bar{a}')$ and $H(M) = \St_{G}(M, \bar{a})$ for all $\bar{a} \in  Y$. 
Here $\bar{a}' = \left(\bar{a}_{\varphi, 1/k} : \varphi(x,y) \in L', k \in \mathbb{N} \right)$ is a tuple of length $|T'|$, and $\bar{a} = \bar{a}'\restriction_{L} := \left(\bar{a}_{\varphi, 1/k} : \varphi(x,y) \in L, k \in \mathbb{N} \right)$ is a subtuple of length $|T|$ (and similarly for $\bar{y}'$ and $\bar{y} := \bar{y}' \restriction L$).
The following claims refine/uniformize the steps in the proof of Proposition \ref{prop: ext def am descr in M'}.

\begin{claim}\label{cla: uniform directedness in H}
	For every $\varphi(x,y) \in L', k \in \mathbb{N}$ there are some $\varphi'(x,y') \in L', k' \in \mathbb{N}$ so that: for every $n \in \mathbb{N}$ and $\bar{a}'_1, \ldots, \bar{a}'_n \in X'_{H}$ there is some $\bar{a}' \in X'_{H}$ so that $\St'_{H, \varphi', 1/k'} \left(x,(\bar{a}'_i)_{\varphi', 1/k'} \right) \vdash^{T'} \St'_{H, \varphi, 1/k}\left(x,\bar{a}'_{\varphi, 1/k} \right)$ for all $i \in [n]$.
\end{claim}
\begin{proof}
	By Proposition \ref{prop: getting unif directed} applied to $H$ in $T'$.
\end{proof}

\begin{claim}\label{cla: approx stabs imply H}
	There are some $\varphi(x,y) \in L, k \in \mathbb{N}$ so that: for every $\bar{a}' \in X'_{H}$,  $\St_{G, \varphi, 1/k} \left(x, \bar{a}_{\varphi, 1/k} \right) \vdash^{T'} \psi(x)$.
\end{claim}
\begin{proof}
Suppose not. Hence, using Remark \ref{rem: one stabilizer implies two stabilizers}, for any $n \in \mathbb{N}$ and $\varphi_i(x,y_i)  \in L, k_i \in \mathbb{N}$ for $i \in [n]$ there are some $b, \bar{a}'$ in $M'$ with 
$$\widetilde{M}' \models \neg H(b) \land X'_{H}(\bar{a}') \land  \bigwedge_{i \in [n]} \St_{G, \varphi_i, 1/(k_i)}(b, \bar{a}).$$
By saturation of $\widetilde{M}'$ we thus find some  $b, \bar{a}'$ in $M'$ so that

$$\widetilde{M}' \models \neg H(b) \land X'_{H}(\bar{a}') \land  \St_{G}(b, \bar{a}),$$
contradicting \eqref{eq: internal stab subrp of H}.
\end{proof}

\begin{claim} \label{cla: approx stab in T imply T'}
	For every $\varphi(x,y) \in L', k \in \mathbb{N}$ there is $\varphi'(x,y') \in L, k' \in \mathbb{N}$ so that: for every $\bar{a}' \in X'_{H}$	 we have $\St_{G, \varphi', 1/k'}\left(x,\bar{a}_{\varphi', 1/k'} \right) \vdash^{T'} \St'_{H, \varphi, 1/k} \left(x,\bar{a}'_{\varphi, 1/k} \right) $.
\end{claim}
\begin{proof}
Similarly, follows by compactness in $T'$ as $\St_{G}(M', \bar{a}) = \St'_{H}(M', \bar{a}') \leq H(M')$ for all $\bar{a}' \in X'_{H}$ by \eqref{eq: internal stab subrp of H}.
\end{proof}

Uniform (product-) directedness holds for arbitrary tuples  in $H(M')$  as opposed to just $\bar{a}' \in X'_{H}$:
\begin{claim}\label{cla: putting things together}
	For any $\varphi(x,y) \in L'$ and $ k\in \mathbb{N}$ there are $\varphi'(x,y') \in L, k' \in \mathbb{N}$ so that:  for any $n \in \mathbb{N}$ and finite tuples $\bar{a}^{\ast}_1, \ldots, \bar{a}^{\ast}_n $ in $H(M')$ of length corresponding to $\bar{y}_{\varphi', 1/k'}$, there is a tuple $\bar{a}^{\ast}$ in $H(M')$ of length corresponding to $\bar{y}_{\varphi, 1/k}$ so that $\left( \bigcup_{i \in [n]}\St_{G, \varphi', 1/k'}\left(M',\bar{a}^{\ast}_i \right) \right)^2 \subseteq \St'_{H, \varphi, 1/k}\left(M',\bar{a}^{\ast} \right)$.
\end{claim}
\begin{proof}

Given $\varphi(x,y) \in L'$ and $k \in \mathbb{N}$, by \eqref{eq: epsilon stabilizers 2} there exist $\varphi_1(x,y_1)$ and $k_1 \in \mathbb{N}$ so that $(\St_{G, \varphi_1, 1/k_1}(M';\bar{a}'))^2 \subseteq \St_{G, \varphi, 1/k}(M';\bar{a}')$ for any $\bar{a}' \in X'_{H}(M')$. 

Applying Claims \ref{cla: uniform directedness in H} and \ref{cla: approx stab in T imply T'} to $\varphi_1, k_1$ we find $\varphi'(x,y') \in L, k' \in \mathbb{N}$ so that for any $n$ and $\bar{a}'_1, \ldots, \bar{a}'_n \in X'_{H}$ there is $\bar{a}' \in X'_{H}$ so that 
$$\left( \bigcup_{i \in [n]}\St_{G, \varphi', 1/k'}\left(M',(\bar{a}_i)_{\varphi', 1/k'} \right) \right)^2 \subseteq \St'_{H, \varphi, 1/k}\left(M',\bar{a}'_{\varphi, 1/k} \right).$$

Now let $\bar{a}^{\ast}_1, \ldots, \bar{a}^{\ast}_n $ be arbitrary tuples in $H(M')$ of length corresponding to $\bar{y}_{\varphi', 1/k'}$. Then there exist $\bar{a}'_1, \ldots, \bar{a}'_n \in X'_{H}$ so that $(\bar{a}'_i)_{\varphi', 1/k'} = a^{\ast}_i$ for all $i \in [n]$. Namely, $\Av^{L'}(\bar{a}_i^{\ast})$ (calculated in $T'$) is a (finitely supported) generically stable measure in $\mathfrak{M}_{H}^{L'}(M')$, so we can choose $\bar{a}'_i \in X'_{H}$ so that $\mu_{\bar{a}'_i} = \Av^{L'}(\bar{a}_i^{\ast})$ and $(\bar{a}'_i)_{\varphi', 1/k'} = a^{\ast}_i$). Then there is some $\bar{a}' \in X'_{H}$ so that 
$$\left( \bigcup_{i \in [n]}\St_{G, \varphi', 1/k'}\left(M',\bar{a}^{\ast}_i \right) \right)^2 \subseteq \St'_{H, \varphi, 1/k}\left(M',\bar{a}'_{\varphi, 1/k} \right).$$
And, by Remark \ref{rem: meas code rep inside the partial type} we can choose $\bar{a}'$ contained in $H(M')$, so in particular $\bar{a}^{\ast} := \bar{a}'_{\varphi,1/k}$ is contained in $H(M')$.
\end{proof}

	Now we let $\Gamma(x,\bar{y}) := \St_{G}(x,\bar{y})$ and take the strict $L'(M)$-pro-definable set 
	$$Z := \left\{\bar{a} \in (M')^{\bar{y}}: \bigwedge_{\varphi \in L, k \in \mathbb{N}} \bar{a}_{\varphi, 1/k} \subseteq H(M') \right\}.$$
	
	By Claim \ref{cla: putting things together} and using $M^{\Sh} \prec^{L'} \widetilde{M}'$ repeatedly it follows that the family $\{\St_{G}(M, \bar{a}) : \bar{a} \in Z(M) \}$ is uniformly product-directed. 
	Using Claim \ref{cla: approx stabs imply H} we have $\Gamma(M,\bar{a}) \subseteq H(M)$ for every $\bar{a} \in Z(M)$. And as $H(M) = \St'_{H}(M, \bar{a}')$ for some $\bar{a}' \in X'_{H}$, given any finite $A \subseteq H(M)$ and using $M^{\Sh} \prec^{L'} \widetilde{M}'$ repeatedly we can choose $\bar{a} \in Z(M)$ with $A \subseteq \St_{G}(M, \bar{a})$.
\end{proof}


\begin{remark}
	By Proposition \ref{prop: ext def subgroups approx V-def}, we can also write $H(M)$ as a union of a uniformly $L(M)$-$\bigvee$-definable directed family of its subgroups.
\end{remark}

Finally, we discuss the relation to some results in the literature on definable envelopes of subgroups of NIP groups preserving certain properties:
\begin{fact}
	Let $M \models T$ be NIP, $\cU \succ M$ a monster model and $G$ an $M$-definable group. Let $H$ be a (not necessarily definable) subgroup of $G(M)$.
	\begin{enumerate}
		\item \cite{zbMATH05662715} If $H$ is abelian, there exists a $\cU$-definable abelian subgroup $K \leq G(\cU)$ with $H \subseteq K(M)$.
		\item \cite{zbMATH06177016} If $H$ is nilpotent, there exists a $\cU$-definable subgroup $K \leq G(\cU)$ of the same nilpotency class as $H$ with $H \subseteq K(M)$.
		\item \cite{zbMATH06177016} If $H$ a normal subgroup of $G(M)$ and solvable, there exists a $\cU$-definable subgroup $K \leq G(\cU)$ of the same derived length  as $H$ with $H \subseteq K(M)$.
		\item \cite{zbMATH06578125} If $H$ is solvable, there exists a type-definable in $\cU$ subgroup $K \leq G(\cU)$ given by a downwards directed family of  uniformly $\cU$-definable sets with $H \subseteq K(M)$.
	\end{enumerate}
\end{fact}
\begin{question}
Is the analogous statement true for $H$ an amenable subgroup of $G(M)$?	
\end{question}

In particular, every externally definable subgroup of $G(M)$ with a corresponding property is contained in the trace on $M$ of an $\cU$-definable subgroup $K(\cU)$ of $G(\cU)$ with this property.  We can actually choose $K$ to be $M$-definable in this case, e.g.:

\begin{proposition}\label{prop: ext def ab in def ab}
	Let $M \models T$ be NIP, $G$ a  definable group in $M$, and $H \leq G(M)$ an externally definable abelian subgroup. Then there exists an $M$-definable abelian subgroup $K(M) \leq G(M)$ so that $H(M) \leq  K(M)$.
\end{proposition}
\begin{proof}
	Let $\theta(x) \in L(M')$ be an honest definition for $\neg H(x)$, i.e.~$\theta(x) \vdash^{T'} \neg H(x)$ and $\neg H(M) = \theta(M)$. Then $H(M) =\neg \theta(M)$ and $H(x) \vdash^{T'} \neg \theta(x)$.
	
	Let $K(M') := C_{G}(C_{G}(\neg \theta(M')))$ (where $C_{G}(A)$ denotes the centralizer of $A$ in $G(M')$). As  $\widetilde{M}' \succ^{L'} M^{\Sh}$, we have that $H(M')$ is an abelian subgroup of $G(M')$, hence $K(M')$ is an abelian $L(M')$-definable subgroup of $G(M')$ and $H(M') \subseteq  K(M')$.  Say $K(M')$ is defined by $\psi(x,b)$ for $\psi(x,y) \in L$ and $b \in (M')^y$. Using $\widetilde{M}' \succ^{L'} M^{\Sh}$ again, there is some $b' \in M^y$ so that still $H(M) \subseteq \psi(M,b')$ and $\psi(M,b')$ is an abelian subgroup of $G(M)$.
\end{proof}

\section{Hyperdefinable group chunk for partial type-definable types}
\label{sec: Hyperdefinable group chunk for partial type-definable types}

\subsection{Preliminaries on hyperdefinability}\label{sec: Basics on hyperdefinability}
We refer to e.g.~\cite[Chapter 3]{wagner2000simple}, \cite[Chapter 15]{casanovas2011simple} for basics on hyperdefinability. We set up some notation in a way convenient for the discussion of large partial types and their type-definability (which agrees with the usual treatment of hyperimaginaries, see e.g.~``uniform definitions'' in \cite[Section 1]{fanlo2023piecewise}).

A hyperdefinable set is a quotient $P = X/E$  with $X$ non-empty ($\emptyset$-)type-definable set and $E$ a type-definable equivalence relation. Given $a \in X$, $a_E$ denotes $a/E \in P$, and given $a \in P$, $a^* \in a$ denotes a representative of $a$; we let $\quot_E: X \to P, \quot_E(a) := a_E$ denote the quotient map. The elements of a hyperdefinable set are called \emph{hyperimaginaries}. We let $\cU^{\heq}$ denote $\cU$ together with the collection of all hyperimaginaries modulo $(\emptyset)$-type-definable equivalence relations.

Let $E_i$ be type-definable equivalence relations on $\cU^{x_i}$. A partial type $\pi((x_i)_{i \in I})$ over $\cU$ is \emph{$E_i$-invariant on $x_i$} if $\pi((x_i)_{i \in I}) \land E_i(x_i, x'_i) \vdash \pi(\bar{x}')$ where $\bar{x}'$ is obtained from $(x_i)_{i \in I}$ replacing $x_i$ by $x'_i$. For a formula $\varphi((x_i)_{i \in I_0}) \in L(\emptyset)$, consider the partial type  
$$\Delta_{\varphi}((x_i)_{i \in I_0}) := \exists (x'_i)_{i \in I_0} \left( \bigwedge E_i(x_i,x'_i) \land \varphi((x'_i)_{i \in I_0}) \right)$$
$E_i$-invariant on $x_i$ for all $i \in I_0$.
\begin{lemma}\label{lem: hyperdef part types formulas}
	If $\pi((x_i)_{i \in I})$ is closed under conjunctions, it is $E_i$-invariant on $x_i$ for all $i \in I$ if and only if $\pi(x) \equiv \bigcup_{\varphi \in \pi} \Delta_{\varphi}((x_i)_{i \in I_0})$.
\end{lemma}
\begin{proof}
	Right to left is clear.  Conversely, assume  $\varphi((x_i)_{i \in I_0}) \in \pi$. Clearly, $\varphi((x_i)_{i \in I_0}) \vdash \Delta_{\varphi}((x_i)_{i \in I_0})$. On the other hand, as $\pi((x'_i)_{i \in I}) \land \bigwedge_{i \in I} E_i(x,x') \vdash \pi((x'_i)_{i \in I})$, by compactness and closure of $\pi$ under conjunctions there is some $\psi((x_i)_{i \in I_0}) \in \pi$ so that $\psi((x'_i)_{i \in I_0}) \land \bigwedge_{i \in I_0}E(x_i,x'_i) \vdash \varphi((x_i)_{i \in I_0})$, hence $\Delta_{\psi}((x_i)_{i \in I_0}) \vdash \varphi((x_i)_{i \in I_0})$. 
\end{proof}

Let $P = X/E$ and $Q_i = Y_i/F_i$ be hyper-definable. 
By a \emph{partial type $\pi(x)$ on $P$ over a (possibly large) set of hyperimaginaries} $B = \{ b_{F_i} \in Q_i : i \in I\}$ we mean a partial type of the form $\pi^*(x, (b_i)_{i \in I})$ over a set of real elements $\{b_i : i \in I\}$ so that $\pi^*(x; (y_i)_{i \in I})$ is a partial type over $\emptyset$ which is $E$-invariant on $x$ and $F_i$-invariant on $y_i$  and $\pi^*(x, (b_i)_{i \in I}) \vdash X(x) \cup \{Y_i (y_i)\}_{i \in I}$. Up to logical equivalence of partial types, this does not depend on the choice of the representatives $b_i, i \in I$.
	We say that $a_E \models \pi(x)$ if $a \models \pi^*(x; (b_i)_{i \in I})$. Again, this does not depend on the choice of representatives. For $\pi(x), \rho(x)$ partial types on $P$, we write $\pi(x) \vdash \rho(x)$ if $\pi^*(x, (b_i)_{i \in I}) \vdash \rho^*(x, (c_j)_{j \in J})$. We say that $\pi(x)$ (over $\cU^{\heq}$) is \emph{closed under implication} if for every small partial type $\rho(x)$ on $P$ (over hyperimaginary parameters) with  $\pi(x) \vdash \rho(x)$ we have $\rho^*(x, (c_j)_{j \in J}) \subseteq \pi^*(x; (b_i)_{i \in I})$. In particular if $\pi, \rho$ are closed under implication and $\pi \vdash \rho, \rho \vdash \pi$, then $\pi^*(x, (b_i)_{i \in I}) = \rho^*(x, (c_j)_{j \in J})$.

For a small set $B \subseteq \cU^{\heq}$, a set $V \subseteq P$ is $B$-type-definable if $V = \{b \in P : \models \pi(b) \}$ for $\pi$ a partial type on $P$ over $B$ (then $\quot^{-1}_{E}(V)$ is a $B^*$-type-definable subset of $X$ for any set $B^*$ of representatives of $B$).
Given hyperdefinable sets $P = X/E$ and $Q = Y/F$, the cartesian product $P \times Q$ is canonically identified with the hyperdefinable set $(X \times Y) / (E \land F)$, where $(x,y) (E\land F) (x',y')$ if $x E x'$ and $y F y'$. This allows to talk about hyperdefinable relations and (partial) functions. In particular, if $f$ is a $B$-type definable function from $P$ to $Q$, then for any $B$-type-definable sets $V \subseteq P, W \subseteq Q$, both $f(V) \subseteq Q$ and $f^{-1}(W) \subseteq P$ are $B$-type-definable. If $\rho(x,y)$ is a partial type on $P \times Q$ over $c \subseteq \cU^{\heq}$, witnessed by $\rho^*(x,y;c^*)$ and $b \in Q$, then $\rho(x;b)$ is a partial type on $P$ witnessed by $\rho^*(x;b^*,c^*)$.

\begin{definition}\label{def: def type hyperim}
	We say that a partial type $\pi(x)$ (over $\cU^{\heq}$) on $P$ is \emph{type-definable over $A \subseteq \cU^{\heq}$} if for every sort $Q$ in $\cU^{\heq}$ and small partial type  $\gamma(x,y)$ on $P \times Q$ over $\emptyset$,  the set $\{b \in Q : \pi(x) \vdash \gamma(x,b)\}$ is type-definable over $A$ by some partial type $\rho(y) =: (d_{\pi}x) \gamma(x,y)$ on $Q$.
\end{definition}
\begin{remark}\label{rem: around typ-def types def}
	By Lemma \ref{lem: hyperdef part types formulas}, it suffices to require this for $\gamma(x,y) := \Delta_{\varphi} (x,y)$, $\varphi(x,y) \in L$; equivalently, we can let $\gamma(x,y)$ be an arbitrary small type over $A$, as then we can take $(d_\pi x) \gamma(x,y) := (d_\pi x) \gamma'(x,y, \bar{a})$, where $\bar{a}$ is a tuple enumerating $A$ and $\gamma'(x,y,z)$ is a partial type on $P \times Q \times R$ such that $\gamma(x,y) = \gamma'(x,y,\bar{a})$.
\end{remark}
	
\begin{definition}
		If $C \subseteq \mathbb{M}^{\heq}$ is a set of parameters, we define $\pi|_{C} := \bigcup \{\gamma(x,m) : \mathbb{M} \models (d_{\pi}x) \gamma(x,m), m \in C, \gamma(x,y)\textrm{ small partial type over }\emptyset\}$. If $C$ is an $|A|^+$-saturated model, then $\pi|_C$ determines $\pi$.
\end{definition}

\begin{remark}\label{rem: realize over dcl heq}
	If $C \subseteq \mathbb{M}^{\heq}$ is a small set of parameters, $b \in \cU^{\heq}, a \models \pi|_{C b}$ and $d \in \dcl^{\heq}(C, b)$, then $a \models \pi|_{C d}$.
	
	Indeed, as $d \in \dcl^{\heq}(C, b)$, there is a partial $C$-type-definable function $f: Y/F \to Z/G$ so that $f(b) = d$. Let $\Gamma_f(y,z)$ be the partial type over $C$ defining the graph of $f$, then $\Gamma_f(b, d)$.
	Now assume that $\gamma(x,z)$ is a partial type on $X/E \times Y/F$ over $C$ such that $\pi(x) \vdash \gamma(x,d)$. Then $\pi(x) \vdash \exists z \left( \Gamma_f(b,z) \land \gamma(x,z) \right)$, with a small partial type on $X/E \times Z/G$ over $C, b$ on the right. Hence $ \models \exists z \left( \Gamma_f(b,z) \land \gamma(a,z) \right)$ by assumption, so $ \models \gamma(a,d)$ as $d$ is the unique realization of $\Gamma_f(b,z)$.
\end{remark}

\begin{definition}\label{def: pushforward of types}
	Let $A \subseteq \cU^{\heq}$, $\pi(x)$ a partial type over $\cU^{\heq}$ on $P$, $Q$ a sort in $\cU^{\heq}$ and $f: P \to Q$ an $A$-type-definable partial function \emph{defined on $\pi$} (i.e.~$\pi(x) \vdash \exists y \Gamma_{f}(x,y) $), in which case already $\pi(x)|_{A} \vdash \exists y \Gamma_f(x,y) $). We define the \emph{pushforward} 
	\begin{gather*}
		(f_{\ast} \pi)(y) := \bigcup\Big\{\gamma(y) : \pi(x) \vdash \exists u (\Gamma_f(x,u) \land \gamma(u)), \\
		\gamma(y) \textrm{ is  a small partial type on } Q\Big\}.
	\end{gather*}
	Then $f_{\ast}\pi$ is a partial type on $Q$. (This corresponds to the pushforward of the corresponding filter of type-definable subsets of $P$.)
\end{definition}

\begin{remark}\label{rem: pushfoward preserves type-def}
	If $\pi, f$ are type-definable over $A$, then $f_{\ast}\pi$ is also type-definable over $A$: for every small partial type $\gamma(y,z)$ on $Q$ over $\emptyset$, we have 
	$$(d_{(f_{\ast} \pi)} y) \gamma(y,z) = d_{\pi}x (\exists u (\Gamma_f(x,u) \land \gamma(u,z))).$$
\end{remark}

\begin{remark}\label{rem: pushforward sets of realiz}
	In the same situation as above, if $\pi$ is a partial type closed under implication, then for any $B \supseteq A$ we have $f( \pi|_{B}( \cU)) = (f_{\ast} \pi)|_{B}(\cU)$.
	
	Indeed, we clearly have $f(\pi|_{B}(\cU)) \subseteq (f_{\ast} \pi)|_{B}(\cU)$ from the definition. Then $\pi|_{B}(\cU) \subseteq f^{-1}(f(\pi|_{B}(\cU))) \subseteq f^{-1}((f_{\ast} \pi)|_{B}(\cU))$, hence $\pi|_{B}(\cU) = f^{-1}((f_{\ast} \pi)|_{B}(\cU))$ using $B$-type-definability of these sets and closure of $\pi$ under implication.
\end{remark}

\begin{definition}
	Let $\pi(x), \rho(y)$ be $A$-type-definable partial types over $\cU^{\heq}$ on sorts $P,Q$ respectively. We define a partial type $(\pi \otimes \rho)(x,y)$ on $P \times Q$  by taking
	$$(d_{\pi \otimes \rho} x y)\gamma(x,y;z) := (d_{\pi} x) \left( (d_{\rho} y) \gamma(x,y,z) \right)$$
	for every small partial type $\gamma(x,y;z)$ over $\emptyset$ on $P \times Q \times R$. (This is well-defined by Remark \ref{rem: around typ-def types def} as $(d_{\rho} y) \gamma(x,y,z)$ is a small partial type on $P \times R$ over $A$.) 
\end{definition}

\begin{remark}\label{rem: prod of partial types realiz}
	It follows from the definition that for any $B \supseteq A$, $(a, b) \models (\pi \otimes \rho)|_{B}$ if and only if $a \models \pi|_{B}$ and $b \models \rho|_{B,a} $.
\end{remark}

\begin{definition}\label{def: germs}
	Let $\pi(x)$ be a partial $C$-type-definable type over $\cU^{\heq}$. Two type-definable functions $f(x,b)$ and $g(x,c)$, for $b,c$ hyperimaginaries, defined on $\pi$, are said to have the same \emph{$\pi$-germ}, written $f(x,b) \sim_{\pi} g(x,c)$, if $f(a,b) = g(a,c)$ \emph{for all} $a \models \pi|_{C, b,c}$. Equivalently, if $\pi(x) \vdash \exists z \left( \Gamma_f(x,b,z) \land \Gamma_g(x,c,z)\right)$; if and only if $ \models (d_{\pi}x) \exists z  \left(\Gamma_f(x,b,z) \land \Gamma_g(x,c,z) \right)$ --- a type-definable condition on $(b,c)$. Then $\sim_{\pi}$ is an equivalence relation on all type-definable partial functions defined on $\pi$, and we let $[f(x,b)]_{\pi}$ denote the $\pi$-germ of $f(x,b)$.
	\end{definition}
	
	\begin{remark}\label{rem: push forward germs}
		In fact, push-forwards of $\pi$ are well-defined for $\pi$-germs of type-definable partial functions defined on $\pi$: if  $f(x) \sim_{\pi} g(x)$, then $f_{\ast} \pi = g_{\ast} \pi$. Indeed, if $\gamma(u)$ is  a small partial type,  $\pi(x) \vdash \exists u (\Gamma_f(x,u) \land \gamma(u))$ if and only if $\pi(x) \vdash \exists u (\Gamma_g(x,u) \land \gamma(u))$.
\end{remark}
	
		\begin{remark}\label{rem: real partial type in heq}
		If $\pi$ is a partial type over $\cU$ closed under implication, then  $\pi$ viewed as a partial type on the home sort over $\cU^{\heq}$ is also closed under implication.  And if $\pi$ was $A$-type-definable, $A \subseteq \cU$, then $\pi$ remains an $A$-type-definable partial type in $\cU^{\heq}$.
	\end{remark}
\subsection{Hyperdefinable group chunk}\label{sec: Hyperdefinable group chunk proof}

 We will need a hyper-definable generalization of the group chunk theorem from \cite[Section 3.4]{hrushovski2019valued} to \emph{type-definable} partial types (as opposed to definable partial types; see also \cite[Section 3]{rideau2021short}; and \cite[Section 4.7]{wagner2000simple} for a different hyper-definable group chuck theorem in simple theories).  We refer to e.g.~\cite[Section 4.3]{wagner2000simple} or \cite{fanlo2023piecewise} for the basics on hyper-definable groups. A hyper-definable group is given by a hyper-definable set $G$ and a type-definable binary function $\cdot: G^{2} \to G$ such that $(G, \cdot)$ forms a group.

\begin{remark}\label{rem: translates of partial types}
	If $(G, \cdot_{G})$ is a hyper-definable group, $\pi(x)$ a partial type on $G$ (over $\cU^{\heq}$) and $g \in G(\cU)$, we let $g \cdot \pi := (t_g)_{\ast} \pi$ where $t_g: G \to G$ is the type-definable map $h \in G \mapsto g \cdot h$. Then $g \cdot \pi$ is a partial type on $G$, and if $\pi$ was $C$-type-definable, then $\pi$ is $C \cup \{g \}$-type-definable. Note that if $g \cdot \pi = \pi$, then also $g^{-1} \cdot \pi = \pi$.
\end{remark}

 \begin{definition}\label{defn: group chunk}
 	Let $P = X/E$ with $X,E$ type-definable over $C \subseteq \cU^{\heq}$. A \emph{group chunk} over a small set of parameters $C$ on $P$ is  a $C$-type-definable partial type $\pi(x)$ on $P$ in $T^{\heq}$ closed under implication and $C$-type-definable partial functions $F,H: P^2 \to P$ defined on $\pi^{\otimes 2}$ (and hence on $\pi^{\otimes 2}|_{C}$) and $K: P^2 \to P$, so that:
 	\begin{enumerate}
 		\item for all $a \models \pi|_{C}$, $(F_a)_{\ast} \pi = \pi$, where $F_a(x) := F(a,x)$;
 		\item if $(a,b) \models \pi^{\otimes 2}|_{C}$, then $H(a,F(a,b)) = b$ and  $ K(b, F(a,b)) = a$ (in particular the partial function $K$ is  defined).
 		\item if $(a,b,c) \models \pi^{\otimes 3}|_{C}$, then $F(a,F(b,c)) = F(F(a,b), c)$.  	\end{enumerate}
 \end{definition}
 
 \begin{remark}
 \begin{enumerate}
 
 	\item Note that by Definition \ref{defn: group chunk}(1), if $(a,b) \models \pi^{\otimes 2}|_{C}$, then $(a, F(a,b)) \models \pi^{\otimes 2}|_{C}$, hence $H(a,F(a,b))$ in Definition \ref{defn: group chunk}(2) is defined. However, the requirement that $ K(y, F(x,y))$ is defined on $g_{\ast}(\pi^{\otimes 2})$ for  $g(x,y) := (y, F(x,y))$ needs to be made separately (follows from (1) if $\pi$ $\otimes$-commutes with itself).
 	\item The composition of partial functions in Definition \ref{defn: group chunk}(3) is well-defined. Indeed, assume $(a,b,c) \models \pi^{\otimes 3}|C$. Then $(b,c) \models \pi^{\otimes 2}|_{C,a}$, so $F(b,c) \models \pi^{\otimes 2}|_{C,a, b}$ by Definition \ref{defn: group chunk}(1), so $(a, F(b,c)) \models \pi^{\otimes 2}|_{C}$. Also $c \models \pi|_{C,a,b}$, hence $c \models \pi|_{C, F(a,b)}$ (by Remark \ref{rem: realize over dcl heq}).
 	
 	\item In \cite[Section 3.4]{hrushovski2019valued}, in Definition \ref{defn: group chunk}(2) one simply requires $a \in \dcl(C, b, F(a,b))$, $b \in \dcl(C, a, F(a,b))$, which by compactness allows to find definable partial functions $H,K$. Since  $\dcl^{\heq}$ is only given by partial types rather than formulas, we cannot apply the same compactness argument here, and require the existence of type-definable partial functions $H,K$ explicitly.
 \end{enumerate}
 \end{remark}

 Of course, every type-definable translation invariant partial type on a hyper-definable group gives rise to a group chunk:
\begin{remark}\label{rem: group chunk from a group}
	Assume $G$ is a hyper-definable over $C \subseteq \cU^{\heq}$ group and $\pi(x)$ is a partial left-$G$-invariant type on $G$ (over $\cU^{\heq}$), type-definable over $C$. Let $\Gamma(x,y,z)$ be the $C$-type-definable graph of multiplication in $G$, without loss of generality $\Gamma(x,y,z)$ defines a graph of partial function from any two of its coordinates to the third one (intersecting with the partial type defining $G$ on every coordinate, if necessary). Consider the partial type-definable functions  $F(x,y)$ with $\Gamma_{F}(x,y,z) := \Gamma(x,y,z)$, $H$ with $\Gamma_{H}(x,y,z) := \Gamma(x,z,y)$ and $K$ with $\Gamma_{K}(x,y,z) := \Gamma(z,x,y)$. Then $(\pi, F,H,K)$ is a group chunk satisfying Definition \ref{defn: group chunk}.
\end{remark}

Conversely,

\begin{theorem}\label{thm: hyperdef group chunk gives group}
	Let $(\pi,F,H,K)$ be a group chunk over $C$ on $P$. Then there exists a $C$-hyper-definable group $G = (Y/F', \cdot_G)$ and a $C$-type-definable partial function $f: \pi \to G$ defined and injective on $\pi$, so that for any $(a,b) \models \pi^{\otimes 2}|_{C}$ we have $f(F(a,b)) = f(a) \cdot_{G} f(b)$ and the global partial type $f_{\ast} \pi$ concentrates on $G$ and is left-$G$-invariant.
\end{theorem}
\begin{proof}
	We follow closely the proof of \cite[Proposition 3.15]{hrushovski2019valued}, but provide some details.
	
	We let $\mathfrak{F}(\pi)$ be the semigroup under composition of all $\cU^{\heq}$-type-definable functions $h$ such that $h_{\ast} \pi = \pi$. This semigroup has a quotient consisting of the $\pi$-germs of elements in $\mathfrak{F}(\pi)$. Inside of which the invertible $\pi$-germs form a group, denoted $\mathfrak{G}(\pi)$.
	
		Let $P' := \{ a \in P : a \models \pi|_{C}(x) \}$. 
	
	For every $a \in P'$, the $\pi$-germ $[F_a] := [F(a,x)]_{\pi}$ is invertible, in fact $[H_a] := [H(a,x)]_{\pi}$ is the two-sided inverse of $[F_a]$. Indeed, for any $a \models \pi|_{C}$ and $a_0 \models \pi|_{C,a}$, let $b := F(a,a_0)$, $b \models \pi|_{C,a}$ by by Definition \ref{defn: group chunk}(1). Then $a_0 = H(a, F(a,a_0))$ by Definition \ref{defn: group chunk}(2). So $ [H_a] \circ [F_a](a_0) = a_0$. Conversely, for any $a \models \pi|_{C}$, $B \supseteq C$ and $b \models \pi|_{B,a}$, there is some $a_0 \models \pi|_{B,a}$ with $b = F(a,a_0)$ (by Remark \ref{rem: pushforward sets of realiz}, as $(F_{a})_{\ast} \pi = \pi$). Again, $a_0 = H(a, F(a,a_0))$ (as $B \supseteq C$ was arbitrary, this shows in particular that $(H_a)_{\ast} \pi = \pi$), hence $[F_a] \circ [H_a] (b) = b$.
	
	Let $\mathfrak{G}_F(\pi)$ be the subgroup of $\mathfrak{G}(\pi)$ generated by the elements $\{[F_a] : a \in P'\}$ and their inverses. 
	
	The map $a \in P' \mapsto [F_a] \in \mathfrak{G}_{F}(\pi)$ is injective. Indeed, let $a,a' \in P'$ be such that $[F_a] = [F_{a'}]$. Let $b \models \pi|_{C, a, a'}$. Then by Definition \ref{defn: group chunk}(2), $a = K(b, F(a,b)) = K(b, F(a',b))  = a'$ (note that we are testing injectivity on a single realization of $\pi$).

\begin{claim*}
	Every element of $\mathfrak{G}_{F}(\pi)$ is of the form $[F_a] \circ [F_b]^{-1}$ for some $a,b \in P'$.
\end{claim*}
\begin{proof}
To see this, first we make the following two observations.
	\begin{enumerate}
	\item For any $a \models \pi|_{C}$, $D \supseteq C$ and $c \models \pi|_{D,a}$, there is some $b \models \pi|_{D,a}$ so that $[F_b] = [F_a] \circ [F_c]$.
	
		\item For any $a \models \pi|_{C}$, $D \supseteq C$ and $b \models \pi|_{D,a}$, there is some $c \models \pi|_{D,a}$ such that $[F_{a}]^{-1} \cdot [F_{b}] = [F_c]$.
		\end{enumerate}
		
For (1), let $b := F(a,c)$, $b \models \pi|_{D,a}$ by Definition \ref{defn: group chunk}(1). Let $e \models \pi|_{D, a,b,c}$ be arbitrary, then $(a,c,e) \models \pi^{\otimes 3}|_{C}$, and by Definition \ref{defn: group chunk}(3) we have $[F_{b}](e) = F(b,e) = F(F(a,c),e) = F(a, F(c,e)) = [F_a] \circ [F_c] (e)$, so $[F_b] = [F_a] \circ [F_c]$.
		
For (2), we have $F(a,b) \models \pi|_{D,a}$ (by Definition \ref{defn: group chunk}(1)). And there is $c \models \pi|_{D,a}$ such that $F(a,c) = b$ (by Definition \ref{defn: group chunk}(1) and Remark \ref{rem: pushforward sets of realiz}). Let $e \models \pi|_{D, a,b,c}$ be arbitrary. As $(a,c,e) \models \pi^{\otimes 3}|_{C}$, by Definition \ref{defn: group chunk}(3) we have $[F_b](e) = F(F(a,c), e) = F(a, F(c,e)) = [F_a] \circ [F_c] (e)$. Hence $[F_b] = [F_a] \circ [F_c]$, so $[F_{a}]^{-1} \cdot [F_{b}] = [F_c]$.

To prove the claim, first note that $[F_a]$ is of the required form for every $a \in P'$: by (1) there exist $b,c \models \pi|_{C}$ with $[F_b] = [F_a] \circ [F_c]$, hence $[F_a] = [F_b] \circ [F_c]^{-1}$. Then it suffices to show that $[F_a] \circ [F_b]^{-1} \circ [F_c] \circ [F_{d}]^{-1}$ has the required form for all $a,b,c,d \in P'$. Let $e_0 \models \pi|_{C,a,b,c,d}$. By (2) there is $e_1 \models \pi|_{C,a,b,c,d}$ with $[F_d]^{-1} \circ [F_{e_0}] = [F_{e_1}]$. By (1) there is $e_2 \models \pi|_{C,a,b,c,d}$ with $[F_c] \circ [F_{e_1}] = [F_{e_2}]$. By (2) there is $e_3 \models \pi|_{C,a,b,c,d}$ with $[F_{b}]^{-1} \circ [F_{e_2}] = [F_{e_3}]$, and by (1) there is $e_4 \models \pi|_{C,a,b,c,d}$ with $[F_a] \circ [F_{e_3}] = [F_{e_4}]$. It follows that $[F_a] \circ [F_b]^{-1} \circ [F_c] \circ [F_d]^{-1} \circ [F_{e_0}] = [F_{e_4}]$, hence $[F_a] \circ [F_b]^{-1} \circ [F_c] \circ [F_d]^{-1}  = [F_{e_4}] \circ [F_{e_0}]^{-1}$, as wanted.
\end{proof}

%

	Using the claim, we can identify $\mathfrak{G}_{F}(\pi)$ with a hyperdefinable group. Let $Y := P' \times P'$, $F' := \{((a,b), (a',b')) \in Y^2 : [F_a] \circ [F_b]^{-1} = [F_{a'}] \circ [F_{b'}]^{-1}\}$. By $C$-type-definability of $\pi$, $F'$ is a $C$-type-definable equivalence relation on $Y$ (see Definition \ref{def: germs}) by 
	\begin{gather*}
		F((a,b), (a',b')) \iff
		 (d_{\pi} y) \Big( \exists u \exists v \exists u' \exists v' \\
		 \left(\Gamma_{H}(b,y,u) \land 
		\Gamma_{F}(a,u,v) \land \Gamma_{H}(b', y, u') \land  \Gamma_{F}(a',u', v') \land F'(v,v') \right) \Big)
	\end{gather*}

		We let $G := Y/F'$ (note that a quotient of a hyperdefinable set by a type-definable equivalence relation is a hyper-definable set).
		Similarly, the set 
	\begin{gather*}
		\left\{ (a_1, \ldots, a_6) \in P' : [F_{a_1}] \circ [F_{a_2}]^{-1} \circ [F_{a_3}] \circ [F_{a_4}]^{-1} = [F_{a_5}] \circ [F_{a_6}]^{-1} \right\}
	\end{gather*} 
	is $C$-type-definable by $C$-type-definability of $\pi$. Hence $\cdot_{G}$ is $C$-type-definable.
	For $a \in P'$, we let $f(a) := (b,c)/F'$ for some $(b,c) \in Y$ with $[F_a] = [F_{b}] \circ [F_c]^{-1}$. Then $f: P' \to G$ is a partial function defined on $\pi|_{C}$ (by the Claim) and injective (as discussed above).
	Note that the graph of $f$ is a $C$-type-definable subset of $ P \times G$ using type-definability of $\pi$ as above.
	
	For any $(a,b) \models \pi^{\otimes 2}|_{C}$, we have $c := F(a,b) \models \pi|_{C,a}$. For any $e \models \pi|_{C, a,b}$,  $[F_c](e) = F(F(a,b), e) = F(a,F(b,e)) = [F_a] \circ [F_b] (e)$, so $[F_c] = [F_a] \circ [F_b]$. This shows that $f(F(a,b)) = f(a) \cdot_{G} f(b)$.

Finally, given an arbitrary $D \supseteq C$, assume $g \models (f_{\ast} \pi)|_{C}$ and 	$h \models (f_{\ast} \pi)|_{D,g}$. By Remark \ref{rem: pushforward sets of realiz} there  exist $a \models \pi|_{C}$ with $f(a) = g$ and  $b \models \pi|_{D,g}$ with $f(b) = h$. As $f$ is injective on $\pi|_{C}$, $a \in \dcl^{\heq}(C, g)$, so $b \models \pi|_{D,a}$ (by Remark \ref{rem: realize over dcl heq}). Then $F(a,b) \models \pi|_{D,a}$ and $g \cdot h = f(a) \cdot f(b) = f(F(a,b))$, hence $g \cdot h \models (f_{\ast}\pi)|_{D,a}$, so $g \cdot h \models (f_{\ast}\pi)|_{D,g}$ (by Remark \ref{rem: realize over dcl heq} again). This implies that $g \cdot (f_{\ast}\pi) = f_{\ast}\pi$. And for an arbitrary $g \in G$ we have $g = g_1 \cdot g_2^{-1}$ for $g_1, g_2 \models (f_{\ast} \pi)|_{C}$. By Remark \ref{rem: translates of partial types} it follows that $g \cdot (f_{\ast}\pi) = f_{\ast}\pi$, as wanted.
\end{proof}

\begin{theorem}\label{thm: equiv of cats groups vs chunks}
	Let $G = (P, \cdot_G)$ and $H = (Q, \cdot_H)$ with $P = X/E, Q = Y/F$ be groups hyperdefinable over $C$. Let $\pi$ be a left-$G$-invariant partial type concentrating on $G$ and type-definable over $C$. Let $f: P \to Q$ be a $C$-type-definable partial function defined on $\pi|_{C}$ and such that for all $a \models \pi|_{C}, b \models \pi|_{C, a}$ we have $f(a \cdot_G b) = f(a) \cdot_H f(b)$. Then there exists a unique type-definable over $C$ group homomorphism $f': G \to H$ that agrees with $f$ on realizations of $\pi|_{C}$. Moreover, if $f$ is injective on $\pi|_{C}(\cU)$, then $f'$ is also injective.
\end{theorem}
\begin{proof}
	We follow closely the proof of \cite[Proposition 3.16]{hrushovski2019valued}, but provide some details.
	
	First note that $\pi|_{C}(\cU)$ generates $G(\cU)$, which clearly implies that there can only exist a unique group homomorphism extending $f$. Indeed, given $g \in G(\cU)$, let $e \models \pi|_{C,g}$ in $G(\cU)$. Then $g \cdot e \models \pi|_{C,g}$ by $G$-invariance of $\pi$. And $g = (g \cdot e) \cdot e^{-1}$. (Note that we do not claim that $a^{-1} \models \pi|_C$ --- this would follow if we assumed additionally that $\pi$ $\otimes$-commutes with itself.)

		To show existence, for any $g \in G(\cU)$ we define $f'(g) := f(a) \cdot f(b)^{-1}$ for some $a,b \models \pi|_{C}$ with $g = a \cdot b^{-1}$. To show that $f'$ is well-defined, it suffices to show that for all $a,b,c,d \models \pi|_{C}$ with $a \cdot b^{-1} = c \cdot d^{-1}$ we have $f(a) \cdot f(b)^{-1} = f(c) \cdot f(d)^{-1}$. As $H$ is a group, it suffices to show that for any $e \models \pi|_{C, a,b,c,d}$ we have $f(a) \cdot f(b)^{-1} \cdot f(e) = f(c) \cdot f(d)^{-1} \cdot f(e)$. 
	In particular $e \models \pi|_{C,a,b,b^{-1}}$ (by Remark \ref{rem: realize over dcl heq}), hence $b^{-1} \cdot e \models \pi|_{C,a,b}$ by invariance of $\pi$. Then, by assumption on $f$, we have $f(e) = f( (b \cdot b^{-1}) \cdot e) = f( b \cdot (b^{-1} \cdot e)) = f(b) \cdot f(b^{-1} \cdot e)$, hence $f(b)^{-1} \cdot f(e) = f(b^{-1} \cdot e)$; and $f(a) \cdot f(b^{-1} \cdot e) = f(a \cdot b^{-1} \cdot e)$. So $f(a) \cdot f(b)^{-1} \cdot f(e) = f(a \cdot b^{-1} \cdot e)$.
	 Similarly, $f(d)^{-1} \cdot f(e) = f(d^{-1} \cdot e)$ and $f(c) \cdot f(d^{-1} \cdot e) = f(c \cdot d^{-1} \cdot e)$, so $f(c) \cdot f(d)^{-1} \cdot f(e) = f(c \cdot d^{-1} \cdot e)$. As $a \cdot b^{-1} \cdot e = c \cdot d^{-1} \cdot e$ by assumption, the required equality holds.
	 
	 We show that $f'$ is a homomorphism. First, for any $g,h \in G(\cU)$ we have $f'(g \cdot h^{-1}) = f'(g) \cdot (f'(h))^{-1}$. Indeed, let $e \models \pi|_{C, g, h}$. Then, using definition of $f'$ and that $g\cdot e, h \cdot e \models \pi|_{C}$ by $G$-invariance of $\pi$, we have $f'(g) = f(g \cdot e) \cdot f(e)^{-1}$, $f'(h) = f(h \cdot e) \cdot f(e)^{-1}$ and $f'(g \cdot h^{-1}) = f'(g \cdot (e \cdot e^{-1}) \cdot h^{-1}) = f'((g \cdot e) \cdot (h \cdot e)^{-1}) = f( g\cdot e) \cdot f(h \cdot e)^{-1}$. Combining $f'(g) \cdot f'(h)^{-1} = f(g \cdot e) \cdot f(e)^{-1} \cdot f(e) \cdot f(h \cdot e)^{-1} = f'(g) \cdot f'(h)$. Second, we have $f'(g^{-1}) = f'(g)^{-1}$ for any $g \in G(\cU)$. Indeed, if $g = a \cdot b^{-1}$ for some $a,b \models \pi|_{C}$, then by definition of $f'$ we have $f'(g^{-1}) = f'( (a \cdot b^{-1})^{-1}) = f'(b \cdot a^{-1}) = f(b) \cdot f(a)^{-1} = \left(f(a) \cdot f(b)^{-1} \right)^{-1} = f'(a \cdot b^{-1})^{-1} = f'(g)^{-1}$. Finally, for any $g,h \in G(\cU)$ using these two remarks we have $f'(g \cdot h) = f'(g \cdot (h^{-1})^{-1}) = f'(g) \cdot f'(h^{-1})^{-1} = f'(g) \cdot (f'(h)^{-1})^{-1} = f'(g) \cdot f'(h)$.

	 Assume that $f$ is injective on $\pi|_{C}(\cU)$, and let $f'(a) = f'(b)$ for some $a,b \in G$. Let $e \models \pi|_{C, a, b}$, then $a = (a \cdot e) \cdot e^{-1}, b = (b \cdot e) \cdot e^{-1}$ and $f'(a) = f(a \cdot e) \cdot f(e)^{-1}, f'(b) = f
	 (b \cdot e) \cdot f(e)^{-1}$. Then $f(a \cdot e) = f(b \cdot e)$, so $a \cdot e = b \cdot e$ by injectivity of $f$, hence $a = b$. Hence $f'$ is injective as well.

	 Now for $C$-type-definability of $f'$. Given any $a \in G, b \in H$, let $e \models \pi|_{C,a,b}$ be arbitrary. Then $a = ( a \cdot e) \cdot e^{-1}$ and, 	 by definition of $f'$, $f'(a) = f(a \cdot e) \cdot f(e)^{-1}$. That is, $f'(a) = b$ if and only if $f(a \cdot_{G} e) = b \cdot_{H} f(e)$. In other words, if and only if
	  \begin{gather}
	 	\pi(x) \vdash \label{eq: explicit formula for group iso}\\ 
	 	\exists u, u', v, v' \Big( \Gamma_{\cdot_G}(a,x,u) \land \Gamma_f(u,u') \land \Gamma_f(x,v) \land \Gamma_{\cdot_H}(b, v, v') \land F(u',v') \Big) \nonumber,
	 \end{gather}
	where on the right we have a small partial type over $C$ (invariant with respect to $E,F$ in the relevant coordinates). By type-definability of $\pi$ over $C$, the set of pairs $(a,b) \in P \times Q$ for which this holds is type-definable over $C$.
\end{proof}

In particular this implies that the group $G$ in Theorem \ref{thm: hyperdef group chunk gives group} is unique, up to a type-definable isomorphism:
\begin{corollary}\label{cor: unique group from group chunk}
	Let $(\pi,F,H,K)$ be a group chunk over $C$ on $P$, and for $i \in \{1,2\}$ let $G_i$ be $C$-hyper-definable groups and $C$-type-definable partial functions $f_i: \pi \to G_i$ defined and injective on $\pi$, so that for any $(a,b) \models \pi^{\otimes 2}|_{C}$ we have $f_i(F(a,b)) = f_i(a) \cdot_{G} f_i(b)$ and the global partial type $(f_i)_{\ast} \pi$ concentrates on $G_i$ and is left-$G_i$-invariant. Then there is  $C$-type-definable group isomorphism $\iota: G_1 \to G_2$.
\end{corollary}
\begin{proof}

	By assumption $(f_i)_{\ast}\pi$ concentrates on $G_i$ and is left-$G_i$-invariant, and $C$-type-definable (Remark \ref{rem: pushfoward preserves type-def}). 
	By Remark \ref{rem: pushforward sets of realiz} and injectivity of $f_i$ on $\pi$, 	$g_1 := f_2 \circ f_1^{-1} : (f_1)_{\ast}\pi \to (f_2)_{\ast}\pi$ and $g_2 := f_1 \circ f_2^{-1} : (f_2)_{\ast}\pi \to (f_1)_{\ast}\pi$ are  $C$-type-definable injections with $g_1 \circ g_2 = \id_{(f_2)_{\ast}\pi}, g_2 \circ g_1 = \id_{(f_1)_{\ast}\pi}$; and for any $(a,b) \models ((f_i)_{\ast}\pi)^{\otimes 2}|_{C}$, $g_i(a \cdot_{G_i} b) = g_i(a) \cdot_{G_{3-i}} g_i(b)$. By Theorem \ref{thm: equiv of cats groups vs chunks}, there exist unique $C$-type-definable injective group homomorphisms $h_i: G_i \to G_{3-i}$ with $h_i \restriction_{(f_i)_{\ast}\pi } = g_i$. As (by Theorem \ref{thm: equiv of cats groups vs chunks}) $h_i \circ h_{3-i}$ is the unique  isomorphism of $G_i$ extending $g_i \circ g_{3-i}$, $h_i \circ h_{3-i} = \id_{G_{3-i}}$, hence $h_i: G_i \to G_{3-i}$ is a group isomorphism.
\end{proof}

\section{Externally definable fsg groups}\label{sec: Externally definable fsg groups}

\subsection{Type-definable filters associated to definable measures}\label{sec: Type-definable filters associated to definable measures}

\begin{definition}\label{def: pi mu}
	Assume that $\mu \in \mathfrak{M}_x(\cU)$ is a global Keisler measure. We define the partial type $\pi_{\mu} := \{ \varphi(x) \in \mathcal{L}_x(\cU) : \mu(\varphi(x)) = 1 \} = \{ \neg \varphi(x) \in \mathcal{L}_x(\cU) : \mu(\varphi(x)) = 0 \}$ (note that $\pi_\mu$ is closed under conjunction and implication). \end{definition}
	
\begin{remark}\label{rem: pi mu defines support}
	We have $S(\mu) = \{ p \in S_x(\cU) : \pi_{\mu} \subseteq p \}$ and for any $C \subseteq \cU$, $a \models \pi_{\mu}|_{C} \iff \tp(a/C) \in S(\mu|_{C})$.
\end{remark}

The following are straightforward adaptations of the usual notions from complete to partial types.
\begin{definition}\label{def: props of partial types}
Given $A \subseteq \cU$, a partial type $\pi(x)$ over $\cU$ is
\begin{enumerate}
\item \emph{invariant over $A$} if for every formula $\varphi(x,y)$ and $b \in \cU^y$, we have $\varphi(x,b) \in \pi \iff \varphi(x, \sigma(b)) \in \pi$.
	\item \emph{(type-)definable over $A$} if  for every formula $\varphi(x;y) \in L(\emptyset)$, the set $\{ b \in \cU^y : \varphi(x,b) \in \pi \}$ is (type-)definable over $A$.
	\item \emph{finitely satisfiable in $A$} if every formula $\varphi(x) \in \pi$ has a realization in $A$. And $\pi(x)$ is \emph{strongly finitely satisfiable in $A$} if for every $\varphi(x) \in L(\cU)$ such that $\varphi(A) = \emptyset$, $\neg \varphi(x) \in \pi$.
\end{enumerate}

\end{definition}

\begin{remark}
	Note that if $\pi(x)$ is a partial type over $\cU$ type-definable over $A$ and closed under implication, then viewed as a partial type on the home sort over $\cU^{\heq}$ (Remark \ref{rem: real partial type in heq}), it remains type-definable over $A$ in the sense of Definition \ref{def: def type hyperim}.
\end{remark}

\begin{remark}\label{rem: pi mu type-def for mu def}
If $\mu$ is invariant (definable, finitely satisfiable) over $A$, then  $\pi_\mu$ is invariant (type-definable, strongly finitely satisfiable) over $A$.  Indeed, if the measure  $\mu$ is definable over $A$, then for every formula $\varphi(x,y) \in L$, the set $\{b \in \cU^{y} : \mu( \varphi(x,b))=1\}$ is type-definable over $A$.
\end{remark}

\begin{lemma}\label{lem: props of pi mu}

	\begin{enumerate}
	\item If $\mu_x, \nu_y \in \mathfrak{M}(\cU)$ are definable measures, then $\pi_{\mu} \otimes \pi_{\nu} = \pi_{\mu \otimes \nu}$. In particular, if $\mu$ is generically stable then  $\pi_{\mu}$ commutes with itself.
	\item For $\mu \in \mathfrak{M}_x(\cU)$  and $f$ a definable map, we have $f_{\ast} \pi_{\mu} = \pi_{f_{\ast} \mu}$.
		\item If $G$ is a type-definable group and $\mu \in \mathfrak{M}_{G}(\cU)$, then $\pi_{\mu}(x) \vdash G(x)$. And if  $\mu$ is $G$-invariant, then $\pi_{\mu}$ is also $G$-invariant.
	\end{enumerate}

\end{lemma}
\begin{proof}
(1) Given a formula  $\varphi(x,y) \in L(\cU)$, let $M$ be a small model such that $\mu_x,\nu_y, \varphi(x,y)$ are $M$-definable. 
  By definition of $\otimes$, we have 
	$\varphi(x,y) \in \pi_{\left(\mu \otimes \nu \right)} \iff \mu \otimes \nu \left( \varphi(x,y) \right) = 1 \iff \int_{S(\mu|_{M})} F^{\varphi(x,y)}_{\nu, M} (p)  d \mu_{M} = 1$, where $\mu|_{M} \in \mathfrak{M}_x(M)$ is the restriction of $\mu$ to $M$, $S(\mu|_{M})$ is its support,  $\mu_M$ is the unique regular Borel probability measure extending $\mu$ on $S_x(M)$. As the function $F^{\varphi(x,y)}_{\nu, M}: p \in S_x(M) \mapsto \nu(\varphi(a,y)) \in [0,1]$ for some/any $a \models p$ is continuous by definability of $\nu$, this integral is $=1$ if and only if $F^{\varphi(x,y)}_{\nu, M}(p) = 1$ for all $p \in S(\mu|_M)$. That is, if and only if for every $a \models \pi|_{M} $ and $b \models \pi|_{M,a}$ we have $\models \varphi(a,b)$ (by Remark \ref{rem: pi mu defines support}); if and only if $\varphi(x,y) \in \pi_{\mu} \otimes \pi_{\nu}$ (by  Remark \ref{rem: prod of partial types realiz}).
	
	(2) Given $\varphi(x) \in L(\cU)$ and a definable map $f$, by Definition \ref{def: pushforward of types}, we have $\varphi(x) \in f_{\ast} \pi_{\mu} \iff \varphi(f(x)) \in \pi_{\mu} \iff \mu(\varphi(f(x))) = 1$, by Definition \ref{def: definable pushforward} $\iff f_{\ast} \mu(\varphi(x)) = 1 \iff \varphi(x) \in \pi_{(f_{\ast \mu})}$.
	
	(3) If $\mathfrak{M}_{G}(\cU)$, then by definition $\pi_{\mu}(x) \vdash G(x)$.  And assuming $\mu$ is $G$-invariant, for every $g \in G(\cU)$ we have (see Definition \ref{rem: translates of partial types}): 
	$g \cdot \pi_{\mu} = (t_g)_{\ast} \pi_{\mu} = \pi_{(t_g)_{\ast} \mu} = \pi_{g \cdot \mu} = \pi_{\mu}$ by (2), where $t_g: G \to G$ is the definable map $h \in G \mapsto g \cdot h$.
	\end{proof}
	
\begin{remark}
	In the context of Lemma \ref{lem: props of pi mu}(1), we have $S(\mu) \otimes S(\nu)$ is a dense subset of $S(\mu \otimes \nu)$ (similarly to \cite[Proposition 4.3]{chernikov2022definable}), and $\pi_{\mu}\otimes \pi_{\nu}$ is a partial type, hence corresponds to a closed set of types.
\end{remark}

\begin{remark}
In relation to Lemma \ref{lem: props of pi mu}(3), note that if $G$ is a (type-)definable group and $\mu \in \mathfrak{M}_{G}$ is a definable measure, then its (left) stabilizer $\Stab_{G}(\mu) = \{ g \in G(\cU) : g \cdot \mu = \mu \}$ (Definition \ref{def: stab of a measure}) is a type-definable subgroup of $G$ (Fact \ref{fac: stab of def meas type def}). Similarly, for a partial type $\pi(x) \vdash G(x)$, we can consider its (left) stabilizer $\Stab_{G}(\pi) = \{ g \in G(\cU) : g \cdot \pi = \pi \}$. Then, using Remark \ref{rem: pi mu defines support}, $\Stab_{G}(\mu) \subseteq \Stab_{G}(\pi_{\mu}) = \{g \in G(\cU) : g \cdot S(\mu) = S(\mu) \}$ is a subgroup of $G$ and $\pi_{\mu}$ is type-definable (Remark \ref{rem: pi mu type-def for mu def}), however $ \Stab_{G}(\pi_{\mu})$ need not be type-definable. 
	
Indeed, let $M$ be a countable strongly minimal expansion of a (necessarily abelian) group $G(M)$, $\cU \succ M$ a monster model. Consider the Keisler measure $\mu := \sum_{q \in M} \alpha_q \cdot \delta_{q}$ for some \emph{pairwise-distinct} $\alpha_q \in (0,1), \sum_{q \in M} \alpha_q = 1$, then $\mu \in \mathfrak{M}_{G}(\cU)$ is smooth (over $M$). Note that $S(\mu) = \{\tp(q/\cU) : q \in M\} \cup \{p\}$, where $p \in S_{G}(\cU)$ is the unique (by strong minimality) non-algebraic type. Note that $\Stab_{G}(\mu) = \{ 0 \}$: given any $g \in G(\cU) \setminus \{0\}$, take any $q \in M$, then  $0 < (g \cdot \mu)(x = q) = \mu(x = g^{-1} \cdot  q )$,  and as $g^{-1} \cdot q \neq q$ we conclude $(g \cdot \mu)(x = q) \neq \mu(x = q)$ using that $\alpha_q$ are pairwise distinct.  And $\pi_{\mu} = \{ \varphi(x) \in L(\cU) : \models \bigwedge_{q \in M} \varphi(q) \}$, hence  $M \subseteq \Stab_G (\pi_{\mu}) $ as $M$ is a subgroup. In fact,  $M = \Stab_G (\pi_{\mu}) $: given any $g \in G(\cU) \setminus M$ and $q \in M$, $g \cdot q \notin M$, hence  $\mu(x \neq g \cdot q) = 1$. But $g \cdot \mu( x \neq g \cdot q) = \mu( g \cdot x \neq g \cdot q) = \mu(x \neq q) \leq 1 - \alpha_q < 1$. In particular $\Stab_G (\pi_{\mu})$ is not type-definable, as a countable subset of $\cU$.
\end{remark}

\begin{corollary}\label{cor: measure theoretic group chunk}
Let $T$ be NIP, $G$ an $M$-type-definable group, and $\mu \in \mathfrak{M}_{G}(\cU)$ an $M$-definable left-invariant measure (such $M$ and $\mu$ exist  whenever $G$ is a definably amenable NIP group, Fact \ref{fac: def am def measure}). Then $G$ can be recovered, up to an $M$-type-definable isomorphism of $M$-hyperdefinable groups in $\cU^{\heq}$, from its $M$-type-definable group chunk $(\pi_{\mu}, F,H,K)$ (Remark \ref{rem: group chunk from a group}).
\end{corollary}
\begin{proof}
	The partial type $\pi_{\mu}$ concentrates on $G$, is type-definable over $M$ (Remark \ref{rem: pi mu type-def for mu def}) and is left $G$-invariant (Lemma \ref{lem: props of pi mu}). Then Theorems \ref{thm: hyperdef group chunk gives group} and Corollary \ref{cor: unique group from group chunk} apply.
\end{proof}

\begin{proposition}\label{prop: fsg iff invariant filter str fin sat}
	($T$ NIP) Assume $G(x)$ is a type-definable group over $M$ and $\pi(x)$ is a filter over $\cU$ so that $\pi(x) \vdash G(x)$, $\pi$ is left-$G(\cU)$-invariant and strongly finitely satisfiable in $G(M)$. Then $G$ is fsg and $\pi = \pi_\mu$ for the unique left-invariant measure $\mu \in \mathfrak{M}_G(\cU)$. 
\end{proposition}
\begin{proof}
	Let $p \in S_{G}(\cU)$ be a complete type extending $\pi$, by $G(\cU)$-invariance of $\pi$ we have $g \cdot p \supseteq \pi$ for every $g \in G(\cU)$. By strong finite satisfiability of $\pi$ in $G(M)$ every type extending $\pi$ is finitely satisfiable in $G(M)$, hence $p$ witnesses that $G$ is fsg (over $M$). 
		In particular $p$ is generic. Let $\mu$ be the unique left and right-invariant measure on $G$ (Fact \ref{fac: fsg groups basic props}). As $S(\mu) = \overline{G \cdot p}$ (by \cite[Theorem 3.36, Remark 3.37, Proposition 3.31]{chernikov2018definably}), we get $p'(x) \vdash \pi(x)$ for every $p' \in S(\mu)$. Hence $\pi \subseteq \pi'_{\mu}$. 
		Conversely, assume $\varphi(x) \in \pi_{\mu}$. Then $\mu(\neg \varphi(x)) = 0$, so $\neg \varphi(x)$ is not right generic (by right-invariance of $\mu$). Hence $\neg \varphi(g \cdot x) \cap G(M) = \emptyset$ for some $g \in G(\cU)$. (Indeed, assume $X$ is (relatively) definable and $g \cdot X \cap G(M) \neq \emptyset$ for all $g \in G(\cU)$. It follows by compactness that there are $n \in \omega$ and $h_i \in G(M), i \in [n]$ so that for every $g \in G(\cU)$, $h_i \in g \cdot X$ for some $i \in [n]$. Then, given any $g \in G(\cU)$, $h_i \in g^{-1} \cdot X$ for some $i \in [n]$, hence $g \in  X \cdot h_i^{-1}$. This shows that $X$ is right generic.)
		Hence $\varphi(g \cdot x) \in \pi$ by strong finite satisfiability of $\pi$, hence $\varphi(x) \in \pi$ by left-$G(\cU)$-invariance of $\pi$.
	
	Conversely, if $G$ is fsg over $M$ witnessed by $\mu$, then $\pi_{\mu}$ is strongly finitely satisfiable ($\varphi(M) = \emptyset$ implies $\mu(\varphi(x)) = 0$, so $\mu(\neg \varphi(x)) = 1$) and $G(\cU)$-invariant.
\end{proof}

\subsection{An eliminable hyperdefinable group in a reduct}\label{sec: eliminable hyper-def group}

We show that if a hyperdefinable group is type-definably isomorphic to a definable group in some expansion of the theory, it is already definable in $\cU^{\eq}$:

\begin{proposition}\label{prop: elim hyperdef group}
	Let $T \subseteq T'$ be complete theories and $\cU \models T'$ a monster model. Let $H = (X,\gamma(x,y,z))$ be a definable group in $T'$, $G = (Y/F, \Gamma(x,y,z))$ a hyperdefinable group in $T$ (with $Y$ a partial type in finitely many variables), and assume $\iota(x,y)$ is a type-definable in $T'$ isomorphism between $H(\cU)$ and $G(\cU)$, all over a small set of real parameters $C \subseteq \cU$. Then there is a group $G_0$ definable over $C$ in $T^{\eq}$ and a $C$-definable in $T'$ isomorphism $\iota'$ between $H$ and $G_0$.  
\end{proposition}

\begin{proof}
	Without loss of generality $F(x,y)$ is a $C$-type-definable in $T$ equivalence relation on the whole universe $\cU^x$, $\Gamma(x,y,z) \vdash Y(x) \land Y(y) \land Y(z)$ is $C$-type-definable in $T$ and $F$-invariant in every coordinate, i.e.
	$$\Gamma(x,y,z) \land F(x,x') \land F(y,y') \land F(z,z') \vdash \Gamma(x',y',z'),$$
	and $\iota(x,y) \vdash X(x) \land Y (y)$ is $C$-type-definable in $T'$ and $F$-invariant on $y$, i.e.~$\iota(x,y) \land F(y,y') \vdash \iota(x,y')$.
	
	As $\iota$ in particular defines an injection from $X$ to $Y/F$, $\iota(x,y) \land \iota(x',y) \vdash x = x'$, then by compactness there is finite $\iota_0 \subseteq \iota$ so that  
	\begin{gather}
		\iota_0(x,y) \vdash X(x) \textrm{ and } \iota_0(x,y) \land \iota_0(x',y) \vdash x = x'.\label{eq: elim hyperdef 1}
	\end{gather}
	As $\iota$ in particular defines a surjection from $X$ onto $Y/F$, we have 
 	\begin{gather}
		Y(y') \vdash \exists x \  \iota (x,y'). \label{eq: elim hyperdef 2}
	\end{gather}
	
	We claim that the equivalence
	\begin{gather}
		\left( \iota_0(x,y) \land Y(y) \right) \leftrightarrow \iota(x,y)\label{eq: elim hyperdef 3}
	\end{gather}
	of partial types holds. 
	Right to left is by the choice of $\iota$ and $\iota_0$. Conversely, assume  that $\iota_0(x,y) \land Y(y) $ holds. By \eqref{eq: elim hyperdef 2}, $Y(y) \vdash \exists x' \iota(x',y)$. And by  \eqref{eq: elim hyperdef 1}, $\iota_0(x,y) \land \iota(x',y) \vdash x = x'$, so combining we get that $\iota(x,y)$ holds.
	
	Again, as $\iota$ is $F$-invariant on $y$ and defines a bijection between $X$ and $Y/F$, we have (using \eqref{eq: elim hyperdef 3} for the second part):
	\begin{gather*}
		Y(y_1) \land Y(y_2)  \land F(y_1,y_2) \vdash \exists x \left( \iota(x,y_1) \land \iota(x,y_2) \right) \vdash \exists x( \iota_0(x,y_1) \land \iota_0(x,y_2)),\\
	\textrm{ and }	Y(y_1) \land Y(y_2)  \land \exists x( \iota_0(x,y_1) \land \iota_0(x,y_2)) \vdash \\ \vdash 	Y(y_1) \land Y(y_2)  \land (\exists x \left( \iota(x,y_1) \land \iota(x,y_2) \right) \vdash F(y_1,y_2).
	\end{gather*}
	By compactness this implies that there is some finite $F_0 \subseteq F$ so that 
	\begin{gather}
		Y(y_1) \land Y(y_2) \vdash F(y_1, y_2) \leftrightarrow F_0(y_1,y_2). \label{eq: elim hyperdef 4}
	\end{gather}

	Similarly, using that $\iota$ defines a group isomorphism and assumption on $\Gamma$, the following partial types are equivalent:
	\begin{gather*}
		 \Gamma(y_1,y_2, y_3) \leftrightarrow \Big( Y(y_1) \land Y(y_2) \land Y(y_3) \land \\
		 \exists x_1, x_2, x_3 \  \left( \iota(x_1, y_1) \land \iota(x_2, y_2) \land \iota (x_3, y_3) \land \gamma(x_1,x_2,x_3) \right) \Big).
		\end{gather*}
	By \eqref{eq: elim hyperdef 3} we then have 
	\begin{gather*}
		 \Gamma(y_1,y_2, y_3) \leftrightarrow \Big( Y(y_1) \land Y(y_2) \land Y(y_3) \land \\
		 \exists x_1, x_2, x_3 \  \left( \iota_0(x_1, y_1) \land \iota_0(x_2, y_2) \land \iota_0 (x_3, y_3) \land \gamma(x_1,x_2,x_3) \right) \Big).
		\end{gather*}
		
	By compactness this implies that there is some finite $\Gamma_0 \subseteq \Gamma$ so that 
	\begin{gather}
		Y(y_1) \land Y(y_2) \land Y(y_3) \vdash \Gamma(y_1, y_2, y_3) \leftrightarrow \Gamma_0(y_1,y_2, y_3). \label{eq: elim hyperdef 5}
	\end{gather}

	By assumption on $F$ and $\Gamma$, \eqref{eq: elim hyperdef 4}
 and \eqref{eq: elim hyperdef 5} we have
	\begin{gather*}
		Y(y_1) \land Y(y_2) \land Y(y_3) \vdash \Big( F_0(y_1,y_1) \land \left( F_0(y_1, y_2) \rightarrow F_0(y_2,y_1) \right) \land  \\
		\left( F_0(y_1, y_2) \land F_0(y_2, y_3) \rightarrow F_0(y_1, y_3) \right) \land \left( \exists x \  \iota_0(x,y_1) \right) \land \\
		\left( \forall x  \ (\iota_0(x,y_1) \land \iota_0(x,y_2) \rightarrow F_0(y_1, y_2)) \right) \land\\
		 \big( \Gamma_0(y_1, y_2, y_3) \leftrightarrow \\
		 \exists x_1, x_2, x_3 \left( \iota_0(x_1, y_1) \land \iota_0(x_2,y_2) \land \iota_0(x_3, y_3) \land \gamma(x_1,x_2,x_3) \right) \big)\Big).
	\end{gather*}
	
	By compactness there is some finite $Y_0 \subseteq Y$ so that $F_0$ defines an equivalence relation on $Y_0(\cU)$, $\iota_0(x,y)$ (also using \eqref{eq: elim hyperdef 1}) defines a bijection between $X$ and $Y_0/F_0$, and for any $x_i \in X$ and $y_i \in Y_0$ with $\models \iota_0(x_i,y_i)$  we have $\gamma(x_1,x_2,x_3) \leftrightarrow \Gamma_0(y_1, y_2, y_3)$. Hence $G_0 := (Y_0/F_0, \Gamma_0)$ is a $C$-definable group in $T^{\eq}$ and $\iota_0$ is a $C$-definable in $T'$ isomorphism between the groups $H$ and $G_0$.
\end{proof}

\subsection{Externally definable fsg groups are isomorphic to definable ones}
\label{sec: Externally definable fsg groups are isomorphic to definable ones}
First we give an example of an externally definable fsg (and even generically stable group) in an NIP theory that is not definable.

\begin{example}\label{ex: ext def fsg not def}
	Let $\cU = (K,\Gamma,k)$ be a monster model of the theory of algebraically closed fields.

	Let $G := (K,+)$ be the additive group of the field. For $a \in K$, $\alpha \in \Gamma$ and $\square \in \{\geq, >\}$, let $B_{\square \alpha}(a) := \{ x \in K : v(x-a) \square \alpha \}$ be the closed (respectively, open) valuational ball of radius $\beta$ with center $b$. Given a ball $B_{\square \alpha}(a)$, by quantifier elimination in $\ACVF$, density of $\Gamma$ and the fact that for any two balls, either one is contained in the other or they are disjoint, the set of formulas 
	\begin{gather*}
		\{ x \in B_{\square \alpha}(a)\} \cup \{ x \notin B_{\square \beta}(b) : b \in K, \beta \in \Gamma, B_{\square \beta}(b) \subsetneq B_{\square \alpha}(a) \}
	\end{gather*}
	determines a complete type $p_{\square \alpha, a}(x) \in S_G(\cU)$, called the generic type of $B_{\square \alpha}(a)$. 
	
	Fix some $\alpha \in \Gamma$ and let $B := B_{\geq \alpha}(0)$ and $p := p_{\geq \alpha, 0}$. For any $a,b \in K$ we have $a + B_{\square \beta}(b) = B_{\square \beta}(a+b)$, hence for any $b \in B$ we have $b + B = B$, and if $B' \subsetneq B$ then $b + B' \subsetneq B$ for any ball $B'$. It follows that $B$ is a subgroup of $G(\cU)$ and $p \in S_B(\cU)$ is left-$B(\cU)$-invariant. Moreover, $p$ is generically stable over any $a \in K$ with $v(a) = \alpha$ (as it is internal to the residue field; or can see directly that $p$ commutes with itself), hence $B$ is a generically stable group.

	 Choose a small model $M \prec \cU$ with $\Gamma^M \cong (\mathbb{Q}, +, <)$ and $\alpha := 1 \in \Gamma^M$. Then $B(M)$ is an infinite (by density of $\Gamma^M$) group definable in $M$ and $p$ is generically stable over $M$, hence $B(M)$ is also a generically stable group in $\Th(M^{\Sh})$ (by Fact \ref{fac: group props preserved in Sh exp}). Then $X := B_{> \sqrt 2}(0) \cap M$ is an externally definable subset of $M$, $X \subsetneq B(M)$, but it is not definable in $M$ (as by quantifier elimination, every definable subset of $M$ is a Boolean combination of balls with parameters in $M$).
	Let $Y \subseteq (K^M)^2$ be  the $M^{\Sh}$-definable set
	 \begin{gather*}
	 	Y := \{ (0,x) :  x \in X \} \cup \{ (1,x) :  x \in B(M) \setminus X \}.
	 \end{gather*}
	 We equip it with an $M^{\Sh}$-definable group operation by copying it from $B$: for $(t_1,a_1), (t_2,a_2), (t_3,a_3) \in Y$, we define $(t_1,a_1) +_Y (t_2,a_2) = (t_3,a_3)$ if 
	 \begin{gather*}
	 	\left(\bigwedge_{i=1}^3 (t_i = 0 \land a_i \in X) \lor (t_i =1 \land a_i \in B(M) \setminus X) \right) \land (a_1 + a_2 = a_3).
	 \end{gather*}
	
	Then $(Y, +_Y)$ is $M^{\Sh}$-definably isomorphic to $(B(M), +)$ via an isomorphism $f$ defined by
	\begin{gather*}
		f((t,a)) = b :\iff a = b \land \left( ( a \in X \land t = 0) \lor (a \in B(M) \setminus X \land t = 1) \right).
	\end{gather*}
	
	But $Y$ is not definable in $M$ as then $X$ would be definable in $M$ via $X = \{ a \in K^M : (0,a) \in Y\}$.
	 \end{example}
	
\begin{remark}
	Similarly, we can find an externally definable fsg group in an $o$-minimal structure which is not definable --- take the circle group, and split it into two externally definable, non-definable intervals.
\end{remark}

We show that this is the only obstruction:

\begin{theorem}\label{thm: ext def fsg groups is to definable}
		Let $T$ be NIP and $M \models T$. Then a group definable in $M^{\Sh}$ is fsg (in the sense of $\Th(M^{\Sh})$) if and only if it is definably (in $M^{\Sh}$) isomorphic to a group definable in $M^{\eq}$ and fsg in $T$.
\end{theorem}

\begin{proof}
One direction is by Fact \ref{fac: group props preserved in Sh exp},  Remark \ref{rem: def iso preserves fsg} and the fact that every subset of the main sort definable in $(M')^{\eq}$ is already definable in $M'$. We prove the other direction.

We are following the notation in Definition \ref{def: Shelah exp context}, in particular we have $(N,M) \prec^{L_P} (N',M') \prec^{L_P} (N'', M'')$ with each pair saturated over the previous one, and $\widetilde{M}'$ is a monster model for $T' = \Th_{L'}(M^{\Sh})$.

Let $G' = G(M')$ be an $M$-definable fsg group, both in the sense of the $L'$-theory $T' := \Th \left(M^{\Sh} \right)$. Let $\cdot_{G'}$ be given by the partial $L'(M)$-definable function $F'(x,y)$ restricted to $G'(M')$, let $\Gamma_{F'}(x,y,z) \in L'(M)$ denote the graph of $F'$. We may assume that $\Gamma_{F'}(x,y,z)$ defines the graph of a partial function from any two of its coordinates to the third one.
	
By Proposition \ref{prop: fsg iff inv gen stab meas} and Fact \ref{fac: fsg groups basic props} (as $T'$ is NIP by Fact \ref{fac: Sh exp qe}), there is $\bar{\mu}' \in \mathfrak{M}^{L'}_{G'}(M')$ generically stable over $M$ (in $T'$) and such that $g \cdot \bar{\mu}' = \bar{\mu}'$ for all $g \in G'(\cU')$. 
Let $\bar{\nu}'(x,y,z) := h_{\ast} \left( (\bar{\mu}')^{\otimes 2} \right) \in \mathfrak{M}^{L'}_{(G')^{3}}(M')$ for the $L'(M)$-definable (partial) map $h: (x,y) \mapsto (x,y,x \cdot_{G'} y)$. Note that $\bar{\nu}'$ is generically stable over $M$ in $T'$ (by Fact \ref{fac: pushforward gen stab}).

Let $\bar{\mu} := \bar{\mu}'|_{L} \in \mathfrak{M}_{x}^L(M')$. By Proposition \ref{prop: properties of reducts of measures}, $\bar{\mu}$ is generically stable over $M$ (in $T$); and by Lemma \ref{lem: unique extension of def meas to Sh}, $\bar{\mu}'$ is the unique extension of $\bar{\mu}$ to a measure in $\mathfrak{M}_{x}^{L'}(M')$. Let $\bar{\nu} := \bar{\nu}'|_{L} \in \mathfrak{M}_{x,y,z}^L(M')$. Similarly, $\bar{\nu}$ is generically stable over $M$ in $T$, and $\bar{\nu}'$ is its unique extension to a measure in $\mathfrak{M}_{x,y,z}^{L'}(M')$.

Let $M_0 := M$, $N_0 := N$. Applying Fact \ref{fac: honest defs} repeatedly, we can choose $(N_i,M_i)$ by induction on $i \in \omega$ so that: 
\begin{itemize}
	\item $(N_i,M_i) \preceq^{L_P} (N_{i+1},M_{i+1}) \preceq^{L_P} (N',M')$,
	\item  $|M_{i+1}| + |N_{i+1}| \leq |N_i|+|L|$,
	\item for every $\varphi(x,y) \in L$ and $c \in N^y_{i}$ there is $\theta(x,z) \in L$ and $d \in M^{z}_{i+1}$ so that 
\begin{itemize}
\item $\theta(x,d) \land P(x) \vdash \varphi(x,c)$, and
\item no $p \in S^L_x(M')$ invariant over $M_i$ (in $T$) is consistent with $(\varphi(x,c) \setminus \theta(x,d)) \land P(x)$ (in $(N',M')$).
\end{itemize}
\end{itemize}

We let $M_{\omega} := \bigcup_{i \in \omega} M_i$, $N_{\omega} := \bigcup_{i \in \omega} N_i$. By construction it follows:
\begin{itemize}
	\item $(N,M) \preceq^{L_P} (N_{\omega}, M_{\omega}) \preceq^{L_P} (N',M')$ (in particular $M \prec^{L'}  M_{\omega} \prec^{L'} M'$),
	\item $|N_{\omega}|+|M_{\omega}| \leq |N|$ (in particular  $\widetilde{M}'$ is still $|M_{\omega}|^{+}$-saturated),
	\item for every $\varphi(x,y) \in L$ and $c \in N^y_{\omega}$ there is $\theta(x,z) \in L$ and $d \in M^{z}_{\omega}$ so that 
\begin{itemize}
\item $\theta(x,d) \land P(x) \vdash \varphi(x,c)$, and
\item no $p \in S^L_x(M')$ invariant over $M$ (in $T$) is consistent with $(\varphi(x,c) \setminus \theta(x,d)) \land P(x)$ (in $(N',M')$).
\end{itemize}

\end{itemize}

We let $\pi_{\bar{\mu}'}$ and $\pi_{\bar{\nu}'}$ be the partial types over $M'$ in $T'$ (as in Definition \ref{def: pi mu}), they are type-definable over $M$ (in $T'$; Definition \ref{def: props of partial types}), using definability of $\bar{\mu}', \bar{\nu}'$ (Remark \ref{rem: pi mu type-def for mu def}); and $\pi_{\bar{\mu}'}|_{M}(x) \vdash G'(x)$. Similarly, $\pi_{\bar{\mu}}$ and $\pi_{\bar{\nu}}$ are partial types over $M'$, type-definable over $M$ in $T$, and $\pi_{\bar{\mu}} = \pi_{\bar{\mu}'} \restriction_{L}$ and $\pi_{\bar{\nu}} = \pi_{\bar{\nu}'}\restriction_{L}$. Conversely, we have:

\begin{claim}\label{cla: pi mu implies pi mu'}
 $\pi_{\bar{\mu}}|_{M_{\omega}}(x) \vdash \pi_{\bar{\mu}'}|_{M_{\omega}} (x) $, $\pi_{\bar{\nu}}|_{M_{\omega}}(x,y,z) \vdash \pi_{\bar{\nu}'}|_{M_{\omega}} (x,y,z) $ and $\pi_{\bar{\mu}^{\otimes n}}|_{M_{\omega}}(x) \vdash \pi_{(\bar{\mu}')^{\otimes n}}|_{M_{\omega}}(x)$ for all $n \in \omega$ (in $T'$).
\end{claim}
\begin{proof}
Assume $\varphi(x) \in L'(M_{\omega})$ is such that $\varphi(x) \in \pi_{\bar{\mu}'}|_{M_{\omega}}$. We have $\varphi(x) \leftrightarrow^{T'} R_{\psi(x,y_1;d_2)}(x,d_1)$ for some $\psi(x,y_1,y_2) \in L, d_2 \in N, d_1 \in M_{\omega}$ (Remark \ref{rem: Shelah QE formulas}). And, by definition of $\widetilde{M}'$, $(N',M') \models R_{\psi(x,y_1;d_2)}(x,y_1) \leftrightarrow (P(x) \land P(y_1) \land \psi(x,y_1;d_2))$. 
By the choice of $M_{\omega}$ there is some $\theta(x; z) \in L$ and $e \in M_{\omega}^{z}$ so that $(N',M') \models (\theta(x,e) \land P(x)) \rightarrow \psi(x, d_1, d_2)$ and no $p \in S^L_x(M')$ invariant over $M$ (in $T$) is consistent with $(\psi(x,d_1,d_2) \setminus \theta(x,e)) \land P(x)$. So $\theta(x,e) \vdash \varphi(x)$  and no $p \in S^L_x(M')$ invariant over $M$ (in $T$) is consistent with $\varphi(x) \setminus \theta(x,e)$, both in $T'$.

Let $\bar{p}' \in S(\bar{\mu}')$ be arbitrary. Then $\bar{p}:= \bar{p}'\restriction_{L} \in S(\bar{\mu})$. The measure $\bar{\mu}$ is $M$-invariant in $T$ (by generic stability over $M$), hence $\bar{p} \in S(\bar{\mu})$ is also $M$-invariant in $T$ (by Fact \ref{fac: types in sup of inv meas are inv NIP}). It follows that  $(\varphi(x) \setminus \theta(x,e)) \notin p'$. Hence $\bar{\mu}' \left( \varphi(x) \setminus \theta(x,e) \right) = 0$, so $\bar{\mu}'(\theta(x,e)) = \bar{\mu}(\theta(x,e)) = 1$, so $\theta(x,e) \in \pi_{\bar{\mu}}|_{M_{\omega}}$.

Note that the argument applies to any measure in $T'$ generically stable over $M$, in particular to $(\bar{\mu}')^{\otimes n}, n \in \omega$ and $\bar{\nu}'$.
\end{proof}

Note that $\Gamma_{F'}(x,y,z) \in \pi_{\bar{\nu}'}|_{M_{\omega}}$. Indeed, by definition of pushforward  measure $\bar{\nu}'(\Gamma_{F'}(x,y,z)) =  (\bar{\mu}')^{\otimes 2}  (\Gamma_{F'}(h(x,y))) = (\bar{\mu}')^{\otimes 2}  (\Gamma_{F'}(x,y, x \cdot_{G'} y)) \geq (\bar{\mu}')^{\otimes 2}(H(x) \land H(y)) = 1$.
Then, by Claim \ref{cla: pi mu implies pi mu'}, there is some $L(M_{\omega})$-formula $\Gamma(x,y,z) \in \pi_{\bar{\nu}}|_{M_{\omega}}$ so that $\Gamma(x,y,z) \vdash \Gamma_{F'}(x,y,z)$ (in $T'$). From this and assumption on $\Gamma_{F'}$, we still have that $\Gamma(x,y,z)$ defines (in $T$) the graph of a partial function from any two of its coordinates to the third one. We consider the partial $L(M)$-definable (in $T$) functions $F$ with graph defined by $\Gamma_{F}(x,y,z) := \Gamma(x,y,z)$, $H$ with $\Gamma_{H}(x,y,z) := \Gamma(x,z,y)$ and $K$ with $\Gamma_{K}(x,y,z) := \Gamma(z,x,y)$.

%
%

\begin{claim}\label{cla: proof of group chunk in the reduct}
We claim that $\left( \pi_{\bar{\mu}}(x), F, H,K \right)$ 	is a group chunk over $M_{\omega}$ in $T$, in the sense of Definition \ref{defn: group chunk} (on the home sort).
\end{claim}
\begin{proof}
We already know that $\pi_{\bar{\mu}}(x)$ is a partial type over $M'$ closed under implication and type-definable over $M_{\omega}$, all in $T$.

By Lemma \ref{lem: props of pi mu} we have $(\pi_{\bar{\mu}})^{\otimes 2} = \pi_{({\bar{\mu}}^{\otimes 2})}$ and $(\pi_{\bar{\mu}'})^{\otimes 2} = \pi_{((\bar{\mu}')^{\otimes 2})}$. By Claim \ref{cla: restriction commutes with tensor} we have $(\bar{\mu}')^{\otimes 2} \restriction_{L} = (\bar{\mu}' \restriction_{L})^{\otimes 2}  = \bar{\mu}^{\otimes 2}$.
So, using Claim  \ref{cla: pi mu implies pi mu'},
$$\pi_{\bar{\mu}}^{\otimes 2}|_{M_{\omega}} = \pi_{({\bar{\mu}}^{\otimes 2})} |_{M_{\omega}}  \vdash \pi_{((\bar{\mu}')^{\otimes 2})}|_{M_{\omega}} = \pi_{\bar{\mu}'}^{\otimes 2}|_{M_{\omega}}.$$

Now if $(a,b) \models \pi_{\bar{\mu}}^{\otimes 2}|_{M_{\omega}}$, then $(a,b) \models \pi_{((\bar{\mu}')^{\otimes 2})} |_{M_{\omega}} = \pi_{\bar{\mu}'}^{\otimes 2}|_{M_{\omega}}$, in particular $a, b \in G'(M')$. By Lemma \ref{lem: props of pi mu},  $h_{\ast}\left( \pi_{(\bar{\mu}')^{\otimes 2}} \right) =  \pi_{h_{\ast}((\bar{\mu}')^{\otimes 2})} = \pi_{\bar{\nu}'}$, it follows (Remark \ref{rem: pushforward sets of realiz}) that $h(a,b) = (a,b, a \cdot_{G'} b) \models \pi_{\bar{\nu}'}|_{M_{\omega}}$. As $\Gamma(x,y,z) \in \pi_{\bar{\nu}}|_{M_{\omega}} \subseteq \pi_{\bar{\nu}'}|_{M_{\omega}}$, in particular $\models \Gamma(a,b, a \cdot_{G'} b)$. This shows that the partial function
\begin{gather}
	F \textrm{ is defined on } \pi_{\bar{\mu}}^{\otimes 2}|_{M_{\omega}} \textrm{ and agrees with  } \cdot_{G'} \textrm{ on it.}\label{eq: main thm proof 1}
\end{gather}

As $\models \Gamma(a,b, a \cdot_{G'} b)$, by definition of $H$ we have $\models \Gamma_H(a, a \cdot_{G'} b, b)$ and $F(a,b) = a \cdot_{G'} b$, so $H$ is defined on $(a, F(a,b))$ and $H(a,F(a,b)) = b$.
Similarly, by definition of $K$ we have $\models \Gamma_K(b, a \cdot_{G'} b, a)$, so $K$ is defined on  $(b, F(a,b))$ and $K(b, F(a,b)) = a$.

Given any $a \models \pi_{\bar{\mu}}|_{M_{\omega}}$, by the above  the partial map $F_a: b \mapsto F(a,b)$ is defined on $\pi_{\bar{\mu}}|_{M_{\omega},a}$, and for any $b \models \pi_{\bar{\mu}}|_{M_{\omega},a}$, $F_a(b) = a \cdot_{G'} b$. We have $(F_{a})_{\ast} \pi_{\bar{\mu}'} = \pi_{\bar{\mu}'}$. Indeed, take any $D \supseteq M \cup \{a\}$, let $b \models \pi_{\bar{\mu}'}|_{D}$. Then $(F_{a})(b) = F(a,b) = a \cdot_{G'}b \models \pi_{\bar{\mu}'}|_{D}$ by $G'$-invariance of $\pi_{\bar{\mu}'}$.
%
%
Hence we have $(F_a)_{\ast} \pi_{\bar{\mu}} = (F_a)_{\ast} (\pi_{\bar{\mu}'}\restriction_{L}) =  ((F_a)_{\ast} \pi_{\bar{\mu}'})\restriction_{L} = (\pi_{\bar{\mu}'})\restriction_{L} = \pi_{\bar{\mu}}$.

Finally, assume $(a,b,c) \models \pi_{\bar{\mu}}^{\otimes 3}|_{M_{\omega}} \vdash \pi_{\bar{\mu}'}^{\otimes 3}|_{M_{\omega}}$ (by Claim \ref{cla: pi mu implies pi mu'}). In particular, $a,b,c \in G'(M')$. First, $(b,c) \models \pi_{\bar{\mu}'}^{\otimes 2}|_{M_{\omega}, a}$, so $F(b,c) = b \cdot_{G'} c \models \pi_{\bar{\mu}'}|_{M_{\omega},a,b}$ by $G'$-invariance of $\pi_{\bar{\mu}'}$. Hence $F(a,F(b,c)) = F(a, b \cdot_{G'} c) = a \cdot_{G'} b \cdot_{G'} c$. Similarly, as $(a,b) \models \pi_{\bar{\mu}'}^{\otimes 2}|_{M_{\omega}}$, $F(a,b) = a \cdot_{G'} b$. And as $c \models \pi_{\bar{\mu}'}|_{M_{\omega}, a, b}$, we have $c \models \pi_{\bar{\mu}'}|_{M_{\omega}, a \cdot_{G'} b}$ (Remark \ref{rem: realize over dcl heq}), so $F(F(a,b),c) = F(a \cdot_{G'} b, c) = a \cdot_{G'} b \cdot_{G'} c$.
\end{proof}

Using Claim \ref{cla: proof of group chunk in the reduct} and applying Theorem \ref{thm: hyperdef group chunk gives group} in $T$, we find an $M_{\omega}$-hyper-definable in $T$ group $G^{\ast} = (X/E, \cdot_{G^{\ast}})$ and an $L(M_{\omega})$-type-definable partial function $\iota: \pi_{\bar{\mu}} \to G^{\ast}$ defined and injective on $\pi_{\bar{\mu}}$, so that for any $(a,b) \models \pi_{\bar{\mu}}^{\otimes 2}|_{M_{\omega}}$ we have $\iota(F(a,b)) = \iota(a) \cdot_{G^{\ast}} \iota(b)$ and the global partial type $f_{\ast} \pi$ concentrates on $G^{\ast}$ and is left-$G^{\ast}(M')$-invariant.

Now we work in $T'$. We know that $\pi_{\bar{\mu}'}$ is type-definable over $M_{\omega}$, concentrates on $G'$ and is $L'(M_{\omega})$-type-definable.
 As $\pi_{\bar{\mu}}|_{M_{\omega}}(x) \subseteq \pi_{\bar{\mu}'}|_{M_{\omega}}(x)$ (Claim \ref{cla: pi mu implies pi mu'}),  $\iota$ is of course still defined on $\pi_{\bar{\mu}'}|_{M_{\omega}}$ and is injective on it. And by  \eqref{eq: main thm proof 1} we have that for any $(a,b) \models \pi_{\bar{\mu}'}^{\otimes 2}|_{M_{\omega}}$, $\iota( a \cdot_{G'} b) = \iota(F(a,b))$. 
Applying Corollary \ref{cor: unique group from group chunk} in $T'$ (for $\pi_{\bar{\mu}'}$, $F$ and $f_1 = \iota: \pi_{\bar{\mu}'}|_{M_{\omega}} \to G^{\ast}$ and $f_2 = \id: \pi_{\bar{\mu}'}|_{M_{\omega}} \to G'$, we find an $L'(M_{\omega})$-type-definable group isomorphism $\iota': G' \to G^{\ast}$.

Applying Proposition \ref{prop: elim hyperdef group},  we find a group $G_0$ which is $L(M_{\omega})$-definable  in $T^{\eq}$ and an $L'(M_{\omega})$-definable isomorphism $\iota_0: G'(M') \to G_0(M')$.

Finally, say $G'(M') = \varphi(M')$ for some formula $\varphi(x) \in L'(M)$, $\Gamma_{F'}(x,y,z) \in L'(M)$ is the formula defining its graph of multiplication.  Say $G_0 = (X_0/E_0, \cdot_{G_0})$ for $X_0$ defined by $\psi(y, b)$,  the equivalence relation $E_0$ defined by $\xi(y_1,y_2,b)$, $\Gamma_{\cdot_{G_0}}$ defined by $\chi(y_1,y_2,y_3; b)$, and $\Gamma_{\iota_0}$ defined by $\rho(x,y; b)$ for some  $L(\emptyset)$-formulas $\psi(y,z), \xi(y_1,y_2,z), \chi(y_1,y_2,y_3; b), \rho(x,y; z)$ and a tuple $b$ in $M_{\omega}$.

As $(N',M') \succ^{L_P} (N,M)$, there exist some tuple $b'$ in $M$ so that $ \xi(y_1,y_2,b')$ defines an equivalence relation $E'$ in $M$,
 $G := \left( \psi(M, b')/E_0, \chi(y_1,y_2,y_3; b')  \right)$ is a group definable in $M^{\eq}$ and $\rho(x,y;b')$ defines a group isomorphism $G'(M) \to G(M)$, as wanted.
\end{proof}

\begin{question}
	Is $M^{\eq}$ needed in the conclusion of Theorem \ref{thm: ext def fsg groups is to definable}?
\end{question}

\begin{remark}\label{rem: ext def fsg iso def applies to reducts}
	Theorem \ref{thm: ext def fsg groups is to definable} also applies to arbitrary reducts of $M^{\Sh}$ expanding $M$ (with the new point being the definability of the isomorphism $\iota$ in such a reduct). 
	\end{remark}
	\begin{proof}
		This follows by inspection of the proof of Theorem \ref{thm: ext def fsg groups is to definable}. Indeed, assume $M'' = M^{\Sh} \restriction_{L''}$ for some $L \subseteq L'' \subseteq L' = L^{\Sh}$ and $G'$ is group definable  in $M''$  and fsg in the sense of $\Th_{L''}(M'')$. Then $G'$ is definable in $M^{\Sh}$ and fsg in $T' = \Th_{L'}(M^{\Sh})$, by Fact \ref{fac: group props preserved in Sh exp}. 
			
		Following the proof of Theorem \ref{thm: ext def fsg groups is to definable}, we let $\bar{\mu}' \in \mathfrak{M}_{G'}^{L'}(M')$ be an invariant generically stable over $M$ measure in $T'$, and let $\bar{\mu}'' := \bar{\mu} \restriction_{L''}$. Then $\bar{\mu}'' \in \mathfrak{M}^{L''}_x(M')$ is generically stable over $M$ in $T''$ (Proposition \ref{prop: properties of reducts of measures}), so   $\pi_{\bar{\mu}''}$ is a partial type over $M'$ type-definable over $M$ in $T''$. And $ \pi_{\bar{\mu}''} = \pi_{\bar{\mu}'} \restriction_{L''}, \pi_{\bar{\mu}} = \pi_{\bar{\mu}''} \restriction_{L}$.

		Following the proof of Theorem \ref{thm: ext def fsg groups is to definable}, we find a hyper-definable over $M_{\omega}$ group $G^{\ast}$ in $T$ (given by $L(M_{\omega})$-type-definable $X$,$E$, $\Gamma_{\cdot_{G^{\ast}}}$) and an  $L(M_{\omega})$-type-definable partial map $\iota : \pi_{\bar{\mu}'} \to G^{\ast}$.
			As $G', \iota, G^{\ast}$ are all $L''(M_{\omega})$-type-definable, we get that then the group isomorphism $\iota': G' \to G^{\ast}$ is also $L''(M_{\omega})$-type-definable (indeed, from the proof of Theorem \ref{thm: equiv of cats groups vs chunks} we have that $\Gamma_{\iota'}(y,z)$ is type-defined by $(d_{\pi_{\bar{\mu}'}}x) \rho(x,y,z)$ for $\rho$ a small partial type defined in \eqref{eq: explicit formula for group iso}, which in our case only contains $L''(M_{\omega})$-formulas, hence $(d_{\pi_{\bar{\mu}'}}x) \rho(x,y,z) \equiv^{T'} (d_{\pi_{\bar{\mu}''}}x) \rho(x,y,z)$, and the latter is defined by a partial $L''(M_{\omega})$-type using type-definability of $\pi_{\bar{\mu}''}$ in $T''$).

			  Then Proposition \ref{prop: elim hyperdef group}, applied for $T''$ and $T$, gives a group $G_0$ definable over $M_{\omega}$ in $T^{\eq}$ and an isomorphism $\iota_0:  G'(M') \to G_0(M')$  definable over $M_{\omega}$ in $(T'')^{\eq}$, and we reduce to definability over $M$ instead of $M_{\omega}$ in the same way.
	\end{proof}

\begin{corollary}\label{cor: ext def fsg groups in RCVF}
\begin{enumerate}
	\item If $T$ is $o$-minimal and $M \models T$, then every group definable and fsg in $M^{\Sh}$ is definably  isomorphic to a group definable in $M$.
	\item  Let $M'' \models \RCVF$ be arbitrary. If $G$ is definable in $M''$ fsg group, then $G$ is definably isomorphic to a group definable in the reduct $M \models \RCF$.
\end{enumerate}

\end{corollary}
\begin{proof}
(1) By Theorem \ref{thm: ext def fsg groups is to definable}, as every group definable in $M^{\eq}$ is definably isomorphic to a group definable in $M$ by \cite{eleftheriou2014interpretable}.

(2) By Remark \ref{rem: ext def fsg iso def applies to reducts} and (1), as $M''$ is a reduct of $M^{\Sh}$ (see the proof of Corollary \ref{cor: fsg subgroups in RCVF}).
\end{proof}

\begin{remark}
	In Theorem \ref{thm: ext def fsg groups is to definable}, if $G'$ is  not just fsg but moreover generically stable (i.e.~admits a left-invariant generically stable type $\bar{p}' \in S_{G'}^{L'}(M')$), we can avoid increasing the parameter set from $M$ to $M_{\omega}$ (as we already have $(\bar{p}' \restriction_{L})|_{M} \vdash \bar{p}' |_{M}$, by Lemma \ref{lem: unique extension of def meas to Sh} for types) and the passage through hyperdefinable groups in the proof, and apply a special case of the group chunk from \cite[Section 3.4]{hrushovski2019valued} for complete definable types to directly construct a group type-definable over $M$ in $T^{\eq}$ definably isomorphic to $G'$, which is then immediately definable over $M$ in $T^{\eq}$ by compactness.
\end{remark}

\begin{remark}\label{rem: type-def fsg in MSh recognize as hyper-definable}
The first part of the proof of Theorem \ref{thm: ext def fsg groups is to definable} goes through unchanged only assuming that $G'$ was a group type-definable over $M^{\Sh}$ and fsg in $T' = \Th(M^{\Sh})$, demonstrating that there is a small model $M_{\omega} \succ M$ and an  $M_{\omega}$-hyper-definable in $T$ group $G^{\ast} = (X/E, \cdot_{G^{\ast}})$ and an $L'(M_{\omega})$-type-definable isomorphism $\iota': G' \to G^{\ast}$.
\end{remark}

\begin{remark}\label{rem: ext interpretable groups}
	The situation is more complicated for fsg groups \emph{interpretable} in $M^{\Sh}$. Walsberg \cite{walsberg2019nip} (see also \cite[Section 5.1]{arXiv:2504.05566}) gives an example (related to Fact \ref{fac: G^00 ext def sometimes}) of a weakly $o$-minimal $\aleph_1$-saturated structure $\mathcal{H}$  so that $\mathcal{H}$ does not interpret an infinite groups, but $\mathcal{H}^{\Sh}$ interprets $(\mathbb{R},+,\times)$ (hence, e.g.~considering the associated circle group in it, interprets an infinite fsg group). 
	
	We note that this copy of $\mathbb{R}$ is given by a $\bigvee$-definable set in $\mathcal{H}$ quotiented by a $\bigwedge$-definable in $\mathcal{H}$ equivalence relation. The related fsg groups interpretable in  $\mathcal{H}^{\Sh}$, are  in fact hyper-definable in $\mathcal{H}$ (as we can replace the  $\bigvee$-definable set by a bounded definable set; for definably amenable, e.g. $(\mathbb{R},+)$ --- seem to get only $\bigvee/\bigwedge$-definability in $\mathcal{H}$). This motivates the following question that will be considered in future work:
\end{remark}

\begin{question}\label{que: ext interp fsg}
	Assume $M$ is a sufficiently saturated NIP structure, and $G(M)$ is definable in $(M^{\Sh})^{\eq}$ (which is not a saturated structure) and fsg (in $\Th((M^{\Sh})^{\eq})$). Then $G(M)$ is isomorphic to a group hyper-definable in $M$ (over a small set of parameters; globally in $\widetilde{M}'$ we can only expect a type-definable embedding). 
\end{question}

\section{Some further results}\label{sec: discussion}

\subsection{On the algebraic closure in Shelah expansion}\label{sec: acl in MSh}

The following basic question remains open:
\begin{question}
	Assume that $M$ is NIP and there is an infinite definable group in $\Th(M^{\Sh})$. Does it imply that there is an infinite definable (or at least type-definable) group in $\Th(M^{\eq})$?
\end{question}
\begin{remark}
	Note that this is not true if we only have an interpretable group in $\Th(M^{\Sh})$, by the example in Remark \ref{rem: ext interpretable groups} we would need to at least consider hyper-definable groups.
\end{remark}

\begin{prop}\label{prop: disint acl implies no groups}
	In any theory $T$, if the algebraic closure is \emph{disintegrated}, i.e. $\acl(a,b) = \acl(a) \cup \acl(b)$ for all finite sets $a,b$, then there are no infinite type-definable groups.
\end{prop}
\begin{proof}
	Let $\mathbb{M} \models T$ be a monster model and $(G, \cdot)$ an infinite group type-definable over a small set $A \subset \mathbb{M}$, i.e.~we have  $A$-definable functions $\cdot : (\mathbb{M}^x)^2 \to \mathbb{M}^x, ^{-1}: \mathbb{M}^x \to \mathbb{M}^x$ so that their restrictions to the partial type $G(x)$ defining $G$ are the group operations on $G$. As $G$ is infinite, by Ramsey and compactness we can find sequences $(a_i : i \in \mathbb{Q}), (b_i : i \in \mathbb{Q})$ of pairwise distinct elements of $G$ that are mutually indiscernible over $A$ (e.g. they can be two parts of the same $A$-indiscernible sequence of order type $\omega + \omega$).
	
	Let $c := a_0 \cdot b_0$, then $c \in \dcl(a_0, b_0, A)$.
	If  $c \in \acl(a_0, A)$, then $b_0 = a^{-1}_0 \cdot (a_0 \cdot b_0) = a^{-1}_0 \cdot c \in \acl(a_0, A)$ --- contradicting that the sequence $(b_i : i \in \mathbb{Q})$ is infinite and indiscernible over $a_0, A$ (in general, if a sequence is indiscernible over a set, then it is indiscernible over its algebraic closure). If  $c \in \acl(b_0, A)$, then $a_0 = (a_0 \cdot b_0) \cdot b^{-1}_0 = c \cdot b^{-1}_0 \in \acl(b_0,A)$ --- contradicting that the sequence $(a_i : i \in \mathbb{Q})$ is infinite and indiscernible over $b_0, A$.
\end{proof}

The following is standard:
\begin{remark}\label{rem: choosing alg formula}
		In any $L$-theory, if $a \in \acl(b)$, then there is some $k \in \omega$ and formula $\varphi(x,y) \in L$ so that $\models \varphi(a,b)$ and for every $b' \in \mathbb{M}^y$, $|\varphi(\mathbb{M},b')| \leq k$.

Indeed, given any $\psi(x,y) \in L$ so that $\psi(\mathbb{M},b')$ is finite, let $k:= |\psi(\mathbb{M},b')|$ and let $\varphi(x,y)$ be the formula 
\begin{gather*}
	\psi(x,y) \land \left( \exists x_1 \ldots \exists x_k  \left( \bigwedge_{i=1}^k \psi(x_i,y) \land  \forall x' \left(  \psi(x',y) \rightarrow \bigvee_{i=1}^k x' = x_i \right) \right) \right).
\end{gather*}
\end{remark}

We are following the notation from Definition \ref{def: Shelah exp context}.
Let us say that a binary relation $R \subseteq X \times Y$ is \emph{$k$-fiber-algebraic}, $k \in \omega$, if for every $x \in X$ there are at most $k$ elements $y\in Y$ so that $(x,y) \in R$. The following is a slight generalization of \cite[Corollary 1.5]{chernikov2013externally}:

\begin{cor}
Let $R \subseteq M^x \times M^y$ be an externally definable binary relation that is $k$-fiber-algebraic. Then there exists an $\mathbb{M}$-definable $k$-fiber-algebraic relation $R' \subseteq \mathbb{M}^x \times \mathbb{M}^y$ so that $R = R' \cap \left( M^x \times M^y \right)$. 

In particular, if $R$ is the graph of a function, then $R'$ is the graph of a partial function.
\end{cor}
\begin{proof}
	Say $R$ is the trace on $M$ of some $L(c)$-formula $\varphi(x,y; c)$ with $c \in N \succeq M$. By Fact \ref{fac: honest defs} we find $(N',M') \succ (N,M)$ and $\theta(x,y) \in L(M')$ so that $\theta(M) = \varphi(M,c)$ and $\theta(M') \subseteq \varphi(M';c)$. As the extension of pairs is elementary, we still have that $\varphi(M';c)$ is $k$-fiber-algebraic, then $\theta(M') \subseteq \varphi(M';c)$ implies that $\theta(M')$ is $k$-fiber-algebraic.
\end{proof}

 We let $\acl$ denote the algebraic closure in $T$, and $\acl^{\Sh}$ denote the algebraic closure in $\Th(M^{\Sh})$.

\begin{proposition}\label{prop: disint acl in Shelah exp}
If $T$ is NIP and has disintegrated $\acl$, then $\acl^{\Sh}(A) = \acl(A) \cup M$ for any $A \subseteq M'$ (in particular, $\acl^{\Sh}$ is also disintegrated).
\end{proposition}
\begin{proof}

	Clearly $\acl(A,M) \subseteq \acl^{\Sh}(A)$ for every $A \subseteq M'$ (note that $M \subseteq \dcl^{\Sh}(\emptyset)$ since every singleton subset of $M$  gets named by a relation symbol in $M^{\Sh}$). Assume we have tuples $a, c\in M'$ with $c \in \acl^{\Sh}(a) \setminus M$.

	As $c \in \acl^{\Sh}(a)$, let  $k \in \omega$ and $\varphi(x;y)$ be an $L^{\Sh}$-formula given by Remark \ref{rem: choosing alg formula} applied in $\Th(M^{\Sh})$. By quantifier elimination in $\Th(M^{\Sh})$, $\varphi(M') = \varphi'(M',d)$ for some $\varphi'(x,y;u) \in L$ and some $d \in N$. That is, there exist some $d \in N$, $k \in \omega$ and $\varphi'(x, y; u) \in L$ so that 
	 \begin{gather}
	 	(N', M') \models \varphi'(c, a; d) \land  \forall y \in P\,  \exists^{\leq k} x \in P \, \varphi'(x, y; d) \label{eq: disint acl in Sh 1}
	 \end{gather}
	 (but $\varphi'(x, a; d)$ might have infinitely many realizations in $N$).

	By Fact \ref{fac: honest defs}, let $\theta(x,y;v) \in L$ and $e \in M'$ be such that
	\begin{gather}
		\theta(M;e) = \varphi'(M; d) \textrm{ and }\theta(M';e) \subseteq  \varphi'(M'; d).\label{eq: disint acl in Sh 3}
	\end{gather}
	
 Consider the partial $L$-types 
	\begin{gather*}
		A(x,y) := \left\{ \neg \alpha(x;y) \in L : \exists k' \in \omega \, \forall a' \in (M')^y \,  |\alpha(M', a')| \leq k' \right\},\\
			\Gamma(x; v) := \left\{ \neg \gamma(x;v) \in L : \exists k' \in \omega \, \forall e' \in (M')^v \,  |\gamma(M', e')| \leq k' \right\}.
	\end{gather*}
	Let $A_0 \subseteq A, \Gamma_0 \subseteq \Gamma$ be arbitrary finite subsets. By definition of $\Gamma$, the formula $\neg \bigwedge\Gamma_0(x, e)$ has only finitely many realizations in $M'$, let $c_1, \ldots, c_n$, $n \in \omega$, enumerate all of its realizations in $M$.
	
As $c \notin M$ by assumption, we have 
\begin{gather*}
	(N', M') \models \varphi'(c, a; d) \land A_0(c,a) \land \bigwedge_{i=1}^{n} c \neq c_i.
\end{gather*}
As $ (N, M) \prec (N', M')$, there exist some $c', a'$ in $M$ so that 
\begin{gather*}
	(N, M) \models \varphi'(c', a'; d) \land A_0(c',a') \land \bigwedge_{i=1}^{n} c'\neq c_i.
\end{gather*}
By the choice of $c_i$ and \eqref{eq: disint acl in Sh 3} we thus have 
\begin{gather*}
	(N', M') \models \theta(c', a'; e) \land A_0(c',a') \land \Gamma_0(c',e).
\end{gather*}

	 By saturation of the pair $(N', M')$ we can thus find $a^\ast, c^\ast$  in $M'$ so that 
\begin{gather*}
	(N', M') \models \theta(c^\ast, a^\ast; e) \land A(c^\ast,a^\ast) \land \Gamma(c^\ast,e).
\end{gather*}
Since all of these are $L$-formulas and all parameters are in $M'$ (and $(N', M')$ is an elementary pair), we have 
\begin{gather*}
	M' \models \theta(c^\ast, a^\ast; e) \land A(c^\ast,a^\ast) \land \Gamma(c^\ast,e).
\end{gather*}

As $\theta(M',e) \subseteq \varphi'(M',d)$ by \eqref{eq: disint acl in Sh 3}, using  \eqref{eq: disint acl in Sh 1} we get that $\theta(x, a^\ast; e)$ has at most $k$ realizations in $M'$. It follows by the choice of $A,\Gamma$ and Remark \ref{rem: choosing alg formula} (applied in $M' \models T$) that
\begin{gather*}
	c^\ast \in \acl(a^\ast,e) \setminus \left( \acl(a^\ast)  \cup \acl(e) \right).
\end{gather*}
	 
	This contradicts the disintegration of the  algebraic closure in $T$.
\end{proof}

\begin{corollary}\label{cor: disint acl no groups in MSh}
	If $T$ is NIP, has disintegrated $\acl$ and $M \models T$, then there are no infinite type-definable groups in $\Th(M^{\Sh})$.
\end{corollary}
\begin{proof}
By Propositions \ref{prop: disint acl in Shelah exp} and \ref{prop: disint acl implies no groups}.
\end{proof}

\begin{remark}
	It follows from Corollary \ref{cor: disint acl no groups in MSh} that there are no infinite type-definable groups in $\mathcal{H}^{\Sh}$ from Example \ref{rem: ext interpretable groups}. 
\end{remark}

\subsection{Uniform type-definability for some sets related to generically stable measures}\label{sec: hyperdef sets from measures}

We are interested when the type-definable equivalence relations arising in the proof of Theorem \ref{thm: ext def fsg groups is to definable} (i.e.~the equality of germs in a definable family of partial functions on a generically stable measure) are uniformly type-definable or externally definable. We have:
\begin{lemma}
	For every $\varphi(x,y)$ and $r \in [0,1]$ there is an $\emptyset$-type definable set $D_r \subseteq \cU^y \times X_x$ so that: for any generically stable measure $\mu \in \mathfrak{M}_x(\cU)$ and any $b \in \cU^y$, $\mu(\varphi(x,b)) = r$ if and only if $(b, \bar{a}^{\mu}) \in D_r$ (where $\bar{a}^{\mu}$ is as in Definition \ref{def: a^mu coding measure}).
\end{lemma}
\begin{proof}
	We let  $D_r$ be the set of pairs $(b, \bar{a}) \in \cU^y \times X_x$ satisfying
	\begin{gather*}
\bigwedge_{k \in \mathbb{N} } \left( \left\lvert\Av(\bar{a}_{\varphi, 1/k}; \varphi(x,b))  - r)  \right \rvert \leq 1/k \right).
	\end{gather*}
As $\bar{a}^{\mu}_{\varphi, 1/k}$ is a $1/k$ approximation for $\mu$ on $\varphi(x,y)$ for any generically stable measure $\mu$, we have $(b, \bar{a}^{\mu}) \in D_r$ if and only if $\mu(\varphi(x,b))=r$.
\end{proof}

Note that $D_r$ is defined by a conjunction of instances of infinitely many distinct formulas (the condition ``$\left\lvert\Av(\bar{a}_{\varphi, 1/k}; \varphi(x,b))  - r)  \right \rvert \leq 1/k $'' is expressed by a disjunction of size growing with $k$ of conjunctions of instances of the form $\varphi(a,y)$). We consider a stronger notion:

\begin{definition}\label{def: unif type-def}
A set $X \subseteq \cU^x$ is \emph{uniformly} type-definable over $A \prec \cU$ if it is type-definable over $A$ by a partial type of the form $\pi(x) = \{ \varphi(x,a) : a \in A\}$ consisting of instances of a single formula $\varphi(x,y) \in L$.
\end{definition}

 \begin{proposition}\label{prop: unif def meas}
 	\begin{enumerate}
 		\item Let $T$ be arbitrary. Then for every stable formula $\varphi(x,y) \in L$ there is some $\theta(y,z) \in L$ so that: for every measure $\mu \in \mathfrak{M}_x(\cU)$, the set $\{b \in \cU^y : \mu(\varphi(x,b)) = 1 \}$ is uniformly type-definable over $M$ by instances of $\theta$.
 		Moreover, if $\mu$ is invariant over $M \prec \cU$, then we only need instances of $\theta$ with parameters in $M$.
 		\item Let $T$ be distal. Then for every $\varphi(x,y) \in L$ there is some $\theta(y,z) \in L$ so that: for every $r \in [0,1]$ and every measure $\mu \in \mathfrak{M}_x(\cU)$ generically stable over $M$ (hence smooth over $M$), the set  $\{b \in \cU^y : \mu(\varphi(x,b)) = r \}$ is uniformly type-definable over $M$ by instances of $\theta$.
 		\item If $T$ is NIP, $G$ is $M$-definable and $\mu \in \mathfrak{M}_{G}(\cU)$ witnesses that $G$ is fsg over $M$, then  for every $\varphi(x,y) \in L$, the set $\{b \in \cU^y : \mu(\varphi(x,b)) = 1 \}$ is externally definable.
 	\end{enumerate}
 \end{proposition}
 \begin{proof}
 	(1) Let $\varphi(x,y) \in L$ be a stable formula. By uniform definability of types with respect to a stable formula, there is some $\theta(y,z) \in L$ (given by a Boolean combination of instances of $\varphi^*(y,x)$) so that: for every type $p \in S_x(\cU)$ there is some $c_p \in \cU^z$ so that for every $b \in \cU^y$ we have $\varphi(x,b) \in p \iff \models \theta(b,c_p)$.
 	By Keisler's theorem, for every stable formula $\varphi(x,y) \in L$ and every Keisler measure $\mu \in \mathfrak{M}_x(\cU)$, there are some types $p_i \in S(\mu), i \in \omega$ and $r_i \in (0,1]$ with $\sum_{i \in \omega} r_i = 1$ so that for every $b \in \cU^y$ we have $\mu( \varphi(x,b)) = \left( \sum_{i \in \omega} r_i \cdot p_i \right)(\varphi(x,b))$. By the above, let $c_{p_i} \in \cU^z$ be so that  $\theta(y, c_{p_i})$ is a definition for $p_i$. Let $\Theta_{\mu}(y) := \{ \theta(y, c_{p_i}) : i \in \omega \}$. Now given $b \in \cU^y$, if $\mu(\varphi(x,b)) = 1$ then $\varphi(x,b) \in p_i$ for all $i \in \omega$ as $p_i \in S(\mu)$, so $\models \Theta(b)$. And conversely, if $\models \Theta(b)$ then $\mu(\varphi(x,b)) = \sum_{i \in \omega} r_i = 1$.
 	
 	And if $\mu$ was $M$-invariant, every $p \in S(\mu)$ is non-forking over $M$, hence $p\restriction_{\varphi} \in S_{\varphi}(M)$ is definable over $M$ by basic local stability, so $c_{p_i}$ above can be chosen in $M$. 
 	
 	(2) By uniform strong honest definitions in distal theories (\cite[Section 1.3]{chernikov2015externally}), given $\varphi(x,y)$ there is some $\theta(x_1, \ldots, x_k;y) \in L$ so that: for every finite $A \subseteq \cU^x$, $|A| \geq 2$ and $b \in \cU^y$, there are some $a_1, \ldots, a_k \in A$ so that $\models \theta(a_1, \ldots, a_k; b)$ and $\theta(a_1, \ldots, a_k; y) \vdash \tp_{\varphi}(b/A)$. In particular, we have distal cell decompositions: for every finite $A \subseteq \cU^x$ there are some $\bar{a}_1, \ldots, \bar{a}_N \in A^n$ so that $\cU^y \subseteq \bigcup_{1 \leq i \leq N} \theta(\bar{a}_i, \cU)$ (not necessarily a disjoint union) and for any $i \in [N]$ and $b,b' \models \theta(\bar{a}_i, y)$, $\tp_{\varphi}(b/A) = \tp_{\varphi}(b'/A)$.
 	
 	Now let $\mu \in \mathfrak{M}_x(\cU)$ be an arbitrary measure generically stable over $M$ and $r \in [0,1]$. Fix $\varepsilon > 0$, by generic stability there are some $n_{\varepsilon} \in \omega$ and $(a^{\varepsilon}_1, \ldots, a^{\varepsilon}_{n_{\varepsilon}}) \in (M^x)^{n_{\varepsilon}}$ so that for every $b \in \cU^y$, $\mu(\varphi(x,b)) \approx^{\varepsilon} \Av(a^{\varepsilon}_1, \ldots, a^{\varepsilon}_{n_{\varepsilon}}; \varphi(x,b))$. Let $A^{\varepsilon} := \{a^{\varepsilon}_i : i \in [n_{\varepsilon}]  \}$. Let $B^\varepsilon := \{ \bar{a}^{\varepsilon}_i \in (A^{\varepsilon})^k: i \in [N_{\varepsilon}]\}$ be as given by the distal cell decomposition. That is, $\cU^y = \bigcup_{i \in [N_{\varepsilon}]} \theta(\bar{a}^{\varepsilon}_i; \cU)$ and for any $i \in [N_{\varepsilon}]$ and $b,b' \models \theta(\bar{a}^{\varepsilon}_i, y)$, $\tp_{\varphi}(b/A^{\varepsilon}) = \tp_{\varphi}(b'/A^{\varepsilon})$; in particular, $\Av(a^{\varepsilon}_1, \ldots, a^{\varepsilon}_{n_{\varepsilon}}; \varphi(x,b)) = \Av(a^{\varepsilon}_1, \ldots, a^{\varepsilon}_{n_{\varepsilon}}; \varphi(x,b'))$. Let $C^{\varepsilon}$ be the set of $ i \in [N_{\varepsilon}]$ so that for some (equivalently, any) $b \models \theta(\bar{a}^{\varepsilon}_i, y)$, $|\Av(a^{\varepsilon}_1, \ldots, a^{\varepsilon}_{n_{\varepsilon}}; \varphi(x,b)) - r| > \varepsilon$.
Finally, consider the partial type 
$$\Theta (y) := \left\{ \neg \theta \left(\bar{a}^{1/m}_{i}, y \right) : m \in \omega, i \in C^{1/m}\right\}$$
of size $|T|$ over $M$. 	Let $b \in \cU^y$ be arbitrary.

  Assume $\mu(\varphi(x,b)) = r$. For any  $m \in \omega$ and $i \in C^{1/m}$, we cannot have $\models \theta(\bar{a}^{1/m}_{i};b)$ because by the choice of $C^{1/m}$ we would then have 
  $$|\Av(a^{\varepsilon}_1, \ldots, a^{\varepsilon}_{n_{\varepsilon}}; \varphi(x,b)) - r| > 1/m,$$
   but by the choice of $A^{\varepsilon}$ we have 
  $$|\Av(a^{\varepsilon}_1, \ldots, a^{\varepsilon}_{n_{\varepsilon}}; \varphi(x,b)) - r| = | \Av(a^{\varepsilon}_1, \ldots, a^{\varepsilon}_{n_{\varepsilon}}; \varphi(x,b)) - \mu(\varphi(x,b))| \leq 1/m,$$
  so $\models \Theta(b)$.
 	
 	Conversely, assume $\models \Theta(b)$. For each $m \in \omega$, by the choice of $B^{1/m}$ we must have $\models \theta(\bar{a}^{\varepsilon}_i; b)$ for at least one $i \in [N_{\varepsilon}]$, and necessarily $i \notin C^{\varepsilon}$ since $\models \Theta(b)$, so $\mu(\varphi(x,b)) \approx^{1/m} \Av(a^{\varepsilon}_1, \ldots, a^{\varepsilon}_{n_{\varepsilon}}; \varphi(x,b)) \approx^{1/m} r$, hence $\mu(\varphi(x,b)) \approx^{2/m} r$. As this holds for all $m \in \omega$, we must have $\mu(\varphi(x,b)) = r$.
 	
 	(3) Let $\mu \in \mathfrak{M}^L_G(\cU)$ be the unique left-$G(\cU)$-invariant measure and let $p \in S_{G}^L(\cU)$ be a generic type, then $S(\mu) = \overline{G \cdot p}$ (see the proof of Proposition \ref{prop: fsg iff invariant filter str fin sat}). Then $\mu(\varphi(x,b)) =1 \iff \varphi(x,b) \in q$ for all $q \in S(\mu)$, $\iff \varphi(x,b) \in g \cdot p$ for all $g \in G(\cU)$, $\iff \varphi(g \cdot x,b) \in p$ for all  $g \in G(\cU)$. Let $\psi(x;y,z) := \varphi(y \cdot x,z)$, then $(d_p x) \psi(x;y,z)$ is definable in $\cU^{\Sh}$, hence $\forall y \in G ((d_p x) \psi(x;y,z))$ is also definable in $\cU^{\Sh}$.
 \end{proof}
 
 \begin{question}
 	\begin{enumerate}
 		\item In Proposition \ref{prop: unif def meas}(1), can we replace $1$ by an arbitrary $r \in [0,1]$?
 		\item Does Proposition \ref{prop: unif def meas}(1) hold for arbitrary generically stable measures in NIP theories?
 	\end{enumerate}
 \end{question}
 
 \begin{remark}
	Let $M \models T$ be the theory of equality and a tuple of variables $x$ be arbitrary. Let  $A \subseteq M^x$ be an arbitrary countable set, say $A = \{a_i : i \in \omega \}$.  Choose arbitrary $r_i \in (0,1]$ so that $\sum_{i \in \omega} r_i = 1$. Let $\mu := \sum_{i \in \omega} r_i \cdot \tp(a_i/\cU)$, then $\mu$ is smooth over $M$. Then for any $b \in \cU^y$, $\mu( x \neq b)=1 \iff b \notin A$. So in general the set $\{b \in M^y: \mu(\varphi(x,b))=1\}$ is not externally definable even for smooth measures.
\end{remark}

\subsection{Externally type-definable and $\bigvee$-definable sets}\label{sec: ext type-def sets}
In this section we give an internal description of externally type-definable and $\bigvee$-definable sets in NIP theories, and a soft description of externally $\bigvee$-definable subgroups of definable groups.

\begin{definition}
	Let $M \models T$ and $\kappa$ a cardinal. A set $X \subseteq M^x$ is \emph{externally $\kappa$-type-definable} (\emph{externally $\kappa$-$\bigvee$-definable}) if there is a partial type $\Gamma(x;\bar{y})$ over $\emptyset$ of size $\leq \kappa$ (respectively, a disjunction $\Sigma(x;\bar{y})$  of $\leq \kappa$ formulas over $\emptyset$), where $\bar{y}$ possibly infinite of length at most $\kappa$, and a  tuple $\bar{b} \in \cU^{\bar{y}}$ so that $X = \Gamma(M; \bar{b})$, where $\Gamma(M; \bar{b})$ is the set of realizations of $\Gamma(x,\bar{b})$ in $M$ (respectively, $X = \Sigma(M, \bar{a})$).
\end{definition}

We consider several notions of directedness for type-definable and $\bigvee$-definable families working in a possibly non-saturated model:
\begin{definition}\label{def: unif dir}
	Let $\Gamma(x;\bar{y})$ be a conjunction (or a disjunction) of a tuple of formulas  $(\varphi_t(x, y_t) :  t < \kappa )$ where $\varphi_t(x,y_t) \in L(\emptyset)$, $y_t$ is a finite tuple of variables and $\bar{y} = (y_t : t < \kappa)$ is a tuple of variables of length at most $\kappa$,  and $\Omega \subseteq M^{\bar{y}}$ an arbitrary set.
	We say that the family of subsets $\left\{ \Gamma(M;\bar{b}) : \bar{b} \in \Omega \right\}$ of $M^x$  is
	\begin{enumerate}
	\item \emph{$\lambda$-directed}, for a cardinal $\lambda$, if for every $\Omega_0  \subseteq \Omega$ with $|\Omega_0| < \lambda$ there is some $\bar{b} \in \Omega$  so that $  \bigcup_{i \in \Omega_0} \Gamma(M, \bar{b}_{i}) \subseteq  \Gamma(M, \bar{b})$. We say \emph{directed} for $\aleph_0$-directed.
	
		\item \emph{downwards $\lambda$-directed} if for every $\Omega_0  \subseteq \Omega$ with $|\Omega_0| < \lambda$ there is some $\bar{b} \in \Omega$  so that $\Gamma(M, \bar{b}) \subseteq \bigcap_{i \in \Omega_0} \Gamma(M, \bar{b}_{i})$; 
		
		\item \emph{uniformly directed} if there is some $f: \kappa \to \kappa$ so that for every $n \in \omega$, finite $I \subseteq \kappa$ and $\bar{b}_1, \ldots, \bar{b}_n \in \Omega$ there is some $\bar{b} \in \Omega$  so that: 
		\begin{itemize}
			\item if $\Gamma(x,\bar{y})$ is a disjunction, $\bigvee_{i=1}^{n}  \varphi_t(x, \bar{b}_i) \vdash \varphi_{f(t)}(x, \bar{b})$ for all $t \in I$;
			\item if $\Gamma(x,\bar{y})$ is a conjunction, $\bigvee_{i=1}^{n}  \varphi_{f(t)}(x, \bar{b}_i) \vdash \varphi_{t}(x, \bar{b})$ for all $t \in I$.
			\end{itemize}
	\item \emph{uniformly downwards directed} if $\Gamma(x,\bar{y})$ is a conjunction, and there is $f: \kappa \to \kappa$ so that for all $n \in \omega$, $t < \kappa$ and $\bar{b}_1, \ldots, \bar{b}_n \in \Omega$ there is $\bar{b} \in \Omega$  so that $ \varphi_{f(t)}(x, \bar{b}) \vdash \bigwedge_{i=1}^{n}  \varphi_{t}(x, \bar{b}_i)$ for all $t \in I$.
	\end{enumerate}
	(If $\left\{ \Gamma(M;\bar{b}) : \bar{b} \in \Omega \right\}$ is uniformly directed, then it is also directed; and if $\Gamma(x,\bar{y})$ consists of a single formula, the two notions are equivalent)
\end{definition}

We recall from \cite{hrushovski2016non}, here $M$ is not necessarily saturated:
\begin{definition}
	A set of tuples $\Omega \subseteq M^{\bar{y}}$, with $\bar{y} = (y_t: t<\kappa)$ and each $y_t$ a finite tuple of variables, is \emph{strict pro-definable} if $\Omega = \Psi(M)$ for a partial type $\Psi(\bar{y})$ over $M$ and,  letting $\pi_{I}: M^{\bar{y}} \to \prod_{t \in I}  M^{y_{t}}$ be the projection map, $\pi_I(\Omega) = \psi_I(M)$ for some $\psi_{I}((y_t)_{t \in I}) \in L(M)$ (hence without loss of generality in $\psi_I \in \Psi$).
\end{definition}
\begin{remark}\label{rem: strict pro lifts}
If $N \succ M$ is $(|M|+\kappa)^{+}$-saturated, then $\Psi(N)$ is still strict pro-definable. Indeed, for every finite $I, J \subseteq \kappa$	 and $\Psi_0 ((y_t)_{t \in I \cup J}) \subseteq \Psi$, $M \models \forall (y_t)_{t \in I} ( \psi_I((y_t)_{t \in I}) \rightarrow \exists (y_t)_{t \in J} \Psi_0 ((y_t)_{t \in I \cup J}))$. Hence the same holds in $N$ by elementarity, so by saturation  $\pi_{I}(\Psi(N)) = \psi_{I}(N)$.
\end{remark}

\begin{remark}\label{rem: unif direct and strict pro-def}

(1) Directed does not imply uniformly directed, even if $\Omega$ is type-definable in a saturated $M$ (consider $\Sigma(x,\bar{y}) = \bigvee_{t \in \aleph_0}(x = y_t)$ and $\Omega = M^{\aleph_0}$ for $M \models T$ theory of an infinite set; then $\left\{\Sigma(M,\bar{b})\right\}_{\bar{b} \in \Omega}$ is $\aleph_1$-directed, but not uniformly directed).

(2) If $\left\{\Gamma(M,\bar{b})\right\}_{\bar{b} \in \Omega}$ is uniformly directed, $\Omega \subseteq M^{\bar{y}}$ is type-definable and $M$ is saturated, then $\left\{\Gamma(M,\bar{b})\right\}_{\bar{b} \in \Omega}$ is $\lambda$-directed for every small $\lambda$. Indeed, say $ \Omega = \pi(M)$ for a small partial type $\pi(\bar{y})$, and let $\bar{b}_i \in \pi(M)$ for $i < \lambda$ be given. Let $\Theta(\bar{y}) := \pi(\bar{y}) \cup \bigcup_{t < \kappa}\{\forall x (\varphi_{f(t)}(x, \bar{y}) \rightarrow \varphi_t(x, \bar{b}_{i})) : i < \lambda \}$. Every finite subset of $\Theta(\bar{y})$ is consistent by uniform directedness, hence there is $\bar{b} \in \Theta(M)$ by saturation, and it satisfies the requirement. Similarly for $\Sigma$.

(3) Assume that $M \prec N$ and $N$ is saturated, $\Omega \subseteq M^{\bar{y}}$ is strict pro-definable in $M$ via $\Psi(\bar{y})$ and $\{\Gamma(M,\bar{b}) : \bar{b} \in \Omega\}$ (respectively, $\{\Sigma(M,\bar{b}) : \bar{b} \in \Omega\}$) is uniformly directed. Then $\{\Gamma(N,\bar{b}) : \bar{b} \in \Psi(N) \}$ (respectively, $\{\Sigma(N,\bar{b}) : \bar{b} \in \Psi(N) \}$) is also uniformly directed. 

Indeed, by uniform directedness in $M$ and strict pro-definability, given $t < \kappa$, for any fixed $n \in \omega$ and $b_{1,t}, \ldots, b_{n,t} \in \psi_{\{t\}}(M^{y_t})$ there is some $b_{f(t)} \in \psi_{\{f(t)\}}(M^{y_{f(t)}})$ so that 
$M \models \forall x \left( \varphi_{\{f(t)\}}(x, b_{f(t)}) \rightarrow \bigwedge_{i \in [n]} \varphi_{\{t\}}(x, b_{i,t}) \right)$. 
By elementarity the same holds in $N$, so we conclude by Remark \ref{rem: strict pro lifts}. Similarly for $\Sigma$.
	
\end{remark}

\begin{proposition}\label{prop: unif directed implies ext type def}
	Let $M$ be a small model, $\Omega \subseteq M^{\bar{y}}$ and $\left\{ \Gamma(x;\bar{b}) : \bar{b} \in \Omega \right\}$ (respectively, $\left\{ \Sigma(x;\bar{b}) : \bar{b} \in \Omega \right\}$)  a uniformly directed family. Then there is $\bar{b}^* \in \cU^{\bar{y}}$ with
	$$\bigcap_{\bar{b} \in \Omega}  \Gamma(M, \bar{b}) = \Gamma(M, \bar{b}^*) \textrm{ (respectively, } \bigcup_{\bar{b} \in \Omega}  \Sigma(M, \bar{b}) = \Sigma(M, \bar{b}^*)\textrm{),}$$
	in particular this intersection (respectively, union) is externally $|\Gamma|$-type-definable (respectively, $|\Gamma|$-$\bigvee$-definable).
\end{proposition}
\begin{proof}
	Let $\Gamma(x;\bar{y}) = \{\varphi_t(x, \bar{y}) : t < \kappa \}$ and $f: \kappa \to \kappa$ be given by uniform directedness.
	Let $\pi(\bar{y})$ be the following set of formulas over $M$:
	\begin{gather*}
	 \left\{ \forall x \left( \varphi_{f(t)}(x,\bar{y}) \rightarrow \varphi_t(x, \bar{b}) \right)  : t < \kappa, \bar{b} \in \Omega \right\} \cup \bigcup \left\{ \Gamma(a, \bar{y}) : a \in \bigcap_{\bar{b} \in \Omega} \Gamma(M,\bar{b})  \right\}.
	\end{gather*}
	
	Then $\pi(\bar{y})$ is consistent. Indeed, given any finitely many $\bar{b}_1, \ldots, \bar{b}_n \in \Omega$, by uniform directedness there is some $\bar{b}' \in \Omega$ so that $\varphi_{f(t)}(x, \bar{b}') \vdash \bigwedge_{i=1}^{n}  \varphi_t(x, \bar{b}_i)$ for all $t < \kappa$. And we still have $a \in \Gamma(M,\bar{b}')$ for every $a$ appearing in the second set of conditions.
	This shows that every finite subtype of $\pi$ is consistent, hence $\pi(\bar{y})$ is consistent. Let $\bar{b}^{\ast} \models \pi(\bar{y})$ in $\cU$, we claim that it satisfies the requirement. Indeed, from the first set of conditions, $\Gamma(x,\bar{b}^{\ast})  \vdash \Gamma(x,\bar{b})$ for every $\bar{b} \in \Omega$, so $\Gamma(M,\bar{b}^{\ast}) \subseteq \bigcap_{\bar{b} \in \Omega}  \Gamma(M, \bar{b})$. Conversely, if $a \in \bigcap_{\bar{b} \in \Omega}  \Gamma(M, \bar{b})$, then also $a \in \Gamma(M,\bar{b}^{\ast})$ by the second set of conditions.
	 Similarly for $\Sigma$, letting $\bar{b}^{\ast} \in \cU^y$ realize the following partial type over $M$:
\begin{gather*}
	\left\{ \forall x \left(  \varphi_t(x, \bar{b})  \rightarrow \varphi_{f(t)}(x,\bar{y})\right)  : t < \kappa, \bar{b} \in \Omega \right\} \cup \\ \bigcup \left\{ \neg \Sigma(a, \bar{y}) : a \in M^x \setminus \bigcup_{\bar{b} \in \Omega} \Sigma(M,\bar{b})  \right\}. \qedhere
\end{gather*}

%
%
\end{proof}

We will need a slightly stronger version of honest definitions that work against an additional uniformly definable family $\zeta(x,u)$ of sets from \cite[Proposition 1.11]{chernikov2013externally} (generalizing the usual case $\zeta(x,u) = (x=u)$), which we state here with additional uniformity of the resulting honest definition:
\begin{fact}\label{fac: better uniform honest defs}
	Let $T$ be NIP, $M \models T$. Let $X \subseteq M^{x}$ be externally definable, say $X = \varphi(M,b)$ for some $\varphi(x,y) \in L$ and $b \in \cU^y$. Let $\zeta(x,u) \in L$. Let $\Omega := \{e \in M^u  : \zeta(M,e) \subseteq X\}$, and assume that $X = \bigcup_{e \in \Omega} \zeta(M,e)$. Then there is a formula $\theta(x,z) \in L$ depending only on $\varphi(x,y)$ and $\zeta(x,u)$, and there is $c \in \cU^z$ such that:
	\begin{enumerate}
		\item $\theta(x,c)$ is an honest definition of $X$,
		\item $\cU \models \zeta(x,e) \rightarrow \theta(x,c)$ for every $e \in \Omega$,
		\item for any $n \in \omega$ and $e_1, \ldots, e_n \in \Omega$, there is $c' \in M^z$ so that $\theta(M, c') \subseteq X$ and $\cU \models \zeta(x,e_i) \rightarrow \theta(x,c')$ for all $i = 1, \ldots, n$.
	\end{enumerate}
\end{fact}
\begin{proof}
If $X = \varphi(M,b)$, then by quantifier elimination in Shelah's expansion the set 
$$U := \{ e \in M^u: \forall a \in M^{x}, \zeta(a,e) \rightarrow  \varphi(a,b) \}$$
is externally definable. Moreover, by the existence of uniform honest definitions (\cite[Theorem 11]{chernikov2015externally}), there is a formula $\psi(u, v) \in L$ depending only on $\varphi(x,y), \zeta(x,u)$ so that $U = \psi(M, d)$ for some $d \in \cU^v$ (see the proof of \cite[Proposition 1.9.]{chernikov2013externally}). Let $\alpha(u, z)$ depending only on $\psi(u,v)$ be so that $\alpha(M,f)$, $f \in \cU^z$, is an honest definition for  $\psi(M, d) $.
The proof of \cite[Proposition 1.11]{chernikov2013externally} shows that we can take $\theta(x,z) := \exists u (\alpha(u,z) \land \zeta(x,u))$ (which depends only on $\varphi(x,y)$ and $\zeta(x,u)$).
\end{proof}

\begin{theorem}\label{thm: ext type def presentation}
	Let $M \models T$ be NIP and $\kappa \geq \aleph_0$ a cardinal. A set $X \subseteq M^x$ is externally $\kappa$-type-definable (respectively, externally $\kappa$-$\bigvee$-definable) if and only if  $X = \bigcap_{\bar{b} \in \Omega} \Theta(M;\bar{b})$ (respectively, $X = \bigcup_{\bar{b} \in \Omega} \Theta(M;\bar{b})$) for $\Theta(x;\bar{z})$ a partial type (respectively, a disjunction of formulas) over $\emptyset$ of size $\leq \kappa$, $\Omega \subseteq M^{\bar{z}}$ and $\{\Theta(x;\bar{c}) : \bar{c} \in \Omega \}$ a uniformly directed family. Moreover, if $X = \Gamma(M, \bar{b})$ for some $\bar{b} \in \cU^{\bar{y}}$, then $\Theta(x;\bar{z})$ can be chosen depending only on $\Gamma(x, \bar{y})$; and $\Omega$ can be chosen to be strict pro-definable in $M^{\Sh}$.
\end{theorem}
\begin{proof}
	The right to left implication holds in any theory by Proposition \ref{prop: unif directed implies ext type def}.
	
	We prove the converse in the type-definable case ($\bigvee$-definable case is similar, applying honest definitions directly to the formula defining $X$ rather than its complement). Assume $X = \Gamma(M, \bar{b})$ with $\Gamma(x,\bar{y}) = \{ \varphi_t(x, y_t) : t < \kappa \}$, $\bar{y}_t$ a finite subtuple of $\bar{y}$, and $b_t$ a subtuple of $\bar{b}$ corresponding to it.
	
	Fix $t < \kappa$. By induction on $j < \omega$, we choose formulas $\theta_{t,j}(x,z_{t,j})$ in $L$ so that the following hold for every $j < \omega$:
	\begin{enumerate}
	
		\item  $\varphi_t(M, b_t) = \bigcap \left\{ \theta_{t,j}(M,c) : c \in M^{z_{t,j}}, \varphi_t(M, b_t) \subseteq \theta_{t,j}(M,c)   \right\}$;
		\item for every $n \in \omega$ and $c_1, \ldots, c_n \in M^{z_{t,j}}$ with $ \varphi_t(M, b_t)  \subseteq \theta_{t,j}(M,c)$ there is some $c \in M^{z_{t,j+1}}$ so that $\varphi_t(M, b_t)  \subseteq \theta_{t,j+1}(M,c)$ and
 $\models \theta_{t,j+1}(x,c) \rightarrow \bigwedge_{i=1}^{n} \theta_{t,j}(x,c_i)$.
	\end{enumerate}

	Let $z_{t,0} := x$ and $\theta_{t,0}(x, z_{t,0}) := (\neg x = z_{t,0})$, clearly (1) holds.
	
	Assume we are given $\theta_{t,j}(x,z_{t,j})$ satisfying (1). Let $y := y_t$ and $\varphi(x,y) := \neg \varphi_t(x,y)$, $u := z_{t,j}$ and $\zeta(x, u) := \neg \theta_{t,j}(x,u)$. As $\theta_{t,j}$ satisfies (1), we have $\varphi(M,b_t) = \bigcup \left\{ \zeta(M, e) : e \in M^{u}, \zeta(M, e) \subseteq \varphi(M,b_t) \right\}$. Let $\theta(x, z)$ be as given for $\varphi(x,y), \zeta(x, u)$ by Fact \ref{fac: better uniform honest defs}. Let $z_{t,j} := z$ and $\theta_{t,j+1}(x,z_{t,j}) := \neg \theta(x, z_{t,j})$. Then $\theta_{t,j+1}(x,z_{t,j})$ satisfies (2), as $\theta(x, z)$ satisfies Fact \ref{fac: better uniform honest defs}(3). Which, combined with $\theta_{t,j}(x,z_{t,j})$ satisfying (1), implies that $\theta_{t,j+1}(x,z_{t,j})$ also satisfies (1).
	
	Let $\bar{z}$ be the concatenation of $(z_{t,j} : t < \kappa, j < \omega)$ and $\Theta(x,\bar{z}) := \left\{\theta_{t,j}(x,z_{t,j}) : t < \kappa, j < \omega \right\}$, a set of formulas of size $\leq \kappa$. Let 
	$$\Omega := \left\{ \bar{c} = (c_{t,j})_{t < \kappa, j < \omega} \in M^{\bar{z}} :  \bigwedge_{t<\kappa, j < \omega} \ \left( \varphi_t(M,b_t) \subseteq \theta_{t,j}(M,c_{t,j}) \right) \right\},$$
	it is clearly strict pro-definable in $M^{\Sh}$. We then have $\Gamma(M, \bar{b}) = X = \bigcap_{\bar{c} \in \Omega} \Theta(M, \bar{c})$, by definition of $\Omega$ and $\Theta$, using that $\theta_{t,j}$ satisfy (1) for all $t,j$.

	And the family $\{\Theta(M,\bar{c}) : \bar{c} \in \Omega\}$ is uniformly directed. Indeed, let $\bar{c}^{s} = (c^s_{t,j} : t< \kappa, j < \omega) \in \Omega$ for $1 \leq s \leq n $ be given. For any $t < \kappa, j < \omega$, by definition of $\Omega$ we have  $ \varphi_t(M, b_t)  \subseteq \theta_{t,j}(M,c^s_{t,j})$ for every $1 \leq s \leq n $; hence, by (2), there is some $c_{t,j}$ so that $\varphi_t(M, b_t)  \subseteq \theta_{t,j+1}(M,c_{t,j})$ and
 $\models \theta_{t,j+1}(x,c_{t,j}) \rightarrow \bigwedge_{s=1}^{n} \theta_{t,j}(x,c^s_{t,j})$. Let $\bar{c} := (c_{t,j})_{t < \kappa, j < \omega}$, then $\bar{c} \in \Omega$. Hence we can take the function $f:  (t,j) \in \kappa \times \omega \mapsto (t,j+1) \in \kappa \times \omega$ to witness uniform directedness of the family.
\end{proof}

\begin{remark}\label{rem: directed approximation in the monster}
	It follows  that, for $\widetilde{M}' \succ^{L'} M^{\Sh}$ a saturated extension, $X(\widetilde{M}') = \bigcap_{\bar{b} \in \Omega(\widetilde{M}')} \Theta(\widetilde{M}';\bar{b})$ (by (1), definition of $\Theta$ and $\Omega$ in the Proof of Theorem \ref{thm: ext type def presentation} and elementarity), and the pro-definable family of type-definable sets $\left\{\Theta(\widetilde{M}';\bar{b}) : \bar{b} \in \Omega(\widetilde{M}')\right\}$ is $\lambda$-directed for any small $\lambda$ (by Remark \ref{rem: unif direct and strict pro-def} ((3) and (1)). And similarly in the $\bigvee$-definable case.
\end{remark}

\begin{question}
	If $M$ is saturated, can we also assume that $\{\Theta(M,\bar{b}) : \bar{b} \in \Omega\}$ is $\lambda$-directed for any small $\lambda$? (Note that $M^{\Sh}$ is not saturated, so Remark \ref{rem: unif direct and strict pro-def}(2) does not apply).
\end{question}

\begin{question}\label{q: honest defs directed}
	It is open since \cite{chernikov2013externally} if an analogous presentation holds for externally definable sets, i.e.~if Theorem \ref{thm: ext type def presentation} holds for $\kappa=1$ (or finite).
	Also, can one have directedness for small sets of members of the family rather than finite sets?
	\end{question}	

As an application, we have a (soft) description of externally definable subgroups of definable groups, as unions of (uniformly) directed families of their $\bigvee$-definable  subgroups:
\begin{proposition}\label{prop: ext def subgroups approx V-def}
	Let $M \models T$ be NIP, $G$ a definable group in $M$, $H \leq G(M)$ an externally $\kappa$-$\bigvee$-definable subgroup, and $\widetilde{M}' \succ^{L'} M^{\Sh}$ is saturated. Then there is a disjunction $\Sigma(x;\bar{y})$ of $\leq \kappa$ formulas over $\emptyset$ and a strict pro-definable over $M^{\Sh}$ set $\Omega \subseteq (\widetilde{M}')^{\bar{y}}$ so that $\Sigma(\widetilde{M}' ;\bar{b}) \leq H(\widetilde{M}')$ is a $\bigvee$-definable subgroup for each $\bar{b} \in \Omega$, $H(\widetilde{M}') = \bigcup_{\bar{b} \in \Omega} \Sigma(\widetilde{M}' , \bar{b})$ and the family $\left\{\Sigma(\widetilde{M}', \bar{b}) :  \bar{b} \in \Omega\right\}$ is uniformly directed, hence $\lambda$-directed for any small $\lambda$ (with respect to saturation of $\widetilde{M}'$).
\end{proposition}
\begin{proof}
	Let $\Theta(x,\bar{y}) $ be a disjunction of formulas $\{\varphi_t(x,\bar{y}) : t < \kappa\}$ of size $\leq \kappa$ and $\Omega$ be given by Theorem \ref{thm: ext type def presentation} with respect to $H(M)$, and let $\Sigma(x,\bar{y})$ be so that $\Sigma(x,\bar{b})$ defines the subgroup generated by $\Theta(x,\bar{b})$, i.e.~$\Sigma(x,\bar{y})$ is 
	$$\bigvee_{n \in \omega} \bigvee_{\bar{t} = (t_1, \ldots,t_n) \in  \kappa^n} \bigvee_{f: [n] \to \{1,-1\}} \exists (x_i : i \in [n]) \left( \bigvee_{i \in [n]} \varphi_{t_i}(x_i, \bar{y}) \land x = \prod_{i \in [n]} x_i^{f(i)} \right).$$
	Then $\{\Sigma(M, \bar{b}) : \bar{b} \in \Omega\}$ is still uniformly directed, and satisfies the requirement (using Remark \ref{rem: directed approximation in the monster}).
\end{proof}

\begin{remark}
	If $E(x_1,x_2)$ is an externally $\bigvee$-definable equivalence relation, we can similarly approximate it from below by a uniformly directed family of  $\bigvee$-definable equivalence relations. Namely, if $\Theta(x_1,x_2;\bar{y}) $ is a disjunction of formulas $\{\varphi_t(x_1,x_2,\bar{y}) : t < \kappa\}$ of size $\leq \kappa$ and $\Omega$ be given by Theorem \ref{thm: ext type def presentation} with respect to $E(M)$, and let $\Sigma(x,\bar{y})$ be the disjunction defining  the transitive closure of $\Theta'(x_1,x_2; \bar{y}) := \Theta(x_1,x_2;\bar{y}) \lor \Theta(x_2,x_1; \bar{y}) \lor x_1 = x_2$.
\end{remark}

We conclude with an observation about local coheirs.

\begin{definition}
\begin{enumerate}
	\item Let $\Delta(x,y)$ be a finite set of partitioned formulas $\varphi(x,y)$. By a complete $\Delta(x,y)$-type $q(x)$ over a set of parameters $A$  we mean a maximal consistent set of formulas of the form $\varphi(x,a)$ or $\neg \varphi(x,a)$ with $\varphi(x,y) \in \Delta$ and $a$ a tuple of corresponding length from $A$. We let $S_{\Delta(x,y)}(A)$ denote the set of complete $\Delta$-types over $A$.
	\item  If $\Delta$ is a finite set of (non-partitioned) formulas, by a complete $\Delta$-type over $A$ in the free variables $x$ we mean a complete $\Delta^\ast(x,y)$-type over $A$, where $\Delta^\ast(x,y)$ is a finite set of  all partitioned formulas obtained by taking an arbitrary partition of the variables of appropriate lengths  of a formula from $\Delta$. We let $S_{\Delta,x}(A)$ denote the set of complete $\Delta$-types over $A$ in the free variables $x$.
\end{enumerate}
	\end{definition}

\begin{fact}\cite[Claim 4.11]{kaplan2014examples}\label{fac: local coheirs}
	For every $A \subseteq C \subseteq \cU$ and type $q \in S_{\Delta,x}(\cU)$ which is finitely satisfiable in $A$, and any choice of a coheir $q' \in S_x(\cU)$ over $A$ with $q \subseteq q'$ (there is always one) and $n \in \omega$, we have:
	$(a_0, \ldots, a_{n-1}) \models \left( (q')^{(n)}|_{C}  \right) \restriction_{\Delta}$ if and only if $a_i \models q|_{Ca_{<i}}$ for all $i < n$.
	
	In particular, if $a_i \models q|_{Ca_{<i}}$ for all $i < n$, then $(a_i : i <n)$ is a $\Delta(C)$-indiscernible sequence of length $n$.
\end{fact}

The following is a local refinement of \cite[Proposition 3.9]{chernikov2014external} in the case of coheirs:
\begin{lemma}\label{lem: local honest def}
Assume we are in the context of Definition \ref{def: Shelah exp context}. 	Working in $T'$, let $p(x) \in S^L(M')$ be finitely satisfiable in $M$. For every NIP formula $\varphi(x,y) \in L$ there is some finite set of formulas $\Delta \subseteq L$ so that: for every $b \in N^y$, if $p\restriction_{\Delta}(x) \cup \{ \varphi(x,b) \} \cup \{ P(x) \}$ is consistent, then $p \restriction_{\Delta}(x) \cup \{ P(x) \} \vdash \varphi(x,b)$.
\end{lemma}
\begin{proof}
	As $\varphi(x,y)$ is NIP in $T$, by compactness there exist some finite set of $L$-formulas $\Delta$ and some $n$ so that: if $(a_i)_{i <n}$ is $\Delta$-indiscernible and $b$ is arbitrary (in $\cU \models T$), then we cannot have $\models \varphi(a_i,b) \iff i$ is even.
	
	Now if for some $b$ in $N$ we had $p\restriction_{\Delta} \cup \{ \varphi(x,b) \} \cup \{ P(x) \}$ is consistent but $p \restriction_{\Delta}(x) \cup \{ P(x) \} \not \vdash \varphi(x,b)$, following the proof of the existence of honest definitions \cite[Proposition 1.1]{chernikov2013externally} we can find 
	$(a_i: i < n)$ in $M' = P(N')$  so that $a_i \models \left( p\restriction_{\Delta} \right)_{Ma_{<i}}$ and $\models \varphi(a_i,b) \iff i$ is even.  By Fact \ref{fac: local coheirs} the sequence $(a_i : i <n)$ is $\Delta$-indiscernible, contradicting the choice of $\Delta$ and $n$.
\end{proof}
\noindent It is not known if analogous statements hold for invariant $\Delta$-types.

\section{Acknowledgements}
I am grateful to Pantelis Eleftheriou and Pietro Freni for multiple discussions during my visit to the University of Leeds, and to Udi Hrushovski during my visit to Oxford in September 2024. I am grateful to Matthias Aschenbrenner, Martin Bays, Kyle Gannon, Yatir Halevi, Martin Hils, Chris Laskowski, David Marker, Anand Pillay, Sergei Starchenko, Atticus Stonestrom and Marcus Tressl for helpful conversations. Partially supported by the NSF Research Grant DMS-2246598.

\bibliographystyle{plain}
\bibliography{ref}
\end{document}